\documentclass{amsart}

\usepackage{amsfonts}
\usepackage{mathrsfs}
\usepackage{amsmath}
\usepackage{amsthm}
\usepackage{amssymb,bbm,bm}
\usepackage{latexsym,hyperref}
\usepackage{enumitem,nccmath}
\usepackage[all]{xy}

\usepackage{color}
\usepackage[usenames,dvipsnames]{xcolor}
\usepackage{etoolbox}

\allowdisplaybreaks

\numberwithin{equation}{section}

\newtheorem{thrm}{Theorem}[section]
\newtheorem{prop}[thrm]{Proposition}
\newtheorem{crl}[thrm]{Corollary}
\newtheorem{lemma}[thrm]{Lemma}

\newtheorem{defn}[thrm]{Definition}

\theoremstyle{remark}
\newtheorem{rmk}[thrm]{Remark}

\newcommand{\nc}{\newcommand}

\nc{\End}{\mathrm{End}}
\nc{\Ext}{\mathrm{Ext}}
\nc{\Hom}{\mathrm{Hom}}
\nc{\Ima}{\mathrm{Image}}
\nc{\Ind}{\mathrm{Ind}}
\nc{\Ker}{\mathrm{Ker}}
\nc{\RHom}{\mathrm{RHom}}
\nc{\Sym}{\mathrm{Sym}}

\nc\bb{\mathbb}
\nc\mf{\mathfrak}
\nc\ms{\mathsf}
\nc\mc{\mathcal}

\nc{\mfg}{\mf{g}}
\nc{\mfh}{\mf{h}}
\nc{\mfsl}{\mf{sl}}
\nc{\mfgl}{\mf{gl}}
\nc{\mfso}{\mf{so}}
\nc{\mfsp}{\mf{sp}}

\nc{\nn}{\nonumber}

\nc{\spl}[1]{\begin{equation}\begin{aligned}#1\end{aligned}\end{equation}}
\nc{\eqa}[1]{\begin{align}#1\end{align}}
\nc{\eqn}[1]{\begin{align*}#1\end{align*}}
\nc{\eg}[1]{\begin{gather}#1\end{gather}}
\nc{\egn}[1]{\begin{gather*}#1\end{gather*}}

\nc{\mbfk}{\mathbf{k}}
\nc{\mcA}{\mc{A}}
\nc{\mcB}{\mc{B}}
\nc{\mcU}{\mc{U}}
\nc{\mfU}{\mf{U}}
\nc{\mcP}{\mc{P}}
\nc{\mcQ}{\mc{Q}}
\nc{\mcX}{\mc{X}}
\nc{\mcZ}{\mc{Z}}
\nc{\mcT}{\mc{T}}
\nc{\mcG}{\mc{G}}
\nc{\msS}{\ms{S}}
\nc{\mss}{\ms{s}}
\nc{\msc}{\ms{c}}
\nc{\msd}{\ms{d}}
\nc{\msv}{\ms{v}}
\nc{\msq}{\ms{q}}
\nc{\msw}{\ms{w}}
\nc{\mcS}{\mc{S}}
\nc{\mcI}{\mc{I}}

\nc{\C}{\mathbb{C}}
\nc{\N}{\mathbb{N}}
\nc{\Z}{\mathbb{Z}}

\nc{\ot}{\otimes}
\nc{\op}{\oplus}

\nc{\lan}{\langle}
\nc{\ran}{\rangle}

\nc{\qu}{\quad}
\nc{\qq}{\qquad}

\nc\Tr{{\rm tr}}

\nc{\al}{\alpha}
\nc{\del}{\delta}
\nc{\eps}{\epsilon}
\nc{\veps}{\varepsilon}
\nc{\ga}{\gamma}
\nc{\Ga}{\Gamma}
\nc{\ka}{\kappa}
\nc{\la}{\lambda}
\nc{\om}{\omega}
\nc{\si}{\sigma}
\nc{\Si}{\Sigma}
\nc{\bsi}{\boldsymbol\sigma}
\nc{\bSi}{\boldsymbol\Sigma}
\nc{\Ups}{\upsilon}
\nc{\vphi}{\varphi}

\nc{\btau}{\boldsymbol\tau}
\nc{\bdel}{\boldsymbol\delta}

\nc{\id}{\mathrm{id}}
\nc{\gr}{\mathrm{gr}}

\nc{\lrh}{\leftrightharpoons}
\nc{\iso}{\stackrel{\sim}{\longrightarrow}}
\nc{\liso}{\stackrel{\sim}{\longleftarrow}}
\nc{\wh}{\widehat}
\nc{\wt}{\widetilde}
\nc{\tl}{\tilde}
\nc{\lra}{\longrightarrow}
\nc{\ra}{\rightarrow}
\nc{\into}{\hookrightarrow}
\nc{\onto}{\twoheadrightarrow}

\nc{\qdet}{{\rm qdet\,}}
\nc{\sdet}{{\rm sdet\,}}
\nc{\sign}{{\rm sign}}
\nc{\inv}{{\rm inv}}

\nc{\F}[2]{F'^{\rho}_{#1#2}}

\nc{\mysum}{\textstyle\sum}
\nc{\mysuml}{\textstyle\sum\limits}

\usepackage{color}
\usepackage[usenames,dvipsnames]{xcolor}
\nc{\red}{\color{red}}
\nc{\blu}{\color{blue}}
\nc{\br}{\color{Brown}}
\nc{\gre}{\color{green!50!black}}

\renewcommand{\,}{\kern 0.1em} 

\begin{document}

\hfill DMUS-MP-16/10

\vspace{1.2cm}

\begin{center}
{\Large{\textbf{Representations of twisted Yangians of types B, C, D: I}}} 

\bigskip

Nicolas Guay$^{1a}$, Vidas Regelskis$^{23b}$, Curtis Wendlandt$^{1c}$

\end{center}

\bigskip

\begin{center} \small 
$^1$ University of Alberta, Department of Mathematics, CAB 632, Edmonton, AB T6G 2G1, Canada.\\ 
$^2$ University of Surrey, Department of Mathematics, Guildford, GU2 7XH,  UK. \\
$^3$ University of York, Department of Mathematics, York, YO10 5DD,  UK. \\
\smallskip
E-mail: $^a$\,nguay@ualberta.ca $^b$\,vidas.regelskis@york.ac.uk $^c$\,cwendlan@ualberta.ca
\end{center}

\patchcmd{\abstract}{\normalsize}{}{}{}

\begin{abstract} \small 
We initiate a theory of highest weight representations for twisted Yangians of types B,~C,~D and we classify the finite-dimensional irreducible representations of twisted Yangians associated to symmetric pairs of types CI, DIII and BCD0.
\end{abstract}

\makeatletter
\@setabstract
\makeatother

\medskip

\tableofcontents

\thispagestyle{empty}


\section{Introduction}

Yangians constitute one of the two important families of quantum groups of affine type along with the quantum affine algebras and are of interest to both mathematicians and physicists. Their representation theory has been studied quite a lot over the past thirty years and occasionally applications have been found to the study of other mathematical entities, for instance to isomonodromic deformations \cite{ChMa}, slices in affine Grassmannian \cite{KWWY}, classical centralizers \cite{Mo4,MO,Na} and the geometry of Schubert varieties \cite{RTV}. The category of finite-dimensional representations of Yangians is not semisimple, so understanding the irreducible ones does not provide a complete picture of the category, but it is nevertheless the most important first step that needs to be taken. Such an understanding comes first from classifying those modules in terms of certain polynomials (see \cite{Dr3}) and then from building realizations of those modules \cite{KN1,KN2,KN5,KNP}, studying their behaviour under tensor products \cite{Mo3,NaTa1,NaTa2}, etc. 

In theoretical physics, Yangians first appeared via their relation to rational solutions of the quantum Yang-Baxter equation \cite{Dr1}. It was later discovered that they are also relevant in the study of symmetries of certain integrable systems \cite{Be1,Be2}. In the case of integrable systems with boundaries, it turns out that certain subalgebras of Yangians called twisted Yangians are sometimes relevant to understanding their symmetries, see e.g.~\cite{DMS,Ma,MaRe1,MaRe2,MaSh}. In the mathematical literature, those that have been mostly studied are the twisted Yangians of type AI and AII (corresponding to the pairs $(\mfgl_N,\mfso_N)$ and $(\mfgl_N,\mfsp_N)$) introduced by G. Olshanskii in \cite{Ol}, further studied in \cite{MNO} and whose representation theory was the subject of a good number of papers by A. Molev, M. Nazarov \textit{et al.}, see for instance \cite{KN3,KN4,KN5,KNP,Mo1,Mo2,Mo4,Mo5,Na}. They are coideal subalgebras of the Yangian of $\mfgl_N$. In \cite{GR}, for each symmetric pair of type B, C or D (see Subsection \ref{sec:symm-pairs}), similar (extended) twisted Yangians denoted $Y(\mfg_N,\mcG)^{tw}$ (resp.~$X(\mfg_N,\mcG)^{tw}$ - see \cite{AMR}) were introduced as coideal subalgebras of the (extended) Yangians $Y(\mfg_N)$ (resp.~$X(\mfg_N)$) where $\mfg_N=\mfso_N$ or $\mfg_N=\mfsp_N$. (See Subsection \ref{sec:symm-pairs} for the definition of the  matrix $\mcG$.) Some of their structural properties were determined: in particular, it was shown that $Y(\mfg_N,\mcG)^{tw}$ can also be viewed as a quotient of $X(\mfg_N,\mcG)^{tw}$ which in turn can be defined using the reflection equation that first appeared in \cite{Ch} and \cite{Sk}. 

The goal of this paper is to begin the study of the irreducible finite-dimensional representations of the new twisted Yangians defined in \cite{GR}. Our main results provide the classification of these modules in the cases when the underlying symmetric pair is of type CI (i.e.~$(\mfsp_{2n},\mfgl_{n})$), DIII (i.e.~$(\mfso_{2n},\mfgl_{n})$) or BCD0 (i.e.~$(\mfg_N,\mfg_N)$): see Theorems \ref{CT:Thm.DIII-CI}, \ref{CT:Thm.C0-D0} and \ref{CT:Thm.B0}, which are applicable to the extended twisted Yangians $X(\mfg_N,\mcG)^{tw}$ and their respective Corollaries \ref{CT:Cor.DIII-CI.Yangians} and \ref{CT:Cor.C0-D0.Yangians} which are applicable to $Y(\mfg_N,\mcG)^{tw}$. The classification is stated in terms of certain polynomials as in the case of Yangians of simple Lie algebras \cite{Dr3}, twisted Yangians for symmetric pairs of type A \cite{Mo5} and twisted $q$-Yangians \cite{GM}. To obtain our classification theorems, we follow a well established approach (see \cite{Mo1,Mo2,Mo5,MR,AMR}). As a first step, we prove that any finite-dimensional irreducible module is of highest weight type, hence the quotient of a Verma module: see Theorem \ref{HWT:Thm.HWT}. The classification problem thus reduces to determining conditions on the highest weight which hold exactly when the corresponding simple module is finite-dimensional, but there is another question which must first be considered. In the context of the RTT-presentation which we are using, a Verma module can sometimes be trivial (see Proposition 5.4 in \cite{AMR}), so it is important to find conditions which are equivalent to non-triviality: such conditions are stated in Proposition \ref{HWT:refl.Prop.2} whose proof depends on Theorem 4.2 in \cite{MR} and relies on Proposition \ref{HWT:refl.Prop.1} which states that a certain subspace of any representation of $X(\mfg_N,\mcG)^{tw}$ can be made into a module over an extended reflection algebra (which is an extension of a twisted Yangian of type AIII). The same proposition, along with Theorem 4.6 in \cite{MR}, allows us to deduce certain conditions (see Proposition \ref{HWT:refl.Prop.3}) on the highest weight which are necessary for the irreducible quotient of a Verma module to be finite-dimensional. Similar results were already known for the twisted Yangians $Y^\pm(N)$ associated to the symmetric pairs $(\mfgl_N,\mfg_N)$, where $\mfg_N=\mfso_{N}$ or $\mfsp_{N}$, with the role of the reflection algebra played instead by the Yangian $Y(n)$ of the general linear algebra $\mfgl_n$ ($N=2n$ or $N=2n+1$): Propositions \ref{HWT:refl.Prop.1} and \ref{HWT:refl.Prop.3} should be compared to Proposition 4.2.8 of \cite{Mo5}. 

One of the key ingredients in the proof of the Classification Theorems \ref{CT:Thm.DIII-CI}, \ref{CT:Thm.C0-D0} and \ref{CT:Thm.B0} is provided by Lemma \ref{CT:Lemma.induction} and Proposition \ref{CT:Prop.induction}, which state that a certain subspace $V_+$ of a $X(\mfg_N,\mcG)^{tw}$-module $V$ inherits the structure of a module over a twisted Yangian associated to a symmetric pair $(\mfg',\mathfrak{k}')$, where $\mathrm{rank} \,\mfg' = \mathrm{rank} \, \mfg_N -1$. This allows us to argue by induction on the rank of the Lie algebra $\mfg_N$ in order to establish necessary conditions for the finite-dimensionality of $V$. As for the base case for the induction, it can be deduced from the known Classification Theorem for finite-dimensional irreducible modules over twisted Yangians of type AI or AII (see \cite{Mo1,Mo2,Mo5}) using the isomorphisms established in \cite{GRW}. To prove sufficiency of the conditions in the aforementioned theorems, we build a representation of $X(\mfg_N,\mcG)^{tw}$ with the proper highest weight from a finite-dimensional representation of $X(\mfg_N)$.

In the present manuscript, our results on highest weight modules of twisted Yangians in Section \ref{sec:HWT} are valid in all Cartan types listed in Section \ref{sec:defs} except type DI(b), however in Section \ref{sec:CT} we fully classify finite-dimensional irreducible modules of twisted Yangians of types BCD0, CI and DIII only. We hope to be able to extend our classification results to twisted Yangians of other types in a future publication. Type DI(b) is excluded from our considerations because the generators that should be used to define the notion of highest weight vectors are not exactly the same as in the other types. This is a consequence of the fact that the matrix $\mcG$ below (see Subsection \ref{sec:symm-pairs}) is not diagonal in type DI(b). Furthermore, the proof of similar classification theorems for twisted Yangians of type CII and DI(a) seems to be noticeably more complicated, hence has also been postponed.

Once the full classification of irreducible finite dimensional modules has been achieved, we hope to investigate connections of the twisted Yangians introduced in \cite{GR} to centralizer constructions as in \cite{Mo4,MO,Na} and possible analogs of the functor studied in \cite{KN3,KN4}. Furthermore, the methods of proof of this paper should be applicable to obtain classification theorems as in \cite{GM} for finite dimensional irreducible modules over twisted quantum loop algebras associated to symmetric pairs of types B, C, D similar to those studied in \cite{MRS}.

{\it Outline.} The paper is organized as follows. In Section~\ref{sec:defs} we introduce the necessary definitions and recall the basic facts about symmetric pairs of classical types, the reflection equation and its solutions. In Section~\ref{sec:Y-TY} we recall the definition of Yangians, twisted Yangians and reflection algebras. The main results of this paper are presented in the remaining three sections.  We initiate the highest weight theory for representations of twisted Yangians of types B, C, D, for the cases when the matrix $\mcG$ is diagonal in Section \ref{sec:HWT}. In the following one, we deduce the classification theorems for finite-dimensional irreducible representations of twisted Yangians of types CI, DIII and BCD0 when the rank is small. The last section contains the main classifications theorems of finite-dimensional irreducible modules for twisted Yangians of types CI, DIII and BCD0.

\smallskip

{\it Acknowledgements.} The first author acknowledges the support of the Natural Sciences and Engineering Research Council of Canada through its Discovery Grant program.  Part of this work was done during the second author's visits to the University of Alberta. V.R. thanks the University of Alberta for the hospitality, and also gratefully acknowledges the Engineering and Physical Sciences Research Council (EPSRC) of the United Kingdom for the Postdoctoral Fellowship under the grant EP/K031805/1. The third author was partially supported by a CGS-D scholarship from the Natural Sciences and Engineering Research Council of Canada. The authors would like to warmly thank the referee for a very thorough review of our paper.



\section{Definitions and preliminaries} \label{sec:defs}


\subsection{Notation} \label{sec:notation}

Let $n \in\N$ and set $N=2n$ or $N=2n+1$. We will always assume that $\mfg=\mfgl_N$, $\mfg=\mfsl_N$ or $\mfg=\mfg_N$, where $\mfg_N$ is the orthogonal Lie algebra $\mfso_N$ or the symplectic Lie algebra $\mfsp_N$ (only when $N=2n$). The algebra $\mfg_N$ can be realized as a Lie subalgebra of $\mfgl_N$ as follows. We label the rows and columns of matrices in $\mfgl_N$ by the indices $\{ -n, \ldots, -1,1,\ldots, n\}$ if $N=2n$ and by $\{-n, \ldots, -1,0,1,\ldots, n \}$ if $N=2n+1$. Set $\theta_{ij}=1$ in the orthogonal case and $\theta_{ij}=\mathrm{sign}(i)\cdot \mathrm{sign}(j)$ in the symplectic case for $-n \le i,j\le n$. (It is understood that, when $N$ is even, $i=0$ and $j=0$ are excluded.) 

Let $F_{ij}=E_{ij} - \theta_{ij} E_{-j,-i}$, where the $E_{ij}$ are the usual elementary matrices of $\mfgl_N$. These matrices satisfy the relations 
\eq{ \label{[F,F]}
[F_{ij},F_{kl}] = \delta_{jk}F_{il} - \delta_{il}F_{kj} + \theta_{ij}\delta_{j,-l}F_{k,-i} - \theta_{ij}\delta_{i,-k}F_{-j,l} , \qq
F_{ij} + \theta_{ij}F_{-j,-i}=0.
}
We may identify $\mfg_N$ with $\mathrm{span}_{\C} \{ F_{ij} \, : \, -n\le i,j\le n \}$ and we will use $\mfh_N=\mathrm{span}_{\C} \{ F_{ii} \, : \, 1 \le i \le n \}$ as Cartan subalgebra. 
Given a Lie algebra $\mathfrak{a}$ its universal enveloping algebra will be denoted by $\mfU\mathfrak{a}$.

Next, we need to introduce some operators: $P \in \End(\C^N \ot\C^N)$ will denote the permutation operator on $\C^N \ot \C^N$ and we set $Q=P^{t_1}=P^{t_2}$ where the transpose $t$ is given by $(E_{ij})^t = \theta_{ij} E_{-j,-i}$; explicitly,
\eq{ \label{PQ}
P=\sum_{-n \le i,j \le n} E_{ij} \ot E_{ji}, \qq Q = \sum_{-n \le i,j \le n} \theta_{ij} E_{ij} \ot E_{-i,-j} .
}
In the cases when both orthogonal ($\theta_{ij}=1$) and symplectic ($\theta_{ij} = \mathrm{sign}(i)\cdot \mathrm{sign}(j)$) transpositions are used simultaneously, we will denote the former by $t_+$ and the latter by $t_-$. Let $I$ denote the identity matrix. Then $P^2=I$, $PQ=QP=\pm Q$ and $Q^2=N Q$, which will be useful below. Here (and further in this paper) the upper sign corresponds to the orthogonal case and the lower sign to the symplectic case. In addition, we will also use the notation $(\pm)$ and $[\pm]$, which will be explained in appropriate places.

Let tensor products be defined over the field of complex numbers. For a matrix $X$ with entries $x_{ij}$ in an associative algebra $A$ we write
\eq{
X_s = \sum_{-n\le i,j\le n} \underbrace{ I \ot \cdots \ot I}_{s-1} \ot E_{ij} \ot I \ot \cdots \ot I \ot x_{ij} \in \End(\C^N)^{\ot k} \ot A .
}
Here $k \in \N_{\ge 2}$ and $1\le s\le k$; it will always be clear from the context what $k$ is.

Henceforth, we will employ the convention used above in labelling the rows and columns of $N \times N$ matrices except in the following circumstances. 
When discussing symmetric pairs of type AIII (see Subsection \ref{sec:symm-pairs}) and the corresponding twisted Yangians $\mcB(N,q)$ (see Subsection \ref{sec:TY-MR}), 
we shall follow the notation from \cite{MR} and instead label the rows and columns of $N\times N$ matrices using the ordered set $\{ 1,\ldots, N\}$.


\subsection{Symmetric pairs of classical Lie algebras}  \label{sec:symm-pairs}

The symmetric pairs we are interested in are of the form $(\mfg,\mfg^{\rho})$, where $\rho$ is an involution of $\mfg$ and $\mfg^\rho$ denotes the $\rho$-fixed subalgebra of $\mfg$. The involution $\rho$ is given by $\rho(X) = \mcG X \mcG^{-1}$ for a particular matrix $\mcG$, except that $\rho(X) = - X^t$ for types AI and AII. We will make well-suited choices of $\mcG$ for the purposes of the present paper, which agree with the ones in \cite{MNO}, \cite{MR}, and to the ones in \cite{GR} up to conjugation.  We also introduce a further refinement of Cartan's classification of symmetric spaces that reflects the explicit form of $\mcG$ listed below and differences in the study of the representation theory of the extended twisted Yangians. 

Let $p$ and $q$ be such that $p\ge q>0$ and $p+q=N$. In the list below, for each Cartan type, we indicate the corresponding symmetric pair and give our choice of matrix $\mcG$:

\begin{itemize} [itemsep=0.75ex]

\item AI\hspace{.495cm}: $(\mfg,\mfg^\rho)=(\mfsl_N,\mfso_N)$ and $t=t_+$.

\item AII\hspace{.37cm}: $(\mfg,\mfg^\rho)=(\mfsl_N,\mfsp_N)$ and $t=t_-$.

\item AIII\hspace{.23cm}: $(\mfg,\mfg^\rho)=(\mfsl_N,(\mfgl_p \op \mfgl_{q})\cap \mfsl_N)$ and $\mcG=\sum_{i=1}^{p} E_{ii} - \sum_{i=p+1}^N E_{ii}$. 

\item CI\hspace{.507cm}: $N=2n$, $(\mfg,\mfg^\rho)=(\mfsp_N,\mfgl_{\frac{N}{2}})$ and $\mcG=\sum_{i=1}^{\frac{N}{2}} (E_{ii} - E_{-i,-i})$. 

\item DIII\hspace{.22cm}: $N=2n$, $(\mfg,\mfg^\rho)=(\mfso_N,\mfgl_{\frac{N}{2}})$ and $\mcG=\sum_{i=1}^{\frac{N}{2}} (E_{ii} - E_{-i,-i})$.

\item BI(a)\,: $N=2n+1$, $(\mfg,\mfg^\rho)=(\mfso_N,\mfso_p\op\mfso_q)$ such that $p$ is odd and $q$ is even. The matrix $\mcG$ is
\eq{
\mcG=\sum_{i=-\frac{p-1}{2}}^{\frac{p-1}{2}} E_{ii}
 - \sum_{i=\frac{p+1}{2}}^{\frac{N-1}{2}} (E_{ii} + E_{-i,-i}). \label{G:BIb}
}
In particular, the subalgebra of $\mfg^{\rho}$ spanned by $F_{ij}$ with $-\frac{p-1}{2} \le i,j \le \frac{p-1}{2}$ is isomorphic to $\mfso_p$ and the subalgebra spanned by $F_{ij}$ with $|i|,|j| \ge \frac{p+1}{2}$ is isomorphic to $\mfso_q$. 

\item BI(b)\,: $N=2n+1$, $(\mfg,\mfg^\rho)=(\mfso_N,\mfso_p\op\mfso_q)$ such that $p$ is even and $q$ is odd. The matrix $\mcG$ is
\eq{
\mcG=-\sum_{i=-\frac{q-1}{2}}^{\frac{q-1}{2}} E_{ii}
 + \sum_{i=\frac{q+1}{2}}^{\frac{N-1}{2}} (E_{ii} + E_{-i,-i}). \label{G:BIa}
}
Now it is $\mfso_q$ spanned by $F_{ij}$ with $-\frac{q-1}{2} \le i,j \le \frac{q-1}{2}$ and $\mfso_p$ is spanned by $F_{ij}$ with $|i|,|j| \ge \frac{q+1}{2}$. 

\item CII\hspace{.38cm}: $N=2n$, $(\mfg,\mfg^\rho)=(\mfsp_N,\mfsp_p\op\mfsp_q)$ such that both $p$ and $q$ are even. The matrix $\mcG$ is
\eq{
\mcG= \sum_{i=1}^{\frac{p}{2}} (E_{ii} + E_{-i,-i}) - \sum_{i=\frac{p}{2}+1}^{\frac{N}{2}} (E_{ii} + E_{-i,-i}) . \label{G:CII}
}
In this case the subalgebra of $\mfg^{\rho}$ spanned by $F_{ij}$ with $-\frac{p}{2} \le i,j \le \frac{p}{2}$ is isomorphic to $\mfsp_p$ and the subalgebra of $\mfg^{\rho}$ spanned by $F_{ij}$ with $|i|,|j| > \frac{p}{2}$ is isomorphic to $\mfsp_q$. 

\item DI(a)\,: $N=2n$, $(\mfg,\mfg^\rho)=(\mfso_N,\mfso_p\op\mfso_q)$ such that both $p$ and $q$ are even. We choose $\mcG$ to be the same as for CII case, i.e.\@ given by \eqref{G:CII}. Hence the subalgebras $\mfso_p$ and $\mfso_q$ of $\mfg^\rho$ are defined analogously.

\item DI(b)\,: $N=2n$, $(\mfg,\mfg^\rho)=(\mfso_N,\mfso_p\op\mfso_q)$ such that both $p$ and $q$ are odd. We choose $\mcG$ to be
\eq{
\mcG=E_{1,-1} + E_{-1,1} + \sum_{i=2}^{\frac{p+1}{2}} (E_{ii} + E_{-i,-i}) - \sum_{i=\frac{p+3}{2}}^{\frac{N}{2} } (E_{ii} + E_{-i,-i}). \label{G:DIb}
} 

\end{itemize}

Type DI(b) is exceptional as it is the only case when $\mcG$ can not be chosen to be diagonal. 

\smallskip

In addition, we will consider trivial symmetric pairs:
\begin{itemize}
\item ABCD0: $(\mfg,\mfg^{\rho}) = (\mfg,\mfg)$ and $\mcG=I$. 
\end{itemize}
These can be thought of as limiting cases of types AIII, CII and BDI when $p=N$ and $q=0$. For ease of notation, we will further refer to types CI, DIII, CII, BDI and BCD0 as types B-C-D.


\subsection{R-matrices, K-matrices and the reflection equation} \label{sec:R-K-mats}

The matrices $R(u)\in\End(\C^N\ot \C^N)[[u^{-1}]]$ that we will need are defined by \cite{MNO,AACFR}:
\eq{
a) \qu R(u) = I - \frac{P}{u} \qu\text{for}\qu \mfgl_N, \qq b)\qu R(u) = I - \frac{P}{u} + \frac{Q}{u-\ka} \qu\text{for}\qu \mfg_N,  \label{R(u)}
}
where $\ka=N/2\mp 1$ for $\mfg=\mfg_N$. These $R$-matrices are rational solutions of the quantum Yang-Baxter equation with spectral parameter,
\eq{ \label{YBE}
R_{12}(u)\, R_{13}(u+v)\, R_{23}(v) = R_{23}(v)\, R_{13}(u+v)\, R_{12}(u) ,
}
and satisfy the identity $R(u)\,R(-u) = (1-u^{-2})\,I$. Set $R^t(u) = R^{t_1}(u)= R^{t_2}(u)$; here, $t_1$ is the transpose with respect to the first copy of $\End(\C^N)$ inside $\End(\C^N \ot \C^N) \cong \End(\C^N) \ot \End(\C^N)$; $t_2$ is defined similarly.

\begin{defn} \label{D:K-matrices}
We recall the following matrices $K(u)\in \End(\C^N)(u)$ from \cite{GR,MR,MNO}:
\begin{itemize} [itemsep=0.5ex]
\item $K(u)=I$ for symmetric pairs ABCD0, AI and AII;
\item $K(u)=\mcG$ for symmetric pairs AIII, CI, DIII and DI, CII when $p=q$;
\item $K(u) = (I-c\,u\, \mcG)(1-c\,u)^{-1}$ with $c=\tfrac{4}{p-q}$ for symmetric pairs BDI, CII when $p>q$. 
\end{itemize}
When $K(u)$ is a constant matrix, we will say that it is `of the first kind'; if it depends on $u$, it will be said to be `of the second kind'. 
\end{defn}

The $K(u)$ of types AI and AII is a (scalar) solution of the twisted reflection equation 
\eq{ \label{REt}
R(u-v)\,K_1(u)\,R^t(-u-v)\,K_2(v) = K_2(v)\,R^t(-u-v)\,K_1(u)\,R(u-v) ,
}
where $t=t_{\pm}$  and $R(u)$ is given by a) in \eqref{R(u)}. In all other cases, $K(u)$ are solutions of the reflection equation 
\eq{ \label{RE}
R(u-v)\,K_1(u)\,R(u+v)\,K_2(v) = K_2(v)\,R(u+v)\,K_1(u)\,R(u-v) ,
}
with $R(u)$ given by a) in \eqref{R(u)} for type AIII, or by b) in \eqref{R(u)} for all the B-C-D types. 

\smallskip

Matrix solutions of these reflection equations are often called $K$-matrices in the literature, which is why we are using the notation $K(u)$ in this subsection; however, when working with twisted Yangians in the rest of this paper, we will write $\mcG(u)$ instead of $K(u)$ to conform with the notation in \cite{GR} and \cite{Mo5}. The matrix $K(u)$ defined above satisfies the unitary relation $K(u) K(-u) = I$.  The next lemma is automatic when $K(u)$ is constant; for the other cases, it follows using the same argument as in the proof of Lemma 4.3 in \cite{GR} - see below (4.3.1) in \textit{loc.~cit.} Here we will use the notation $(\pm)$ where the lower sign distinguishes types CI and DIII from the other cases.

\begin{lemma} \label{L:K-RE-Sym}
Except in type A, the matrix $K(u)$ satisfies the following identities:
\eqa{
K^t(u) &= (\pm)\,K(\ka-u) \pm \frac{K(u)-K(\ka-u)}{2u-\ka} + \frac{\Tr(K(u))\,K(\ka-u) - \Tr(K(u))\cdot I }{2u-2\ka}  \label{K-Sym} \\
       &= p(u) K(\kappa-u) \pm \frac{K(u)}{2u-\kappa} - \frac{\Tr(K(u)) \cdot I}{2u-2\kappa} \nonumber, 
}
where
\eq{
p(u)=(\pm)\,1\mp \dfrac{1}{2u-\ka}+ \dfrac{\Tr(K(u))}{2u-2\ka}.  \label{p(u)}
}
Moreover, $p(u)$ satisfies
\eq{
p(u)\,p(\ka-u) = 1 - \frac{1}{(2u-\ka)^{2}}. \label{p(u)p(k-u)}
}
\end{lemma}

Next we give an additional solution $K(u;a)$ to \eqref{RE} parametrized by a free parameter $a\in\C$ which will be used in Section \ref{sec:CT} to construct one-dimensional representations of twisted Yangians of types CI and DIII.

\begin{lemma} \label{L:K-1dim}
Let $a\in\C$ and let $\mcG$ be of type CI or DIII. Then the matrix
\eq{
K(u;a) = \mcG + a\,u^{-1}I  \label{K-1dim:C1D3} 
}
is a one-parameter solution of \eqref{RE}. Moreover, it satisfies the symmetry relation 
\eq{
K^t(u;a) = - \,K(\ka-u;a) \pm \frac{K(u;a)-K(\ka-u;a)}{2u-\ka} - \frac{ \Tr(K(u;a))\cdot I }{2u-2\ka}. \label{K-1dim:Sym}
}
\end{lemma}

\begin{proof}
We begin by showing that $\mcG+a\,u^{-1} I$ satisfies the reflection equation \eqref{RE}, that is
$$
 R_{12}(u-v)(\mcG_1+a\,u^{-1}I_1)R_{12}(u+v)(\mcG_2+a\,v^{-1}I_2)=(\mcG_2+a\,v^{-1}I_2)R_{12}(u+v)(\mcG_1+a\,u^{-1}I_1)R_{12}(u-v). 
$$
The matrices $\mcG$ and $I$ are themselves solutions of \eqref{RE}, thus it suffices to show that 
\spl{
& v^{-1}R_{12}(u-v)\mcG_1R_{12}(u+v)+u^{-1}R_{12}(u-v)R_{12}(u+v)\mcG_2 \\
& \qq =u^{-1}\mcG_2R_{12}(u+v)R_{12}(u-v)+v^{-1}R_{12}(u+v)\mcG_1R_{12}(u-v).\label{CT:Lemma.tracereps.RE.2}
}
Notice first that 
\spl{
& \left(1-\frac{P}{u-v}\right)\mcG_1\left(1-\frac{P}{u+v}\right)v^{-1}+\left(1-\frac{P}{u-v}\right)\left(1-\frac{P}{u+v}\right)\mcG_2u^{-1} 
\\
& \qq = \mcG_2\left(1-\frac{P}{u+v}\right)\left(1-\frac{P}{u-v}\right)u^{-1}+\left(1-\frac{P}{u+v}\right)\mcG_1\left(1-\frac{P}{u-v}\right)v^{-1}.\label{CT:Lemma.tracereps.RE.3}
}
This is verified by expanding both sides. After subtracting \eqref{CT:Lemma.tracereps.RE.3},
the left hand side of \eqref{CT:Lemma.tracereps.RE.2} becomes
\eqn{
&\frac{1}{u+v-\ka}\left(\frac{\mcG_1Q}{v}-\frac{P\mcG_1Q}{(u-v)v}+\frac{Q\mcG_2}{u}-\frac{PQ\mcG_2}{(u-v)u} \right) 
\\
&\qq + \frac{1}{u-v-\ka}\left(\frac{Q\mcG_1}{v}-\frac{Q\mcG_1P}{v(u+v)}+\frac{Q\mcG_1Q}{(u+v-\ka)v}+\frac{Q\mcG_2}{u}-\frac{QP\mcG_2}{(u+v)u}+\frac{Q^2\mcG_2}{(u+v-\ka)u}\right) ,
}
while the right hand side becomes
\eqn{
& \frac{1}{u-v-\ka}\left(\frac{\mcG_2Q}{u}-\frac{\mcG_2PQ}{u(u+v)}+\frac{\mcG_1Q}{v}-\frac{P\mcG_1Q}{(u+v)v} \right) 
\\
&\qq + \frac{1}{u+v-\ka}\left(\frac{\mcG_2Q}{u}-\frac{\mcG_2QP}{u(u-v)}+\frac{\mcG_2Q^2}{u(u-v-\ka)}+\frac{Q\mcG_1}{v}-\frac{Q\mcG_1P}{(u-v)v}+\frac{Q\mcG_1Q}{v(u-v-\ka)}\right).
}
Multiplying both sides by $(u+v-\ka)(u-v-\ka)(u^2-v^2)uv$ and equating the coefficients of $u^iv^j$, we see that it suffices to establish the following relations:
\eqa{
\ka Q\mcG_2+QP\mcG_2-Q^2\mcG_2+\ka Q\mcG_2+PQ\mcG_2&=\ka\mcG_2Q+\mcG_2QP-\mcG_2Q^2+\ka\mcG_2Q+\mcG_2PQ , \label{CT:Lemma.tracereps.RE.4}
\\
Q\mcG_1+Q\mcG_2-\mcG_1Q+Q\mcG_2&=\mcG_2Q-Q\mcG_1+\mcG_2Q+\mcG_1Q , \label{CT:Lemma.tracereps.RE.5}
\\
-Q\mcG_1P+QP\mcG_2+P\mcG_1Q+PQ\mcG_2&=\mcG_2QP+Q\mcG_1P+\mcG_2PQ-P\mcG_1Q. \label{CT:Lemma.tracereps.RE.6}
}
Since $Q^2=NQ$ and $PQ=QP=\pm Q$, \eqref{CT:Lemma.tracereps.RE.4} is equivalent to $(2\ka-N)(Q\mcG_2-\mcG_2Q)=\mp2(Q\mcG_2-\mcG_2Q)$, and since $\ka=N/2\mp1$, this equality is indeed satisfied. 
As $P\mcG_1=\mcG_2P$ and $\mcG^t=-\mcG$, we have $-\mcG_1Q=(P\mcG_1)^{t_1}=(\mcG_2P)^{t_1}=\mcG_2Q$. Similarly $Q\mcG_1=-Q\mcG_2$, from which \eqref{CT:Lemma.tracereps.RE.5} follows.
Relation \eqref{CT:Lemma.tracereps.RE.6} holds since $P\mcG_2=\mcG_1P$, $P\mcG_1=\mcG_2P$ and $PQ=QP$. 
This completes the proof that $K(u;a)=\mcG+a\,u^{-1}I$ is a solution of \eqref{RE}.

\smallskip

By Lemma \ref{L:K-RE-Sym} we already know that \eqref{K-1dim:Sym} holds if $K(u;a)$ is replaced with $\mcG$. Thus it suffices to show that it also does if $K(u;a)$ is replaced with $u^{-1}I$, which can be checked directly.
\end{proof}


\section{Yangians, twisted Yangians and reflection algebras} \label{sec:Y-TY}


\subsection{Yangians and extended Yangians of types B-C-D} \label{sec:Y-BCD}

We briefly recall the algebraic structure of the Yangian $Y(\mfg_N)$ and the extended Yangian $X(\mfg_N)$ of the Lie algebra $\mfg_N$. For complete details and proofs of statements  presented below, please consult \cite{AACFR,AMR}. 

Fix $N\in\Z_{\ge 2}$ if $\mfg=\mfsp_N$ or $N\in\Z_{\ge 3}$ if $\mfg=\mfso_N$. We introduce elements $t_{ij}^{(r)}$ with $-n \le i,j \le n$ and $r\in\Z_{\ge 0}$ such that $t^{(0)}_{ij}= \del_{ij}$. Combining these into formal power series $t_{ij}(u) = \sum_{r\ge 0} t_{ij}^{(r)} u^{-r}$, we can then form the generating matrix $T(u)= \sum_{-n\le i,j\le n} E_{ij} \ot t_{ij}(u)$. 
Let $R(u)$ be given by b) of \eqref{R(u)}. The same $R$-matrix will be used in Subsections \ref{sec:TY-BCD} and \ref{sec:RA-BCD}. 

\begin{defn}
The extended Yangian $X(\mfg_N)$ is the unital associative $\C$-algebra generated by elements $t_{ij}^{(r)}$ with $-n \le i,j \le n$ and $r\in\Z_{\ge 0}$ satisfying the relation
\eq{ \label{RTT}
R(u-v)\,T_1(u)\,T_2(v) = T_2(v)\,T_1(u)\,R(u-v) ,
}
or equivalently
\begin{align*}
[\, t_{ij}(u),t_{kl}(v)]&=\frac{1}{u-v}\Big(t_{kj}(u)\, t_{il}(v)-t_{kj}(v)\, t_{il}(u)\Big)\\
                      {}&-\frac{1}{u-v-\ka}\sum_{-n \le a \le n} \Big(\delta_{k,-i}\,\theta_{ia}\, t_{aj}(u)\, t_{-a,l}(v)-\delta_{l,-j}\,\theta_{ja}\, t_{k,-a}(v)\, t_{ia}(u)\Big). 
\end{align*}
The Hopf algebra structure of $X(\mfg_N)$ is given by 
\eq{ \label{Hopf:X}
\Delta: t_{ij}(u)\mapsto \sum_{-n \le a \le n} t_{ia}(u)\ot t_{aj}(u), \qq S: T(u)\mapsto T(u)^{-1},\qquad \epsilon: T(u)\mapsto I. 
}
\end{defn}

We will call the generating matrix of the extended Yangian $X(\mfg_N)$ the $T$-matrix. We will use the same terminology for the Yangian $Y(N)$, {\it cf.} Subsection \ref{sec:Y-A}. 

\smallskip

Let $A\in GL(N)$ be such that $AA^t=I$ and let $a\in \C$ be a constant. Moreover, let $f(u)\in 1 + u^{-1}\C[[u^{-1}]]$ be an arbitrary formal power series with constant term equal to 1. 
\begin{prop} \label{P:Aut(X)}
The maps  
\begin{equation} 
\varkappa_a : T(u) \mapsto T(u+a), \qquad \mu_f : T(u) \mapsto f(u)\, T(u), \qquad \al_A  : T(u) \mapsto A\,T(u)\, A^t, \label{al_A}
\end{equation}
are automorphisms of $X(\mfg_N)$. 
\end{prop}
\begin{defn} \label{def:YgN}
Consider the subalgebra of $X(\mfg_N)$ defined by
\eq{
Y(\mfg_N) = \{ y \in X(\mfg_N) : \mu_f(y) = y \text{ for any } f(u) \in1+u^{-1}\C[[u^{-1}]] \}. \label{Y(g)}
}
This subalgebra is called the Yangian of the Lie algebra $\mfg_N$.
\end{defn}

There exists a distinguished central series $z(u)=1+\sum_{r>0}z_ru^{-r}$ such that
\eq{ \label{z(u)}
\Delta( z(u) ) = z(u) \ot z(u) , \qq S(z(u)) = z(u)^{-1} , \qq T(u)^{-1} = z(u)^{-1} T^t(u+\ka),
}
Moreover, the coefficients $z_1,z_2,\ldots$ of $z(u)$ are algebraically independent and generate the whole center $ZX(\mfg_N)$ of $X(\mfg_N)$.

By \cite[Thm.~3.1]{AMR}, the algebra $X(\mfg_N)$ is isomorphic to the tensor product of its subalgebras $ZX(\mfg_N)$ and $Y(\mfg_N)$. Since $ZX(\mfg_N)$ is generated by the coefficients of the series $z(u)$, it follows that the Yangian $Y(\mfg_N)$ is isomorphic to the quotient of $X(\mfg_N)$ by the ideal generated by $z_1,z_2,\ldots$ ({\it cf.}~Cor.~3.2 in {\it ibid.}), or in other words, we have that
\eq{
X(\mfg_N) = ZX(\mfg_N) \ot Y(\mfg_N) \qq\Longrightarrow\qq Y(\mfg_N) \cong X(\mfg_N) / (z(u) - 1).
}
Moreover, the centre of $Y(\mfg_N)$ is trivial. We will denote the $T$-matrix of $Y(\mfg_N)$ by $\mcT(u)$. Its matrix elements will be denoted by $(\mcT(u))_{ij} = \tau_{ij}(u)$, and hence $\tau_{ij}(u) = \sum_{r\ge0} \tau^{(r)}_{ij} u^{-r}$ with $\tau_{ij}^{(0)} = \del_{ij}$. 

Let $y(u)$ be the unique series in $1+u^{-1}ZX(\mfg_N)[[u^{-1}]]$ such that $z(u) = y(u)\,y(u+\ka)$. By \eqref{z(u)} the automorphism $\mu_f$ takes $y(u)$ to $f(u)\,y(u)$ and hence the matrix $y(u)^{-1}T(u)$ is $\mu_f$-stable. In summary, we have that
\eq{ \Delta(y(u)) = y(u) \ot y(u), \qq T(u) = y(u)\, \mcT(u), \qq \mcT(u)^{-1} = \mcT^t(u+\ka), \qq \mu_f : \mcT(u) \mapsto \mcT(u). \label{mcT(u)}
}
The Yangian $Y(\mfg_N)$ is a Hopf subalgebra of $X(\mfg_N)$ with the coproduct,
antipode and counit obtained by restricting those of $X(\mfg_N)$; see \cite[Prop.~3.3]{AMR}.

%

\subsection{Twisted Yangians and extended twisted Yangians of types B-C-D} \label{sec:TY-BCD}

These algebras were introduced in \cite{GR}. Let us briefly recall their main algebraic properties. For complete details and proofs of the claims below consult {\it cit.~loc.} It will be convenient to call the generating matrix of a twisted Yangian (or a reflection algebra) the $S$-matrix. Moreover, to be consistent with the notation in \cite{GR} and in  \cite{Mo5}, we denote by $\mcG(u)$ the matrix $K(u)$ of type B--C--D given in Definition \ref{D:K-matrices}.


\begin{defn} \label{D:TX}
The extended twisted Yangian $X(\mfg_N,\mcG)^{tw}$ is the subalgebra of $X(\mfg_N)$ generated by the coefficients  $s_{ij}^{(r)}$ with ${-}n \le i,j \le n$ and $r\in\Z_{\ge0}$ of the entries $s_{ij}(u)$ of the $S$-matrix
\eq{
S(u) = T(u-\ka/2)\,\mcG(u)\,T^t(-u+\ka/2). \label{S=TGT}
}	
\end{defn}

The algebra $X(\mfg_N,\mcG)^{tw}$ is a left coideal subalgebra of $X(\mfg_N)$: $\Delta(X(\mfg_N,\mcG)^{tw}) \subset X(\mfg_N) \otimes X(\mfg_N,\mcG)^{tw}$. We have that $s_{ij}(u) = \sum_{a,b=-n}^n \theta_{jb}\,t_{ia}(u-\ka/2)\,g_{ab}(u)\,t_{-j,-b}(-u+\ka/2)$, where $g_{ab}(u)$ are the matrix entries of $\mcG(u)$ and, by \eqref{Hopf:X},
\eq{ \label{TX-cop}
\Delta(s_{ij}(u)) = \sum_{-n\le a,b\le n} \theta_{jb}\,t_{ia}(u-\ka/2)\,t_{-j,-b}(-u+\ka/2)\ot s_{ab}(u).
}
We now recall some properties of the matrix~$S(u)$. By Lemmas 4.2 and 4.3 in \cite{GR} we have the following.

\begin{lemma} \label{L:TX-RE}
The $S$-smatrix $S(u)$ satisfies the reflection equation 
\eq{ \label{TX-RE}
R(u-v)\,S_1(u)\,R(u+v)\,S_2(v) = S_2(v)\,R(u+v)\,S_1(u)\,R(u-v),
}
or equivalently
\spl{ \label{[s,s]}
[\,s_{ij}(u),s_{kl}(v)]&=\frac{1}{u-v}\Big(s_{kj}(u)\,s_{il}(v)-s_{kj}(v)\,s_{il}(u)\Big)\\
{}& +\frac{1}{u+v} \sum_{a=-n}^n \Big(\delta_{kj}\,s_{ia}(u)\,s_{al}(v)-
\delta_{il}\,s_{ka}(v)\,s_{aj}(u)\Big)\\
{}& -\frac{1}{u^2-v^2} \sum_{a=-n}^n \delta_{ij}\Big(s_{ka}(u)\,s_{al}(v) - s_{ka}(v)\,s_{al}(u)\Big) \\
{}& -\frac{1}{u-v-\ka} \sum_{a=-n}^n \Big( \delta_{k,-i}\,\theta_{i a}\, s_{a j}(u)\, s_{-a,l}(v) - \delta_{l,-j} \,\theta_{a j}\, s_{k,-a}(v)\, s_{i a}(u) \Big) \\
{}& -\frac{1}{u+v-\ka}\, \Big( \theta_{j,-k}\,s_{i,-k}(u)\, s_{-j,l}(v) - \theta_{i,-l}\,s_{k,-i}(v)\, s_{-l,j}(u) \Big) \\
{}& +\frac{1}{(u+v) (u-v-\ka)}\, \theta_{i,-j} \sum_{a=-n}^n  \Big( \delta_{k,-i}\,s_{-j,a}(u)\, s_{a l}(v) - \delta_{l,-j}\,s_{k a}(v)\, s_{a,-i}(u) \Big) \\
{}& +\frac{1}{(u-v) (u+v-\ka)}\, \theta_{i,-j} \Big( s_{k,-i}(u)\, s_{-j,l}(v)-s_{k,-i}(v)\, s_{-j,l}(u) \Big) \\
{}& -\frac{1}{(u-v-\ka) (u+v-\ka)}\, \theta_{i j}  \sum_{a=-n}^n \Big( \delta_{k,-i}\, s_{a a}(u)\, s_{-j,l}(v) - \delta_{l,-j}\,s_{k,-i}(v)\, s_{a a}(u) \Big) ,
}
and the symmetry relation
\eq{
S^t(u) = (\pm)\,S(\ka-u) \pm \frac{S(u)-S(\ka-u)}{2u-\ka} + \frac{\Tr(\mcG(u))\,S(k-u) - \Tr(S(u))\cdot I }{2u-2\ka} \,, \label{TX-Sym}
}
or equivalently
\eqa{ \label{s=s}
\theta_{ij}s_{-j,-i}(u) = (\pm)\,s_{ij}(\ka-u) \pm \frac{s_{ij}(u)-s_{ij}(\ka-u)}{2u-\ka} + \frac{\Tr(\mcG(u))\,s_{ij}(\kappa-u) - \delta_{ij} \sum_{k=-n}^n s_{kk}(u)  }{2u-2\ka} ,
}
where the lower sign in $(\pm)$ distinguishes types CI and DIII from the other cases. 
\end{lemma}

Introduce the formal power series $w(u) = z(-u-\ka/2)\,z(u-\ka/2)$. By \eqref{z(u)} and \eqref{S=TGT} we have that 
\eq{
S(u)\,S(-u) = w(u)\cdot I, \qq w(u)=w(-u), \qq\Delta(w(u))=w(u)\ot w(u) . \label{w(u)}
}
Since the series $z(u)$ are central in $X(\mfg_N)$, the same is true for the series $w(u)$, or in other words, the even coefficients $w_2,w_4,\ldots$ of $w(u)$ are algebraically independent and central in the algebra $X(\mfg_N,\mcG)^{tw}$. The relation $S(u)\,S(-u) = I$ (equivalently, $w(u)=1$) will be called the {\em unitary relation}.

\smallskip


\begin{defn} \label{D:TY}
The twisted Yangian $Y(\mfg_N,\mcG)^{tw}$ is the quotient of $X(\mfg_N,\mcG)^{tw}$ by the ideal generated by the coefficients of the unitary relation, that is
\eq{
Y(\mfg_N,\mcG)^{tw} = X(\mfg_N,\mcG)^{tw} / (S(u)\,S(-u) - I) . \label{Y=X/(SS-I)}
}
\end{defn}

By \cite[Thm.~3.1]{GR} the algebra $Y(\mfg_N,\mcG)^{tw}$ is isomorphic to the subalgebra of $Y(\mfg_N)$ generated by the coefficients $\si_{ij}^{(r)}$ with $r\ge0$ of the matrix entries $\si_{ij}(u)$ of the $S$-matrix $\Si(u)$ defined by 
\eq{
\Si(u) = \mcT(u-\ka/2)\, \mcG(u)\, \mcT^t(-u+\ka/2) . \label{Si=TGT}
}

Define the series $q(u) = y(u-\ka/2)\,y(-u+\ka/2)$. Then, by \eqref{mcT(u)} and \eqref{w(u)}, we have that
\eq{
S(u) = q(u)\, \Si(u), \qq w(u) = q(u)\,q(-u) = q(u)\,q(u+\ka). \label{S=qSi}
}
Given $f(u) \in 1 + u^{-1}\C[[u^{-1}]]$, set $g(u) = f(u)\,f(-u)$ and let $\nu_g$ denote the restriction of the automorphism $\mu_f$ of $X(\mfg_N)$ to the subalgebra $X(\mfg_N,\mcG)^{tw}$. Then, by \eqref{al_A} and \eqref{mcT(u)},
\eq{
\nu_g : S(u) \mapsto g(u-\ka/2)\,S(u), \qq \nu_g : w(u) \mapsto g(u-\ka/2)g(-u-\ka/2) w(u), \qq \nu_g : \Si(u) \mapsto \Si(u). \label{nu_g}
}
This implies that $Y(\mfg_N,\mcG)^{tw}$, viewed as a subalgebra of $X(\mfg_N)$, is the $\nu_g$-stable subalgebra of $X(\mfg_N,\mcG)^{tw}$; see \cite[Cor.~3.1]{GR}. 

\smallskip

Let $ZX(\mfg_N,\mcG)^{tw}$ denote the subalgebra of $X(\mfg_N,\mcG)^{tw}$ generated by the coefficients of $w(u)$. Then, by \cite[Cor.~3.5]{GR}, the subalgebra $ZX(\mfg_N,\mcG)^{tw}$ is the whole centre of $X(\mfg_N,\mcG)^{tw}$ and the centre of $Y(\mfg_N,\mcG)^{tw}$ is trivial. 

\begin{crl} [{\cite[Cor.~3.6]{GR}}]\label{C:TX=ZTX*TY}
The algebra $X(\mfg_N,G)^{tw}$ is isomorphic to the tensor product of its centre and the subalgebra
$Y(\mfg_N,\mcG)^{tw}$, namely
\eq{
X(\mfg_N,\mcG)^{tw} \cong ZX(\mfg_N,\mcG)^{tw} \ot Y(\mfg_N,\mcG)^{tw}. \label{TX=ZTX*TY}
}
\end{crl}

\begin{rmk} \label{R:TX-al_A}
It was noted in \cite[Remark 3.2]{GR}, that the automorphism $\al_A$ of $X(\mfg_N)$ (see \eqref{al_A}) restricts to an isomorphism between the algebras $X(\mfg_N,\mcG)^{tw}$ and $X(\mfg_N,A^t\mcG A)^{tw}$.  Given a matrix $\mcG$ let $A\in GL(N)$ be such that $AA^t=I$ and $A^t\mcG A=\mcG$. Then $\al_A$ restricts to an automorphism of $X(\mfg_N,\mcG)^{tw}$. The same can be said for $Y(\mfg_N,\mcG)^{tw}$.
\end{rmk}

\begin{prop}  \label{P:F->TY}
Set $F^{\prime \rho}_{ij}=\sum_{a=-n}^n(F_{ia}g_{aj}+g_{ia}F_{aj})$ and $\bar{g}_{ij}=0$ if $\mcG(u)$ is of the first kind and $\bar{g}_{ij}=(g_{ij} - \del_{ij})\,c^{-1}$ if $\mcG(u)$ is of the second kind. Then the assignment $F^{\prime \rho}_{ij}\mapsto \si^{(1)}_{ij} - \bar{g}_{ij}$ (resp. $\F{i}{j}\mapsto s_{ij}^{(1)}-\bar{g}_{ij}$) defines an embedding $\mf{U}\mfg^\rho_N\into Y(\mfg_N,\mcG)^{tw}$ (resp. $\mf{U}\mfg^\rho_N\into X(\mfg_N,\mcG)^{tw}$).
\end{prop}

\begin{proof}
The embedding $\mf{U}\mfg^\rho_N\into Y(\mfg_N,\mcG)^{tw}$ was obtained in Corollary 3.3 in \cite{GR}. Since $Y(\mfg_N,\mcG)^{tw}$ can be viewed as a subalgebra of $X(\mfg_N,\mcG)^{tw}$, the assignment $\F{i}{j}\mapsto \si_{ij}^{(1)}-\bar{g}_{ij}$ defines an 
inclusion $\iota:\mathfrak{U}\mfg_N^{\rho}\into X(\mfg_N,\mcG)^{tw}$. By \eqref{S=qSi} we have $\Si(u)=q(u)^{-1}S(u)$, where $q(u)=y(u-\ka/2)\,y(-u+\ka/2)$. The coefficient of $u^{-1}$ in the expansion of $q(u)$ is zero, thus $\si_{ij}^{(1)}=s_{ij}^{(1)}$ 
for all $-n\le i,j\le n$. 
\end{proof}

We end this subsection by rephrasing Theorem 3.2 in \cite{GR} for the choice of matrix $\mcG$ in Subsection \ref{sec:symm-pairs}. The only difference is for type BDI, in which case the matrix $\Sigma(u)$ from \eqref{Si=TGT} is related to the matrix of generators considered in \cite{GR} via conjugation by a suitable matrix $A$ satisfying $A^{-1} = A^t$.

\begin{crl} \label{Y:PBW} 
Given any total ordering on the set of generators $\si_{ij}^{(r)}$ of $Y(\mfg_N,\mcG)^{tw}$ 
a vector space basis of $Y(\mfg_N,\mcG)^{tw}$ is provided by ordered monomials in the following generators $(r\ge 1)$:
\begin{itemize} [itemsep=1ex]
\item {\rm BD0\hspace{2.7mm}:} $\si_{ij}^{(2r-1)}$ with $i+j>0$.

\item {\rm C0\hspace{5.3mm}:} $\si_{ij}^{(2r-1)}$ with $i+j\ge0$.

\item {\rm CI\hspace{5.8mm}:} $\si_{ij}^{(2r-1)}$ with $i,j>0$; and $\si_{ij}^{(2r)}$ with $i+j\ge0$, $ij < 0$.

\item {\rm DIII\hspace{3mm}:} $\si_{ij}^{(2r-1)}$ with $i,j>0$; and $\si_{ij}^{(2r)}$ with $i+j>0$, $ij < 0$.

\item {\rm BI(a)\hspace{1.4mm}:} $\si_{ij}^{(2r-1)}$ with $i+j>0$ and $|i|,|j|\le\frac{p-1}{2}$ or $|i|,|j|\ge \frac{p+1}{2}$; 
\\ 
and $\si_{ij}^{(2r)}$ with $i\ge\frac{p+1}{2}$, $|j|\le \frac{p-1}{2}$ or $j\ge\frac{p+1}{2}$, $|i|\le \frac{p-1}{2}$.

\item {\rm BI(b)\hspace{1.2mm}:} the same as for {\,\rm BI(a)} except $p$ should be replaced with $q$.

\item {\rm CII\hspace{4.4mm}:} $\si_{ij}^{(2r-1)}$ with $i+j\ge0$ and $|i|,|j|\le\frac{q}{2}$ or $|i|,|j|\ge \frac{q}{2}+1$; 
\\ 
\mbox{} \hspace{9.5mm} and $\si_{ij}^{(2r)}$ with $i\ge\frac{q}{2}+1$, $|j|\le \frac{q}{2}$ or $j\ge\frac{q}{2}+1$, $|i|\le \frac{q}{2}$.

\item {\rm DI(a)\hspace{1.2mm}:} the same as for {\,\rm CII} except $i+j\ge0$ should be replaced with $i+j>0$.

\item {\rm DI(b)\;:} $\si^{(2r-1)}_{ij}$ with $i+j>0$ and $2\le |i|,|j|\le \frac{p+1}{2}$ or $|i|,|j|\ge \frac{p+3}{2}$, we should also include $\si^{(2r-1)}_{1j}$ and $\si^{(2r-1)}_{j1}$ with $j\ge 2$;
\\
and $\si_{ij}^{(2r)}$ with $i\ge \frac{p+3}{2}$ and $2\le |j|\le \frac{p+1}{2}$ or $j\ge \frac{p+3}{2}$ and $2\le |i|\le \frac{p+1}{2}$, we should also include $\si^{(2r)}_{1,-1}$ and $\si^{(2r)}_{1j}$, $\si^{(2r)}_{j1}$ with $j\ge 2$.
\end{itemize}
\end{crl}

\begin{crl}[see {\cite[Cor.~3.4]{GR}}] \label{X:PBW} 
Given any total ordering on the set of generators $s_{ij}^{(r)}$ of $X(\mfg_N,\mcG)^{tw}$, a vector space basis of $X(\mfg_N,\mcG)^{tw}$ is provided by the ordered monomials in the generators $w_2,w_4,\ldots\,$ and $s^{(r)}_{ij}$ with $r$, $i$, $j$ satisfying the same constraints as in Corollary \ref{Y:PBW}.
\end{crl}

\subsection{Reflection algebras of types B-C-D} \label{sec:RA-BCD}

It was shown in \cite[Sec.~4]{GR} that certain quotients of the extended reflection algebra $\mcX\mcB(\mcG)$ defined by the reflection equation are isomorphic to twisted Yangians $X(\mfg_N,\mcG)^{tw}$ and $Y(\mfg_N,\mcG)^{tw}$. In this subsection we summarize the results obtained in {\it loc.~cit.}

Let $N$ be as in Subsection \ref{sec:Y-BCD}. Introduce elements $\tl\mss_{ij}^{(r)}$ with $-n \le i,j \le n$ and $r\in\Z_{\ge 0}$ such that $\tl\mss^{(0)}_{ij}= g_{ij}$ and combine them into an $S$-matrix $\wt\msS(u)$ in the same way as we did for the $T$-matrix $T(u)$.


\begin{defn} \label{D:XBG}
The extended reflection algebra $\mcX\mcB(\mcG)$ is the unital associative $\C$-algebra generated by elements $\tl\mss_{ij}^{(r)}$ with $-n\le i,j\le n$, $r\in\Z_{\ge 0}$ and satisfying the reflection equation 
\eq{ \label{RE:XB}
R(u-v)\,\wt\msS_1(u)\,R(u+v)\,\wt\msS_2(v) = \wt\msS_2(v)\,R(u+v)\,\wt\msS_1(u)\,R(u-v) ,
}
or equivalently \eqref{[s,s]} with $s_{ij}(u)$ replaced with $\tl\mss_{ij}(u)$.
\end{defn}

Let $h(u)\in1+u^{-1}\C[[u^{-1}]]$ and let $A\in GL(N)$ be such that $AA^t=I$ and $A\mcG A^t=\mcG$.

\begin{prop} \label{Aut(XBG)}
The maps
\begin{equation} \wt\nu_h : \wt\msS(u) \mapsto h(u)\,\wt\msS(u), \quad \wt\ga : \wt\msS(u) \mapsto \wt\msS(-u)^{-1}, \quad \wt\al_A  : \wt\msS(u) \mapsto A\,\wt\msS(u)\, A^t \label{tnu_h}
\end{equation}
are automorphisms of $\mcX\mcB(\mcG)$.
\end{prop}


\begin{defn} \label{D:BG}
The reflection algebra $\mcB(\mcG)$ is the quotient of $\mcX\mcB(\mcG)$ by the ideal generated by the coefficients of the symmetry relation \eqref{TX-Sym} with $S(u)$ replaced with  $\wt\msS(u)$, or equivalently \eqref{s=s} with $s_{ij}(u)$ replaced with $\tl\mss_{ij}(u)$.
\end{defn}

By \cite[Prop.~5.1 and Cor.~5.1]{GR}, there exists a unique formal power series $\msc(u)=1+\sum_{r\ge1} \msc_r u^{-1}$ with coefficients central in $\mcX\mcB(\mcG)$ such that
\eq{ \label{sc(u)}
Q\,\wt\msS_1(u)\,R(2u-\ka)\,\wt\msS_2(\ka-u)^{-1} = \wt\msS_2(\ka-u)^{-1} R(2u-\ka)\,\wt\msS_1(u)\,Q = p(u)\,\msc(u)\,Q 
}
with $p(u)$ given by \eqref{p(u)}. The uniqueness of $\msc(u)$ is a consequence of the invertibility of $p(u)$ - see \eqref{p(u)p(k-u)}. Moreover, $\msc(u)$ satisfies the relation $\msc(u)^{-1}=\msc(\ka-u)$ and the odd coefficients are algebraically independent.
By \cite[Thm.~4.2 and Cor.~5.2]{GR}, we have the following isomorphism of algebras:
\eq{
X(\mfg_N,\mcG)^{tw} \cong \mcB(\mcG) \cong \mcX\mcB(\mcG) / (\msc(u) - 1) . \label{iso:TX-BG}
}

Combining \eqref{tnu_h} with \eqref{sc(u)} we have that $\wt\nu_h(\msc(u)) = h(u)\,h(\ka-u)^{-1}\msc(u)$. Let $\msv(u)$ be the unique central invertible power series with constant term $1$ such that $\msc(u) = \msv(u)^2$. Hence $\msc(u) = \msv(u)\,\msv(\ka-u)^{-1}$. 
Set $\mss_{ij}(u) = \msv(u)^{-1}\,\wt\mss_{ij}(u)$ and let series $h(u)$ be such that $h(u)^{-1} = h(\ka-u)$. Then, by \cite[Prop.~5.2 and Thm.~5.3]{GR}, the subalgebra of $\mcX\mcB(\mcG)$ generated by the coefficients $\mss_{ij}^{(r)}$ with $r\in\Z_{\ge0}$ of the series $\mss_{ij}(u)$ is the $\wt\nu_h$-stable subalgebra of $\mcX\mcB(\mcG)$ and is isomorphic to $\mcB(\mcG)$. We will denote by $\mcZ\mcX(\mcG)$ the subalgebra of $\mcX\mcB(\mcG)$ generated by the odd coefficients of $\msc(u)$.

\smallskip

Let $\msS(u)$ denote the $S$-matrix of $\mcB(\mcG)$. By \cite[Prop.~4.1]{GR}, the product $\msS(u)\,\msS(-u) = \msw(u)\cdot I$ is a scalar matrix, where $\msw(u)$ is an even formal power series in $u^{-1}$ with coefficients central in $\mcB(\mcG)$. We will denote by $\mcZ\mcB(\mcG)$ the subalgebra of $\mcB(\mcG)$ generated by the even coefficients of $\msw(u)$.


\begin{defn} \label{D:UBG}
The unitary reflection algebra $\mcU\mcB(\mcG)$ is the quotient of $\mcB(\mcG)$ by the ideal generated by the coefficients of the unitary relation $\msS(u)\,\msS(-u)=I$.
\end{defn}

By \cite[Thm.~4.1]{GR}, we have the following isomorphism of algebras:
\eq{
Y(\mfg_N,\mcG)^{tw} \cong \mcU\mcB(\mcG) \cong \mcB(\mcG) / (\msw(u) - 1) . \label{iso:TY-UBG}
}

It will be convenient to denote $\mcX\mcB(\mcG)$ by $\wt X(\mfg_N,\mcG)^{tw}$ and $\mcZ\mcX(\mcG)$ by $\wt{ZX}(\mfg_N,\mcG)^{tw}$.


\subsection{Yangians of type A} \label{sec:Y-A}

We briefly recall the algebraic properties of the Yangian $Y(N)$ of the Lie algebra $\mfgl_N$ and the Yangian $SY(N)$ of the Lie algebra $\mfsl_N$. For more details consult~\cite{MNO,Mo5}. 

\smallskip

Let $N\ge2$ and let the matrix $T(u)$ be defined in the same way as in Subsection \ref{sec:Y-BCD}. Moreover, let $R(u)$ be given by a) of \eqref{R(u)}.  This $R$-matrix will also be used in Subsections \ref{sec:TY-Ols} and \ref{sec:TY-MR}.

\begin{defn}
The Yangian $Y(N)$ is the unital associative $\C$-algebra generated by elements $t_{ij}^{(r)}$ with $-n \le i,j \le n$ and $r\in\Z_{\ge 0}$ satisfying the relation \eqref{RTT}, or equivalently
\eq{ \label{RTT:A}
[\, t_{ij}(u),t_{kl}(v)] =\frac{1}{u-v} \Big(t_{kj}(u)\, t_{il}(v)-t_{kj}(v)\, t_{il}(u)\Big) .
}
The Hopf algebra structure of $Y(N)$ is the same as in \eqref{Hopf:X}.
\end{defn}

All the maps listed in Proposition \ref{P:Aut(X)} are also automorphisms of $Y(N)$, and in addition so is the map $\beta_A : T(u) \mapsto A\,T(u)\, A^{-1}$, where $A\in GL(N)$ is any invertible matrix.

\begin{defn} 
The special Yangian $SY(N)$ is the subalgebra of $Y(N)$ given by
\eq{ 
SY(N) = \{ y \in Y(N) : \mu_f(y) = y \text{ for any } f=f(u)\in 1+ u^{-1}\C[[u^{-1}]] \}. \label{SY(N)}
}
It is also called the Yangian of $\mfsl_N$. 
\end{defn}

There exists a distinguished central series $\qdet T(u)$, called the quantum determinant, whose  coefficients $d_1,d_2,\ldots$ are, by \cite[Thm.~2.13, Cor.~2.18]{MNO}, algebraically independent and generate the whole centre $Z(N)$ of $Y(N)$. Moreover, the quotient algebra $Y(N) / (\qdet T(u) - 1)$ is isomorphic to the Yangian $SY(N)$ and $Y(N) \cong Z(N) \ot SY(N)$. Furthermore, the centre of $SY(N)$ is trivial.

Let $d(u)$ be the unique series such that $\qdet T(u) = d(u)\,d(u-1)\cdots d(u-N+1)$. Then $\mu_f(d(u)) = f(u)\,d(u)$ and hence the matrix $\mcT(u) = d(u)^{-1} T(u)$ is stable under the action of $\mu_f$, i.e.~it is the $T$-matrix of $SY(N)$.

\begin{rmk}
In \cite{Mo5}, the algebras $Y(N)$ and $SY(N)$ are denoted by $Y(\mfgl_N)$ and $Y(\mfsl_N)$, respectively. We will use the former notation to avoid possible confusion with the Yangians $X(\mfg_N)$ and $Y(\mfg_N)$ of the Lie algebra~$\mfg_N$.
\end{rmk}


\subsection{Olshanskii Twisted Yangians} \label{sec:TY-Ols}

In this subsection we briefly recall the presentation of the twisted Yangians $Y^\pm(N)$ introduced by G. Olshanskii in \cite{Ol} and surveyed in \cite{MNO}.

\smallskip

\begin{defn}
The twisted Yangian $Y^\pm(N)$ is the subalgebra of $Y(N)$ generated by the coefficients $s_{ij}^{(r)}$ with $-n\le i,j \le n$ and $r\in \Z_{\ge0}$ of the matrix elements $s_{ij}(u)$ of the $S$-matrix
\eq{
S(u) = T(u) \, T^t(-u) . \label{S=TT}
}
\end{defn}

We have that $s_{ij}(u) = \sum_{a=-n}^n \theta_{aj}\,t_{ia}(u)\,t_{-j,-a}(-u)$ and
\eq{
\Delta(s_{ij}(u)) = \sum_{-n\le a,b \le n} \theta_{ja} t_{ia}(u)\,t_{-j,-b}(-u) \ot s_{ab}(u),
}
which implies that $Y^\pm(N)$ is a left coideal subalgebra of $Y(N)$. 

\begin{lemma}
The $S$-matrix $S(u)$ defined by \eqref{S=TT} satisfies the twisted reflection equation
\eq{
R(u-v)\,S_1(u)\,R^t(-u-v)\,S_2(v) = S_2(v)\,R^t(-u-v)\,S_1(u)\,R(u-v), \label{RE-Ols}
}
where $R(u)$ is given by \eqref{R(u)} a). Equivalently
\spl{ \label{[s,s]-Ols}
[\,s_{ij}(u),s_{kl}(v)]&=\frac{1}{u-v}\Big(s_{kj}(u)\,s_{il}(v)-s_{kj}(v)\,s_{il}(u)\Big)\\
& -\frac{1}{u+v}\Big(\theta_{k,-j}\,s_{i,-k}(u)\,s_{-j,l}(v)-\theta_{i,-l}\,s_{k,-i}(v)\,s_{-l,j}(u)\Big)\\
& +\frac{1}{u^2-v^2}\, \theta_{i,-j} \Big( s_{k,-i}(u)\,s_{-j,l}(v) - s_{k,-i}(v)\,s_{-j,l}(u)\Big) ,
}
and the symmetry relation
\eq{
S^t(u) = S(-u) \pm \frac{S(u)-S(-u)}{2u}.  \label{Sym-Ols}
}
\end{lemma}
\begin{prop} [{\cite[Prop.~3.11]{MNO}}] \label{P:Y-Ols->U}
The map $\mathrm{ev}_{\pm} : Y^\pm(N) \to \mfU\mfg_N$, $s_{ij}(u) \mapsto \del_{ij}\,I + \big(u \pm \tfrac12\big)^{-1} F_{ij}$ is a homomorphism of algebras and the map $\iota : \mfU\mfg_N \to Y^\pm(N)$, $F_{ij} \mapsto s_{ij}^{(1)}$ is an embedding of algebras. Moreover, $\mathrm{ev}_{\pm}\circ\iota = \mathrm{id}$.
\end{prop}
There exists a distinguished central series $\sdet S(u)$, called the Sklyanin determinant, whose even coefficients are algebraically independent and generate the whole centre $Z^\pm(N)$ of $Y^\pm(N)$. Introduce the function $\ga_N(u)$ by 
\eq{
\ga_N(u) = 1 \qu \text{for } Y^+(N) \qu\text{and}\qu \ga_N(u)=\frac{2u+1}{2u-N+1} \qu \text{for } Y^-(N).
}
By \cite[Thms.~4.7, 4.11]{MNO}, we have that $\sdet S(u) = \ga_N(u) \,\qdet T(u)\,\qdet T(-u+N-1)$.

Let $g(u)\in 1+ u^{-2}\C[[u^{-2}]]$ and let $f(u) \in 1+ u^{-1}\C[[u^{-1}]]$ be the unique power series with constant term 1 such that $g(u) = f(u)^2$. Then $g(u)=g(-u) \Rightarrow f(u)^2 = f(-u)^2 \Rightarrow f(u)=f(-u)$ because both have constant term 1. Therefore, $g(u) = f(u)f(-u)$ and we let $\nu_g$ denote the restriction of $\mu_f$ to $Y^\pm(N)$. If $g(u) = f_1(u)f_1(-u) = f_2(u) f_2(-u)$, then $f_1(u)$ and $f_2(u)$ differ by multiplication by a series $h(u)$ such that $h(u)h(-u)=1$. 

\begin{defn}\label{def:SYN}
The special twisted Yangian $SY^\pm(N)$ is the subalgebra of $Y^\pm(N)$ given by
\eq{ 
SY^\pm(N) = \{ y \in Y^\pm(N) : \nu_g(y) = y \text{ for any } g=g(u)\in 1+ u^{-2}\C[[u^{-2}]] \}. \label{STY(N)}
}
\end{defn}

\begin{prop} [{\cite[Thm.~2.9.2, Cor.~2.9.3]{Mo5}}] 
The algebra $Y^\pm(N)$ is isomorphic to the tensor product of its centre $Z^\pm(N)$ and the subalgebra $SY^\pm(N)$:
\eq{
Y^\pm(N) \cong Z^\pm(N) \ot SY^\pm(N).
}
Consequently, $SY^\pm(N) \cong Y^\pm(N) / (\sdet\,S(u) - \ga_N(u))$ and the centre of $SY^\pm(N)$ is trivial.
\end{prop}

Set $q(u) = d(u)\,d(-u)$ and $\Si(u) = q(u)^{-1} S(u)$. Since $\mu_f(d(u)) = f(u)\,d(u)$, we have that $\nu_g(\Si(u)) = \Si(u)$ and thus $\Si(u)$ is the $S$-matrix of $SY^\pm(N)$.
\begin{rmk}
There is an extended twisted Yangian $\wt Y^\pm(N)$ which plays a role analogous to $\mcX\mcB(\mcG)$ for the twisted Yangian $Y^{\pm}(N)$. For the general results concerning
this algebra we refer the reader to \cite[\S 2.13]{Mo5}.
\end{rmk}


\subsection{Molev-Ragoucy reflection algebras} \label{sec:TY-MR}

In this subsection we recall the structure of Molev-Ragoucy reflection algebras $\mcB(N,q)$ introduced in \cite{MR}. We recall that, in this exceptional case, we label the rows and columns of $N\times N$ matrices from $1$ to $N$. In particular, the generators of $Y(N)$ are $t_{ij}^{(r)}$ with $1\le i,j\le N, r\ge 0$.

\begin{defn}
The reflection algebra $\mcB(N,q)$ is the subalgebra of $Y(N)$ generated by the coefficients $b_{ij}^{(r)}$ with $1\leq i,j\leq N$ and $r\in \Z_{\ge0}$ of the matrix elements $b_{ij}(u)$ of the matrix $B(u)$ given by 
\eq{
B(u) = T(u) \, \mcG \, T(-u)^{-1} , \label{B=TGT}
}
where $\mcG$ is the diagonal matrix $\mcG={\rm diag}(\veps_{1},\ldots,\veps_N)$ with $\veps_i=1$ for $1\le i \le N-q$ and $\veps_i = -1$ for $N-q < i \le N$. 
\end{defn}

The algebra $\mcB(N,q)$ is a left coideal subalgebra of $Y(N): \Delta(\mcB(N,q)) \subset Y(N) \ot \mcB(N,q)$. 

\begin{lemma} \label{L:B(u)}
The matrix $B(u)$ defined by \eqref{B=TGT} satisfies the reflection equation \eqref{TX-RE} with $R(u) = I - u^{-1}P$, or equivalently
\spl{ \label{[b,b]}
[\,b_{ij}(u),b_{kl}(v)]&=\frac{1}{u-v}\Big(b_{kj}(u)\,b_{il}(v)-b_{kj}(v)\,b_{il}(u)\Big)\\
& +\frac{1}{u+v} \sum_{a=1}^N \Big(\delta_{kj}\,b_{ia}(u)\,b_{al}(v)-
\delta_{il}\,b_{ka}(v)\,b_{aj}(u)\Big)\\
& -\frac{1}{u^2-v^2} \sum_{a=1}^N \delta_{ij}\Big(b_{ka}(u)\,b_{al}(v) - b_{ka}(v)\,b_{al}(u)\Big) ,
}
and the unitary relation $B(u)\,B(-u) = I$.
\end{lemma}

There exists a distinguished central series $\sdet B(u)$, called the Sklyanin determinant, whose odd coefficients are algebraically independent and generate the whole centre $\mcZ(N,q)$ of $\mcB(N,q)$. Introduce the function $\theta(u)$ by 
\eq{
\theta(u) = (-1)^q \prod_{1\le i\le N-q} (2u-2N+2i) \prod_{1\le i\le q} (2u-2N+2i) \prod_{1\le i\le N} \frac{1}{2u-2N+i+1}.
} 
By \cite[Thm.~3.4]{MR}, we have the identity
\eq{
\sdet B(u) = \theta(u)\, \qdet T(u) \,(\qdet T(-u+N-1))^{-1}. \label{sdetB(u)}
}

Let $g(u) \in 1 + u^{-1}\C[[u^{-1}]]$ be such that $g(u)g(-u)=1$ and let $f(u) \in 1 + u^{-1} \C[[u^{-1}]]$ be any power series such that $g(u) = f(u) f(-u)^{-1}$: the existence of $f(u)$ can be established as before Definition \ref{def:SYN}. Such a power series is unique up to multiplication by an invertible even series.

\begin{defn}
The special reflection algebra $\mcS\mcB(N,q)$ is the subalgebra of $\mcB(N,q)$ given by
\eq{
\mcS\mcB(N,q) = \{ b \in \mcB(N,q) : \nu_g(b) = b \text{ for any } g=g(u) \in 1 + u^{-1}\C[[u^{-1}]] \text{ such that } g(u)g(-u)=1 \}.
}
\end{defn}

\begin{prop}
The algebra $\mcB(N,q)$ is isomorphic to the tensor product of its centre $\mcZ(N,q)$ and the subalgebra $\mcS\mcB(N,q)$:
\eq{
\mcB(N,q)\cong \mcZ(N,q) \ot \mcS\mcB(N,q).
}
Consequently, we have the isomorphism of algebras $\mcS\mcB(N,q) \cong \mcB(N,q) / ( \sdet B(u) - \theta(u) )$ and the centre of $\mcS\mcB(N,q)$ is trivial.
\end{prop}

Introduce the elements $\tl b_{ij}^{(r)}$ with $1\le i,j \le N$ and $r\in\Z_{\ge0}$ and the matrix $\wt B(u)$ given by $\wt B(u) = \sum_{1 \le i,j \le N} E_{ij} \ot \tl b_{ij}(u)$ with $\tl b_{ij}(u) = \sum_{r=0}^{\infty} \tl b_{ij}^{(r)} u^{-r}$ and $\tl b_{ij}^{(0)} = \veps_i\del_{ij}$.

\begin{defn}
The extended reflection algebra $\wt{\mcB}(N,q)$ is the associative $\C$-algebra generated by the elements $\tl b_{ij}^{(r)}$ for $1\le i, j \le N$, $r \in\Z_{\ge 0}$ subject to the reflection equation \eqref{TX-RE} with $S(u)$ replaced by $\wt B(u)$.
\end{defn}

All the maps listed in Proposition \ref{Aut(XBG)} are also automorphisms of $\wt\mcB(N,q)$, and in addition so is the map $\wt\beta_A : \wt B(u) \mapsto A\,\wt B(u)\, A^{-1}$, where $A$ is any invertible matrix.

\begin{prop} [{\cite[Prop.~2.1]{MR}}] \label{P:B=Z*SB}
In the algebra $\wt\mcB(N,q)$ the product $\wt B(u)\,\wt B(-u) = \ms{f}(u)\cdot I$ where $\ms{f}(u)$ is an even series in $u^{-1}$ with coefficients central in $\wt\mcB(N,q)$. In particular, $\mcB(N,q) \cong \wt\mcB(N,q) / (\ms{f}(u)-1)$.
\end{prop}


\section{Highest Weight Representations}  \label{sec:HWT}

In this section, we develop a highest weight theory for representations of the extended twisted Yangians $X(\mfg_N,\mcG)^{tw}$ of type B-C-D where the matrix $\mcG$ is assumed to be diagonal. When $\mcG$ is non-diagonal, the corresponding symmetric pair $(\mfg_N,\mfg_N^\rho)$ is of type DI(b), and the associated twisted Yangians possess many exceptional features which shall be 
addressed in future work. 

\smallskip 

Let us begin by recalling some of the analogous results developed in \cite{AMR} for the extended Yangians $X(\mfg_N)$ of type B-C-D in their $RTT$-presentations. 

\smallskip 

A representation $V$ of $X(\mfg_N)$ is a \textit{highest weight representation} if there exists a nonzero vector $\xi\in V$ such that $V=X(\mfg_N)\,\xi$ and the following 
conditions are satisfied:
\begin{alignat*}{4}
  & t_{ij}(u)\,\xi=0 \quad &&\text{ for } && {-}n\leq i<j\leq n, \quad &&\text{ and } \nonumber\\
  & t_{ii}(u)\,\xi=\lambda_i(u)\,\xi \quad  &&\text{ for } && {-}n\leq i\leq n, && 
\end{alignat*}
and, for each $-n\leq i\leq n$, $\lambda_i(u)$ is a formal power series in $u^{-1}$ with constant term equal to $1$: 
$$\lambda_i(u)=1+\sum_{r=1}^{\infty}\lambda_{i}^{(r)}u^{-r}, \quad \lambda_i^{(r)}\in \mathbb{C}. $$
The vector $\xi$ is called the \textit{highest weight vector} of $V$, and the $N$-tuple $\lambda(u)=(\lambda_{-n}(u),\ldots,\lambda_n(u))$ is called the \textit{highest weight}
of V. 

Given an $N$-tuple $\lambda(u)$, the Verma module $M(\lambda(u))$ is defined as the quotient of $X(\mfg_N)$ by the left ideal generated by all the coefficients of the series 
$t_{ij}(u)$ with $-n\leq i<j\leq n$ and $t_{ii}(u)-\lambda_i(u)$ with $-n\leq i\leq n$. The Verma module $M(\lambda(u))$ is non-trivial if and only if the components of the highest
weight satisfy 
\begin{equation}
 \frac{\lambda_{-i}(u)}{\lambda_{-i-1}(u)}=\frac{\lambda_{i+1}(u-\ka+n-i)}{\lambda_{i}(u-\ka+n-i)} \label{HWT:Ext.non-trivial}
\end{equation}
for $i\in \{0,\ldots,n-1\}$ if $\mfg_N$ is of type B and $i\in\{1,\ldots,n-1\}$ otherwise. (To obtain the exact relation in \cite[Prop.~5.14]{AMR}, replace $i$ by $n-j$, $u$ by $u+\kappa-j$ and then take the reciprocal on both sides.)  Moreover, if $M(\lambda(u))$ is non-trivial, then it has a unique irreducible (non-zero) quotient $L(\lambda(u))$, and any
irreducible highest weight $X(\mfg_N)$-module with highest weight $\lambda(u)$ is isomorphic to $L(\lambda(u))$.  If $L(\lambda(u))$ exists (i.e.~if $M(\lambda(u))$ is non-trivial) then,
by \cite[Thm.~5.16]{AMR}, it is finite-dimensional if and only if there exist monic polynomials $P_1(u),\ldots, P_n(u)$ in $u$ such that
\begin{equation}
 \frac{\lambda_{i-1}(u)}{\lambda_i(u)}=\frac{P_i(u+1)}{P_i(u)} \quad  \textit{ for all }\quad 2\leq i\leq n, \label{HWT:Ext.All.P_n} 
\end{equation}
and in addition
\begin{alignat}{2}
&\frac{\lambda_0(u)}{\lambda_1(u)}=\frac{P_1(u+1/2)}{P_1(u)} \quad &&\textit{if }\quad \mfg_N=\mfso_{2n+1}, \label{HWT:Ext.B.P_1}\\
&\frac{\lambda_{-1}(u)}{\lambda_1(u)}=\frac{P_1(u+2)}{P_1(u)} \quad &&\textit{if }\quad \mfg_N=\mfsp_{2n}, \label{HWT:Ext.C.P_1} \\
&\frac{\lambda_{-1}(u)}{\lambda_2(u)}=\frac{P_1(u+1)}{P_1(u)} \quad &&\textit{if }\quad \mfg_N=\mfso_{2n} \label{HWT:Ext.D.P_1}.
\end{alignat}

The polynomials $P_1(u),\ldots,P_n(u)$ are called the Drinfeld polynomials corresponding to $L(\lambda(u))$, and they determine the module $L(\lambda(u))$
up to twisting by an automorphism $\mu_f$ as given in \eqref{P:Aut(X)} - see Corollary 5.19 of \cite{AMR}.


\subsection{Definitions and general theory}

We now turn to the representation theory of the extended twisted Yangians $X(\mfg_N,\mcG)^{tw}$, where the symmetric pair $(\mfg_N,\mfg_N^\rho)$ is not of type DI(b).
In order to treat all cases uniformly, we introduce the following notation:
\begin{itemize}
\item Let $\mathcal{I}_N$ be the indexing set $\{0,\ldots,n\}$ if $\mfg_N$ is of type B, and $\{1,\ldots,n\}$ otherwise.
\item Whenever the symbol $[\pm]$ or $[\mp]$ occurs, the lower sign corresponds to the case where the pair $(\mfg_N,\mfg_N^\rho)$ is of type BI(b), while the upper sign corresponds to the cases where $(\mfg_N,\mfg_N^\rho)$ is of type BCD0, CI, DIII, DI(a), BI(a), or CII. 
\end{itemize}

\begin{defn}
A representation $V$ of $X(\mfg_N,\mcG)^{tw}$ is called a \textit{highest weight representation} if there exists a nonzero vector $\eta\in V$ such that 
 $V=X(\mfg_N,\mcG)^{tw}\eta$ and the following conditions are met: 
\begin{alignat}{4}
  &s_{ij}(u)\,\eta=0 \quad &&\text{ for }\quad &&-n\leq i<j\leq n, \quad &&\text{ and } \nonumber\\
  &s_{ii}(u)\,\eta=\mu_i(u)\,\eta \quad  &&\text{ for }\quad && i\in \mathcal{I}_N, && \label{HWT:highestweightrep}
\end{alignat}
where $\mu_i(u)$ is a formal power series in $u^{-1}$ of the following form: 
\begin{equation*}
  \mu_i(u)=g_{ii}+\sum_{r=1}^{\infty}\mu_{i}^{(r)}u^{-r}, \quad \mu_i^{(r)}\in \mathbb{C}.
\end{equation*}
Set $\mu(u)=\left(\mu_{1}(u),\ldots,\mu_{n}(u)\right)$ if $N=2n$ and  $\mu(u)=\left(\mu_0(u),\ldots,\mu_n(u)\right)$ if $N=2n+1$. We call $\mu(u)$ the \textit{highest weight} of $V$ and $\eta$ the \textit{highest weight vector}. (The expression ``pseudo highest weight'' is also commonly used in the literature, but we will follow the terminology of \cite{AMR} and \cite{Mo5}.)
\end{defn}

\noindent Given a highest weight representation $V$ with highest weight vector $\eta$, a natural question to ask is whether or not $\eta$ is a simultaneous eigenvector for the diagonal elements $s_{-i,-i}(u)$ with $1\leq i \leq n$. This is indeed the case. By relation \eqref{s=s} we have
\begin{equation}
  s_{-i,-i}(u)+\frac{1}{2u-2\ka}\sum_{\ell=1}^n s_{-\ell,-\ell}(u)=p(u)s_{ii}(\ka-u)\pm\frac{s_{ii}(u)}{2u-\ka}-\frac{1}{2u-2\ka}\sum_{\ell \in \mathcal{I}_N} s_{\ell\ell}(u). \label{HWT:neg_weights.1}
\end{equation}
Taking the sums of both sides as $i$ goes from $1$ to $n$ we obtain: 
\begin{equation}
 \left(\frac{2u-2\ka+n}{2u-2\ka}\right)\sum_{\ell=1}^ns_{-\ell,-\ell}(u)= \sum_{\ell=1}^n\left(p(u) s_{\ell\ell}(\ka-u)\pm\frac{s_{\ell\ell}(u)}{2u-\ka}\right)-\frac{n}{2u-2\ka}\sum_{\ell \in \mathcal{I}_N} s_{\ell\ell}(u).\label{HWT:neg_weights.2}
\end{equation}
Substituting this equation back into \eqref{HWT:neg_weights.1} leads to the following result.

\begin{prop}\label{HWT:Prop.neg_weights}
 Let $V$ be a highest weight representation of $X(\mfg_N,\mcG)^{tw}$ with the highest weight vector $\eta$ and the highest weight $\mu(u)$. Then $\eta$ is an eigenvector 
 for the action of $s_{-i,-i}(u)$ for all $1\leq i \leq n$. More explicitly, for each $1\leq i \leq n$  we have the relation: 
 \begin{equation}
   (2\ka-2u-n)s_{-i,-i}(u)\eta=\sum_{\ell=1}^n \beta_{i,\ell}(u)\left(p(u)\mu_{\ell}(\ka-u)\pm\frac{\mu_{\ell}(u)}{2u-\ka}\right)\eta+\sum_{\ell\in \mathcal{I}_N}\mu_\ell(u)\eta, \label{HWT:neg_weights.0}
 \end{equation}
 where $\beta_{i,\ell}(u)=1$ if $\ell\neq i$ and $\beta_{i,\ell}(u)=(2\ka-2u-n+1)$ otherwise. 
\end{prop}

Given a highest weight $\mu(u)$, we shall frequently make use of the corresponding tuple $\wt \mu(u)$ whose components are given by 
\begin{equation}
 \wt\mu_i(u)=(2u-n+i)\mu_i(u)+\sum_{\ell=i+1}^n\mu_\ell(u), \label{HWT:tilde_mu_i(u)}
\end{equation}
for each $i\in \mathcal{I}_N$. The following proposition imposes one condition on $\wt \mu_0(u)$ which will be important later.

\begin{prop} \label{HWT:Prop.mu_0(u)}
 Suppose $(\mfg_N,\mfg_N^\rho)$ is a symmetric pair of type B0 or BI. Let $V$ be a highest weight representation of $X(\mfg_N,\mcG)^{tw}$, with the highest weight vector $\eta$ and the highest weight $\mu(u)$. Then the series $\wt \mu_0(u)$ satisfies the relation
 \begin{equation}
 \wt\mu_0 (\ka-u)=\frac{\ka-u}{u}\cdot p_0(u)p(u)^{-1}\wt\mu_0(u) , \label{HWT:mu_0(u).0}  
 \end{equation}
where $p_0(u)=1-(2u-\ka)^{-1}+N(2u-2\ka)^{-1}$ is equal to the rational function $p(u)$ (see \eqref{p(u)} with $K(u)=\mcG(u)$) corresponding to the symmetric pair $(\mfg_N,\mfg_N^\rho)=(\mfso_{2n+1},\mfso_{2n+1})$. 
\end{prop}

\begin{proof}
By the symmetry relation \eqref{s=s} we have 
\begin{equation*}
  s_{00}(u)=p(u)s_{00}(\ka-u)+\frac{s_{00}(u)}{2u-\ka}-\frac{1}{2u-2\ka}\sum_{\ell=-n}^ns_{\ell\ell}(u),
\end{equation*}
which can be rearranged to 
\begin{equation}
  \left(1-\frac{1}{2u-\ka}+\frac{1}{2u-2\ka}\right)s_{00}(u)+\frac{1}{2u-2\ka}\sum_{\ell=1}^n s_{\ell\ell}(u)=p(u)s_{00}(\ka-u)-\frac{1}{2u-2\ka}\sum_{\ell=1}^ns_{-\ell,-\ell}(u).\label{HWT:mu_0(u).1}
\end{equation}
Multiplying both sides of this relation by $(2\ka-2u-n)$ and substituting in relation \eqref{HWT:neg_weights.2}, the right hand side becomes:
\begin{equation*}
 p(u)\left((2\ka-2u-n)s_{00}(\ka-u)+\sum_{\ell=1}^ns_{\ell\ell}(\ka-u) \right)+\left(\frac{1}{2u-\ka}-\frac{n}{2u-2\ka}\right)\sum_{\ell=1}^n s_{\ell\ell}(u)-\frac{n}{2u-2\ka}s_{00}(u).
\end{equation*}
Therefore, on $\C\eta$, equation \eqref{HWT:mu_0(u).1} can be expressed as
\begin{equation*}
 p(u)\wt\mu_0(\ka-u)=\left(-1+\frac{(2u-\ka-1)(2\ka-2u-n)}{2u-\ka}\right)\mu_0(u)+\left(-1-\frac{1}{2u-\ka} \right)\sum_{\ell=1}^n\mu_\ell(u).
\end{equation*}
By definition of $\wt\mu_0(u)$, in order to obtain \eqref{HWT:mu_0(u).0}, it remains to see that 
\begin{equation}
 -1+\frac{(2u-\ka-1)(2\ka-2u-n)}{2u-\ka}=(2u-n)p_0(u)\frac{\ka-u}{u} \text{ and }-1-\frac{1}{2u-\ka}=p_0(u)\frac{\ka-u}{u}. \label{HWT:mu_0(u).2}
\end{equation}
We have
\begin{equation}
 p_0(u)=1-\frac{1}{2u-\ka}+\frac{2\ka+2}{2u-2\ka}=\frac{4u^2-2u\ka+2u}{(2u-\ka)(2u-2\ka)}=\frac{u(\ka-2u-1)}{(2u-\ka)(\ka-u)}, \label{p0u}
\end{equation}
and thus
\begin{equation*}
 p_0(u)\frac{\ka-u}{u}=\frac{\ka-2u-1}{2u-\ka}=-1-\frac{1}{2u-\ka}. 
\end{equation*}
Additionally, we have  
\begin{equation*}
 (2u-n)p_0(u)\frac{\ka-u}{u}=\frac{(2u-\ka-1/2)(\ka-2u-1)}{2u-\ka}=\frac{(2u-\ka-1)(\ka-2u-1/2)-2u+\ka}{2u-\ka},
\end{equation*}
which is equivalent to the first equality in \eqref{HWT:mu_0(u).2}. 
\end{proof}

\begin{rmk} \label{HWT:Rem.m_0(u)}
Let $g(u)$ be the formal series in $u^{-1}$ given by
\begin{equation}
g(u)=\begin{cases}
      1 &\text{ if } (\mfg_N,\mfg_N^\rho) \text{ is of type B0},\\
      \frac{1[\pm]\,c(\boldsymbol{\ell}-u)}{1-cu} &\text{ if } (\mfg_N,\mfg_N^\rho) \text{ is of type BI}.
     \end{cases}  \label{HWT:Rem.m_0(u).1}   
\end{equation}
Here we recall that if $(\mfg_N,\mfg_N^\rho)$ is of type BI, then $(\mfg_N,\mfg_N^\rho)=(\mfso_N,\mfso_p\oplus \mfso_q)$ with $p>q$, and the constant $c$ is equal to $4\,(p-q)^{-1}$. 
The constant $\boldsymbol{\ell}$ is then defined to be $p/2$ if $(\mfg_N,\mfg_N^\rho)$ is of type BI(b), and $q/2$ if $(\mfg_N,\mfg_N^\rho)$ is of type BI(a).
We claim that $p_0(u)p(u)^{-1}=g(\ka-u)g(u)^{-1}$. If $\mfg_N$ is of type B and $V$ is a highest weight $X(\mfg_N,\mcG)^{tw}$-module with the highest weight $\mu(u)$, then this claim implies that the relation \eqref{HWT:mu_0(u).0} for $\wt \mu_0(u)$ can be rewritten in a more symmetric way:
\begin{equation}
 u\cdot g(u)\,\wt\mu_0 (\ka-u)=(\ka-u)\cdot g(\ka-u)\,\wt\mu_0(u).\label{HWT:Rem.m_0(u).2}   
\end{equation}
\end{rmk}
\begin{proof}[Proof of the claim that $p_0(u)p(u)^{-1}=g(\ka-u)g(u)^{-1}$:]  If $(\mfg_N,\mfg_N^\rho)$ is of type B0 then we may identify $p$ with $N$ and $q$ with $0$. It follows that $\mathrm{tr}(\mcG(u))=(N-2q)\left(\frac{N-4u}{N-2q-4u} \right)$. This allows 
us to write $p(u)$ explicitly:
\begin{equation}
 p(u)=1-\frac{1}{2u-\ka}+\frac{\mathrm{tr}(\mcG(u))}{2u-2\ka}=\frac{(2u-\ka-1)(N-2q-4u)(2u-2\ka)+(N-2q)(N-4u)(2u-\ka)}{(N-2q-4u)(2u-\ka)(2u-2\ka)}. \label{p}
\end{equation}
We can now combine \eqref{p0u} and \eqref{p} to obtain, after simplifying, the following expression for $p_0(u)p(u)^{-1}$:
\begin{equation}
p_0(u)p(u)^{-1}= \frac{(2u-\ka+1)(\ka+1-q-2u)}{(2u-\ka-1)(\ka-1+q-2u)}. \label{p_0/p.1}
\end{equation}
If $(\mfg_N,\mfg_N^\rho)$ is of type B0 then $q=0$ and the above expression is equal to $1$, which agrees with $g(\ka-u)g(u)^{-1}$. Suppose instead that $(\mfg_N,\mfg_N^\rho)$ is 
of type BI. Then we have 
\begin{equation*}
 g(u)=[\pm]\left(\frac{p+q-4u}{p-q-4u} \right)=[\pm]\left(\frac{\ka+1-2u}{\ka+1-q-2u} \right),
\end{equation*}
which implies that $g(\ka-u)g(u)^{-1}=p_0(u)p(u)^{-1}$ as a consequence of \eqref{p_0/p.1}.
\end{proof}

Recall the generators $\F{i}{j}$ of $\mathfrak{U}\mfg_N^\rho$ defined in the statement of Proposition \ref{P:F->TY}. Since the matrix $\mcG$ is assumed 
diagonal, they are related to the generators $F_{ij}$ of $\mathfrak{U}\mfg_N$ by the expression
\begin{equation*}
\F{i}{j}=(g_{ii}+g_{jj})F_{ij} \quad \text{ for all }-n\leq i,j\leq n.
\end{equation*}

Proposition \ref{P:F->TY} allows us to identify the elements $\F{i}{j}\in \mfg_N^\rho$ with
their image in $X(\mfg_N,\mcG)^{tw}$ under the embedding $\mf{U}\mfg^\rho_N\into X(\mfg_N,\mcG)^{tw}$, and consequentially we can use the explicit form of the reflection equation \eqref{[s,s]} to compute the bracket relations $[\F{i}{j},s_{k\ell}(u)]$. We 
have
\begin{equation}
 [\F{i}{j},s_{k\ell}(v)]=(g_{ii}+g_{jj})\left(\delta_{kj}s_{i\ell}(v)-\delta_{i\ell}s_{kj}(v)-\delta_{k,-i}\theta_{ij}s_{-j,\ell}(v)+\delta_{\ell,-j}\theta_{ij}s_{k,-i}(v)\right)\label{HWT:EmbeddingBracket.2}.
\end{equation}
Observe that, by definition of $\mfg_N^\rho$ and of the elements
$\F{i}{j}$, $\mfh^\rho_N= \mathrm{span}_\C\{\F{i}{i}\,|\,1\leq i\leq n\}$ is a Cartan subalgebra of $\mfg_N^\rho$, which is actually just the Cartan subalgebra $\mfh_N$ of $\mfg_N$. Given $1\leq i\leq n$, let $\eps_i\in \mfh_N^*$ be defined by 
$\eps_i(F_{kk})=\delta_{ik}$ for all $1\leq k\leq n$. In addition, define auxiliary elements $\alpha_{k,\ell}\in \mfh_N^*$ by
$\alpha_{k,\ell}=\mathrm{sign}(k)\eps_{|k|}-\mathrm{sign}(\ell)\eps_{|\ell|}$  for all $-n\leq k,\ell \leq n$, where $\eps_0$ is the zero functional. Then from equation \eqref{HWT:EmbeddingBracket.2} we obtain
\begin{equation}
 [\F{i}{i},s_{k\ell}(v)]= 2g_{ii} \left(\delta_{ik}-\delta_{i\ell}-\delta_{i,-k}+\delta_{i,-\ell}\right)s_{k\ell}(v)=\alpha_{k,\ell}(\F{i}{i})s_{k\ell}(v) \label{HWT:EmbeddingBracket.Cartan}
\end{equation}
for all $1\leq i\leq n$ and $-n\leq k,\ell\leq n$. 

Let $\Delta^+$ be the standard set of positive roots of $\mfg_N$ for our choice of $\mfh_N$ (as in \cite{AMR}), so %
\begin{equation*}
 \Delta^+=\{\alpha_{k,\ell}: -n\leq k<\ell\leq n \text{ and } k+\ell \ge \tfrac{1}{2} \pm \tfrac{1}{2} \}.
\end{equation*}
We may define a partial order $\preceq$ on the set of weights of $\mfh_N$ by $\mu_1 \preceq\mu_2$ if and only if $\mu_2-\mu_1=\sum_{k, \ell}m_{k,\ell}\alpha_{k,\ell}$
with $m_{k,\ell}\in \Z_{\geq 0}$ and $\alpha_{k,\ell} \in \Delta^+$.


\begin{thrm}\label{HWT:Thm.HWT}
Every finite-dimensional irreducible representation $V$ of $X(\mfg_N,\mcG)^{tw}$ is a highest weight representation. Additionally, $V$ contains a unique highest weight vector $\eta$ up to
scalar multiplication. 
\end{thrm}

\begin{proof}
Define the subspace $V^0$ of $V$ by 
$$V^0=\{\eta \in V:s_{ij}(u)\eta=0 \text{ for all }-n\leq i<j\leq n\}.$$

\medskip

\noindent\textit{Step 1: $V^0$ is nonzero.}

\medskip 

Via the embedding $\mfg_N^\rho \into X(\mfg_N,\mcG)^{tw}$, we may view $V$ as a $\mfg_N^\rho$-module. Since $V$ is finite-dimensional, the $F_{ii}'^\rho$ have a mutual weight vector $\eta$. 
Let $L$ be the set of all weights of the $\mfg_N^\rho$-module $V$, so $L$ is a nonempty finite set. Therefore, 
there exists $\mu\in L$ such that $\mu+\alpha_{k,\ell}$ is not a weight for any $-n\leq k<\ell\leq n$. Then the $\mu$-weight vector $\eta$ must belong to $V^0$. Indeed, suppose there exists
$-n\leq k<\ell\leq n$ such that $s_{k\ell}(u)\eta\neq 0$. Then from \eqref{HWT:EmbeddingBracket.Cartan} we obtain: 
\begin{equation*}
 \F{i}{i}\left(s_{k\ell}(v)\eta\right)=\left(\alpha_{k,\ell}+\mu\right)(\F{i}{i})s_{k\ell}(v)\eta
\end{equation*}
for all $1\leq i\leq n$. This contradicts the maximality of $\mu$, and so we must have $\eta\in V^0$. Therefore $V^0$ is nonzero.

\medskip 

\noindent\textit{Step 2: the subspace $V^0$ is preserved by the operators $s_{ii}(u)$ for all $i\in\mathcal{I}_N$. }

\medskip

We will consider separately the cases when $N$ is even and when $N$ is odd. 

\medskip

\noindent\textit{Step 2.1:}  $N=2n$. \nopagebreak

\medskip

By definition of $V^0$, we must show that $s_{k\ell}(u)s_{ii}(v)\equiv 0$ for all $k<\ell $ and $1\leq i\leq n$, where $\equiv$ denotes equality of operators on $V^0$.

\medskip 

\noindent\textit{Claim: It suffices to show that for all  $0<k<\ell$ and $i,j>0$, $s_{k\ell}(u)s_{ii}(v)$, $s_{-k,\ell}(u)s_{ii}(v)$ and $s_{-j,j}(u)s_{ii}(v)$ all are equal to the zero 
operator on the subspace $V^0$.}

\medskip

This claim follows from the symmetry relation \eqref{s=s}.

\medskip

\noindent\textit{Step 2.1.1:} $s_{k\ell}(u)s_{ii}(v)\equiv 0$ for all $0<k<\ell$ and $i>0$. 

\medskip

Assume first that $k<i$. Then it is immediate from \eqref{[s,s]} (using $s_{k\ell}(u)s_{ii}(v)\equiv [s_{k\ell}(u),s_{ii}(v)]$) 
that $s_{k\ell}(u)s_{ii}(v)\equiv 0$ unless $i=\ell$. If $i=\ell$, we obtain
\begin{equation*}
 s_{k\ell}(u)s_{\ell\ell}(v)\equiv\frac{1}{u+v}\sum_{a=\ell}^n s_{ka}(u)s_{a\ell}(v),
\end{equation*}
and thus for $a\geq \ell>k$, \eqref{[s,s]} yields 
\begin{equation*}
 s_{ka}(u)s_{a\ell}(v) \equiv  [s_{ka}(u),s_{a\ell}(v)] \equiv \frac{1}{u+v}\sum_{b=\ell}^n s_{kb}(u)s_{b\ell}(v)\equiv s_{k\ell}(u)s_{\ell\ell}(v).
\end{equation*}
Therefore we obtain the relation $\left(1-\frac{n-\ell+1}{u+v}\right)s_{k\ell}(u)s_{\ell\ell}(v)\equiv 0$ and so we must have $s_{k\ell}(u)s_{\ell\ell}(v)\equiv 0$, as desired. 

If instead $k\geq i$, then we write $[s_{k\ell}(u),s_{ii}(v)]=-[s_{ii}(v),s_{k\ell}(u)]$. Relation \eqref{[s,s]} then gives
\begin{align*}
 [s_{ii}(v),s_{k\ell}(u)]&\equiv \frac{\delta_{ik}}{v+u}\sum_{a=\ell}^n s_{ia}(v)s_{a\ell}(u) -\frac{1}{v^2-u^2}\sum_{a=\ell}^n \left(s_{ka}(v)s_{a\ell}(u)-s_{ka}(u)s_{a\ell}(v) \right) \\
                         &\equiv \left(\frac{\delta_{ik}}{v+u}-\frac{1}{v^2-u^2}\right)\sum_{a=\ell}^n s_{ka}(v)s_{a\ell}(u)+\frac{1}{v^2-u^2} \sum_{a=\ell}^n s_{ka}(u)s_{a\ell}(v).
 \end{align*}
From the above proof that $s_{k\ell}(u)s_{\ell\ell}(v)\equiv 0$, we see that the right-hand side of the previous line is $\equiv 0$. This completes the proof that $s_{k\ell}(u)s_{ii}(v)\equiv 0$ for all $0<k<\ell$ and $i>0$. 

\medskip

\noindent\textit{Step 2.1.2:} $s_{-k,\ell}(u)s_{ii}(v)\equiv 0$ for all $i>0$ and $\ell>k>0$. 

\medskip

This is an immediate consequence of relation \eqref{[s,s]} unless $i=\ell$ or $i=k$. The case $i=\ell$ is similar to Step 2.1.1, so we concentrate on the case $i=k$. By \eqref{[s,s]} we have 
\begin{equation}
 s_{-k,\ell}(u)s_{kk}(v)\equiv-\frac{1}{u-v-\ka}\sum_{a=k}^n s_{-a,\ell}(u)s_{ak}(v)+\frac{1}{(u+v)(u-v-\ka)}\sum_{a=k}^n s_{-\ell,a}(u)s_{ak}(v),\label{HWT:HWT.01}
\end{equation}
Since $s_{-a,\ell}(u)s_{ak}(v) \equiv [ s_{-a,\ell}(u),s_{ak}(v)]$, for $a\geq k$, we have 
\begin{align*}
 s_{-a,\ell}(u)s_{ak}(v)&\equiv \frac{\delta_{a\ell}}{u+v}\sum_{b=k}^ns_{-a,b}(u)s_{bk}(v) -\frac{1}{u-v-\ka}\left(\sum_{b=k}^n s_{-b,\ell}(u)s_{bk}(v)-\frac{1}{u+v}\sum_{b=k}^n s_{-\ell,b}(u)s_{bk}(v)\right)\\
                        &\equiv \frac{\delta_{a\ell}}{u+v}\sum_{b=k}^ns_{-\ell,b}(u)s_{bk}(v)+s_{-k,\ell}(u)s_{kk}(v).
\end{align*}
Substituting this result back into \eqref{HWT:HWT.01} we obtain 
\begin{equation*}
 s_{-k,\ell}(u)s_{kk}(v)\equiv-\frac{n-k+1}{u-v-\ka}s_{-k,\ell}(u) s_{kk}(v).
\end{equation*}
and so we must have $s_{-k,\ell}(u)s_{kk}(v)\equiv 0$ whenever $0<k<\ell$. 

 \medskip
 
\noindent\textit{Step 2.1.3:} $s_{-j,j}(u)s_{ii}(v)\equiv 0$ for all $i,j>0$.   \nopagebreak

\medskip

To begin, it is an immediate consequence of \eqref{[s,s]} that $s_{-j,j}(u)s_{ii}(v)\equiv 0$ unless
$i=j$, so without loss of generality we may assume $i=j$. We have: 
\begin{align}
 [s_{-i,i}(u),s_{ii}(v)] \equiv& \left(\frac{1}{u+v}+\frac{1}{(u+v)(u-v-\ka)}\right)\sum_{a=i}^n s_{-i,a}(u)s_{ai}(v)-\frac{1}{u-v-\ka}\sum_{a=i}^n s_{-a,i}(u)s_{ai}(v).\label{HWT:HWT.2}
\end{align}
Let us compute $s_{-a,i}(u)s_{ai}(v)$ for $a>i$. From \eqref{[s,s]} we see that 
\begin{align*}
 s_{-a,i}(u)s_{ai}(v)&\equiv-\frac{1}{u-v-\ka}\sum_{b=i}^n s_{-b,i}(u)s_{bi}(v)+\frac{1}{(u+v)(u-v-\ka)}\sum_{b=i}^n s_{-i,b}(u)s_{bi}(v)\\
                      &\equiv [s_{-i,i}(u),s_{ii}(v)]-\frac{1}{u+v}\sum_{a=i}^ns_{-i,a}(u)s_{ai}(v) \text{ by } \eqref{HWT:HWT.2}.
\end{align*}
Substituting this result back into relation \eqref{HWT:HWT.2}, we get
\begin{equation*}
 \left(1+\frac{n-i+1}{u-v-\ka} \right)[s_{-i,i}(u),s_{ii}(v)] \equiv \left(\frac{1}{u+v}+\frac{n-i+1}{(u-v-\ka)(u+v)}\right)\sum_{a=i}^n s_{-i,a}(u)s_{ai}(v),                      
\end{equation*}
from which we obtain  
\begin{equation}
 [s_{-i,i}(u),s_{ii}(v)]\equiv\frac{1}{u+v}\sum_{a=i}^n s_{-i,a}(u)s_{ai}(v). \label{HWT:HWT.4}
\end{equation}
By \eqref{[s,s]}, we have that for all $a>i$
\begin{equation*} 
 s_{-i,a}(u)s_{ai}(v)=\frac{1}{u+v}\sum_{b=i}^n s_{-i,b}(u)s_{bi}(v).
\end{equation*}
Substituting this into \eqref{HWT:HWT.4} leads to $[s_{-i,i}(u),s_{ii}(v)]\equiv 0$. This completes the proof of Step 2 when $N=2n$.

\medskip

\noindent\textit{Step 2.2:} $N=2n+1$.   \nopagebreak

\medskip

The argument is essentially the same in this case. By the symmetry relation \eqref{s=s}, it suffices to show that $s_{k\ell}(u)s_{ii}(v)$, $s_{-k,\ell}(u)s_{ii}(v)$, 
$s_{-j,j}(u)s_{ii}(v)$ and $s_{0j}(u)s_{ii}(v)$  are all equal to the zero operator on $V^0$, where $0<k<\ell$, $j>0$ and $i\geq 0$. 

\medskip

\noindent\textit{Step 2.2.1:} $s_{k\ell}(u)s_{ii}(v)$, $s_{-k,\ell}(u)s_{ii}(v)$ and $s_{-j,j}(u)s_{ii}(v)$  all $\equiv 0$ when $0<k<\ell$, $j>0$ and $i\geq 0$.

\medskip

The same arguments as those given for the $N=2n$ case show that 
\begin{equation*}
s_{k\ell}(u)s_{ii}(v)\equiv s_{-k,\ell}(u)s_{ii}(v)\equiv s_{-j,j}(u)s_{ii}(v)\equiv0 
\end{equation*}
whenever $i,j>0$ and $\ell>k>0$. Moreover, given the same restrictions on $j,\ell$ and $k$, the reflection equation \eqref{[s,s]} immediately yields 
\begin{equation*}
s_{-k,\ell}(u)s_{00}(v)\equiv s_{-j,j}(u)s_{00}(v)\equiv0.
\end{equation*}
Moreover if $0<k<\ell$, then $s_{k\ell}(u)s_{00}(v)\equiv-[s_{00}(v),s_{k\ell}(u)]$ and \eqref{[s,s]} yields 
\begin{equation*}
[s_{00}(v),s_{k\ell}(u)]\equiv-\frac{1}{v^2-u^2}\sum_{a=\ell}^n(s_{ka}(v)s_{a\ell}(u)- s_{ka}(u)s_{a\ell}(v)).
\end{equation*}
By the same argument as in Step 2.1.1, the right-hand side vanishes.

\medskip

\noindent\textit{Step 2.2.2:} $s_{0j}(u)s_{ii}(v)\equiv 0$ for all $j>0$ and $i\geq 0$.  \nopagebreak

\medskip

Assume first $i>0$. Then \eqref{[s,s]} implies that
$s_{0j}(u)s_{ii}(v)\equiv 0$ unless $i=j$. Moreover, the proof that $s_{0j}(u)s_{jj}(v)\equiv 0$ proceeds identically to the proof that 
$s_{-k,\ell}(u)s_{\ell\ell}(v)\equiv 0$ for all $\ell>k>0$. 

To prove that $s_{0j}(u)s_{00}(v)\equiv 0$ for all $j>0$, note first that 
$s_{0j}(u)s_{00}(v)\equiv -[s_{00}(v),s_{0j}(u)]$, and by \eqref{[s,s]} we have 
\begin{equation}
 [s_{00}(v),s_{0j}(u)]\equiv \left(1-\frac{1}{v-u}+\frac{1}{v-u-\ka} \right)B(v,u)+\frac{1}{v-u}B(u,v)-\frac{1}{v-u-\ka}\sum_{a=j}^ns_{-a,0}(v)s_{aj}(u), \label{HWT:HWT.02}
\end{equation}
where $B(u,v)=\frac{1}{u+v}\sum_{a=j}^n s_{0a}(u)s_{aj}(v)$. However, since $s_{0j}(u)s_{jj}(v)\equiv 0$ by the previous step, \eqref{[s,s]} yields
\begin{equation*}
 0\equiv s_{0j}(u)s_{jj}(v)\equiv \frac{1}{u+v}\sum_{a=j}^ns_{0a}(u)s_{aj}(v)=B(u,v),
\end{equation*}
and the symmetry relation \eqref{s=s} gives 
\begin{align*}
 \sum_{a=j}^ns_{-a,0}(v)s_{aj}(u)&\equiv p(v)(\ka-v+u)B(\ka-v,u)\pm \frac{v+u}{2v-\ka}B(v,u)\equiv 0.
\end{align*}
Therefore, by \eqref{HWT:HWT.02} we have $s_{0j}(u)s_{jj}(v)\equiv 0$ for all $j>0$. 
 
\medskip 

\noindent \textit{Step 3:} Viewed as operators on $V^0$, $s_{ii}(u)$ and $s_{jj}(v)$ commute for all $i,j\in \mathcal{I}_N$.

\medskip

Again, we will treat the cases $N=2n$ and $N=2n+1$ separately. 

\medskip 

\noindent \textit{Step 3.1:} $N=2n$.  \nopagebreak

\medskip

Let us define the operator $A_{ij}(u,v)$ on $V^0$ by 
\begin{equation}
 A_{ij}(u,v)=s_{ij}(u)s_{ji}(v)-s_{ij}(v)s_{ji}(u). \label{HWT:A_ij(u,v)}
\end{equation}
As consequence of \eqref{[s,s]} we have:
\begin{equation}
 A_{ii}(u,v)\equiv \frac{1}{u+v}\sum_{a=i}^n A_{ia}(u,v). \label{HWT:HWT.5}
\end{equation}
On the other hand, for $0< i<j$ we have $s_{ji}(v)s_{ij}(u)\equiv s_{ji}(u)s_{ij}(v)\equiv 0$, so can rewrite $A_{ij}(u,v)$ as
\begin{equation*}
 A_{ij}(u,v)\equiv [s_{ij}(u),s_{ji}(v)]+[s_{ji}(u),s_{ij}(v)]. 
\end{equation*}
Using \eqref{[s,s]} to compute $[s_{ij}(u),s_{ji}(v)]$ and $[s_{ji}(u),s_{ij}(v)]$, we get: 
\begin{equation}
 A_{ij}(u,v)\equiv \frac{1}{u-v}\left([s_{ii}(u),s_{jj}(v)]+[s_{jj}(u),s_{ii}(v)] \right)+\frac{1}{u+v}\left(\sum_{a=i}^n A_{ia}(u,v) +\sum_{a=j}^n A_{ja}(u,v) \right). \label{HWT:HWT.6}
\end{equation}
We apply \eqref{[s,s]} again to compute 
\begin{equation}
 [s_{ii}(u),s_{jj}(v)]\equiv -\frac{1}{u^2-v^2}\sum_{a=j}^n A_{ja}(u,v), \label{HWT:HWT.7}
\end{equation}
from which it follows that $[s_{ii}(u),s_{jj}(v)]+[s_{jj}(u),s_{ii}(v)]\equiv 0$. Combining this with \eqref{HWT:HWT.5}, equation \eqref{HWT:HWT.6} can be rewritten as
\begin{equation}
 A_{ij}(u,v)\equiv A_{ii}(u,v)+A_{jj}(u,v). 
\end{equation}
Taking the sum  as $j$ goes from $i+1$ to $n$ and adding $A_{ii}(u,v)$ to both sides we arrive at the relation
\begin{equation*}
 \sum_{j=i}^n A_{ij}(u,v)\equiv (n-i+1)A_{ii}(u,v)+\sum_{j=i+1}^nA_{jj}(u,v).
\end{equation*}
However, by \eqref{HWT:HWT.5}, the left hand side is equivalent to $(u+v)A_{ii}(u,v)$, so we may rewrite the above as 
\begin{equation*}
 (u+v-n+i-1)A_{ii}(u,v)\equiv \sum_{j=i+1}^nA_{jj}(u,v)
\end{equation*}
A simple downward induction on $i$ then proves that $A_{ii}(u,v)\equiv 0$ for all $i \in \mathcal{I}_N$. 

Since $A_{ii}(u,v)=[s_{ii}(u),s_{ii}(v)]$, this proves that $s_{ii}(u)$ and $s_{ii}(v)$
commute for all $i \in \mathcal{I}_N$. Moreover, combining equations \eqref{HWT:HWT.7} and \eqref{HWT:HWT.5}, we have that
\begin{equation}
 [s_{ii}(u),s_{jj}(v)]\equiv -\frac{1}{u^2-v^2}\sum_{a=j}^n A_{ja}(u,v)\equiv -\frac{1}{u-v}A_{jj}(u,v)\equiv 0 \label{HWT:HWT.9}
\end{equation}
for all $j>i>0$. 

\medskip 

\noindent \textit{Step 3.2:} $N=2n+1$.  \nopagebreak

\medskip

The arguments from Step 3.1 show that $[s_{ii}(u),s_{jj}(v)]\equiv 0$ whenever $1\leq i,j\leq n$, so it suffices to show that $[s_{00}(u),s_{jj}(v)]\equiv 0$ for all $j\geq 0$. Suppose first
that $j>0$. Then by \eqref{[s,s]} and \eqref{HWT:HWT.9}, we have
\eq{ \label{V0:[00,jj]=0}
[s_{00}(u),s_{jj}(v)]\equiv  -\frac{1}{u^2-v^2}\sum_{a=j}^n A_{ja}(u,v)\equiv -\frac{1}{u-v}A_{jj}(u,v)\equiv 0.
}
Hence, it remains to see $[s_{00}(u),s_{00}(v)]\equiv 0$. The same calculations as those done to obtain \eqref{HWT:HWT.6} give \begin{equation}
A_{0j}(u,v)\equiv \frac{1}{u+v}\sum_{a = 0}^n A_{0a}(u,v) \text{ for any $j>0$}. \label{s00s001}
\end{equation} Summing this expression over $1\leq j\leq n$ and adding $A_{00}(u,v)$ to both sides we obtain the relation 
\begin{equation}
 A_{00}(u,v)\equiv (u+v-n)A_{0j}(u,v) \text{ for any $j>0$}. \label{s00s002}
\end{equation}

It follows from \eqref{[s,s]} that
\begin{align}
 \left(1-\frac{1}{u-v}\right. & \left. +\frac{1}{u+v-\ka}-\frac{1}{(u-v)(u+v-\ka)}\right)A_{00}(u,v) \equiv \left(1-\frac{1}{u-v}+\frac{1}{u-v-\ka}\right)\frac{1}{u+v-n}A_{00}(u,v) \nonumber \\
    &-\frac{1}{u-v-\ka}\sum_{a=0}^n(s_{-a,0}(u)s_{a0}(v)-s_{0a}(v)s_{0,-a}(u))-\frac{1}{(u-v-\ka)(u+v-\ka)}\sum_{a=-n}^0[s_{aa}(u),s_{00}(v)]. \label{s00s003}
    \end{align}
Since $[s_{aa}(u),s_{00}(v)]\equiv 0$ for any $a>0$, the symmetry relation implies that 
\begin{equation}\label{s00s004}
 [s_{-a,-a}(u),s_{00}(v)]\equiv -\frac{1}{2u-2\ka}\sum_{b=0}^n[s_{-b,-b}(u),s_{00}(v)]. 
\end{equation}
Taking the sum of both sides as $a$ goes from $1$ to $n$ and adding $A_{00}(u,v)$ we obtain 
\begin{equation}
 \left(1+\frac{n}{2u-2\ka}\right)\sum_{b=0}^n[s_{-b,-b}(u),s_{00}(v)]\equiv A_{00}(u,v) \text{ and } (2\ka-2u-n)[s_{-a,-a}(u),s_{00}(v)]\equiv A_{00}(u,v) \label{s00s005}
\end{equation}
for any $a>0$. 

On the other hand, the explicit form of the defining reflection equation \eqref{[s,s]} implies that
\begin{align*}
 [s_{-a,-a}(u),s_{00}(v)] \equiv {} & -\frac{1}{u^2-v^2}\sum_{b=0}^nA_{0b}(u,v)+\frac{1}{(u-v)(u+v-\ka)}A_{0a}(u,v)\\
                         &-\frac{1}{u+v-\ka}(s_{-a,0}(u)s_{a0}(v)-s_{0a}(v)s_{0,-a}(u)).
\end{align*}
Multiplying both sides by $(u+v-n)$ and appealing to \eqref{s00s001},\eqref{s00s002} and \eqref{s00s005} we obtain 
\begin{equation*}
 \left(\frac{u+v-n}{2\ka-2u-n}+\frac{1}{u-v}-\frac{1}{(u-v)(u+v-\ka)}\right)A_{00}(u,v)\equiv -\frac{u+v-n}{u+v-\ka}(s_{-a,0}(u)s_{a0}(v)-s_{0a}(v)s_{0,-a}(u))
\end{equation*}
for any $a>0$. Taking the sum of both sides as $a$ goes from $1$ to $n$, adding $-\tfrac{u+v-n}{u+v-\ka}A_{00}(u,v)$, and then multiplying both sides by $\tfrac{u+v-\kappa}{u+v-n}$  we get 
\begin{align}
 \left(\frac{n(u+v-\ka)}{2\ka-2u-n}+\frac{n(u+v-\ka)}{(u-v)(u+v-n)} \right. & \left. -\frac{n}{(u-v)(u+v-n)}-1\right) A_{00}(u,v)\nonumber\\
 & \equiv -\sum_{a=0}^n(s_{-a,0}(u)s_{a0}(v)-s_{0a}(v)s_{0,-a}(u)). \label{s00s006}
\end{align}
Substituting \eqref{s00s006} and \eqref{s00s005} into \eqref{s00s003}, we obtain a relation of the form $f(u,v)A_{00}(u,v)\equiv0$, where $f(u,v)=1+\alpha(u,v)$ with $\alpha(u,v) \in u^{-1} \C[v][[u^{-1}]]$. This implies that we must have $A_{00}(u,v)\equiv 0$. 

\medskip 

\noindent \textit{Step 4:} $V$ is a highest weight representation.

\medskip

By Step 2, for all $r\geq 0$ and each $0\leq i\leq n$, $s_{ii}^{(r)}$ can be viewed as a linear endomorphism $V^0\to V^0$. By Step 3 these linear endomorphisms of $V^0$ all pairwise commute, 
and so they have a common eigenvector $\eta\in V^0$. Denote the eigenvalue of $s_{ii}^{(r)}$ corresponding to the eigenvector $\eta$ by $\mu_{i}^{(r)}$, where if $r=0$ then  $\mu_{i}^{(0)}=g_{ii}$. Set $ \mu_i(u)=g_{ii}+\sum_{r=1}^{\infty}\mu_i^{(r)}u^{-r}$. Then the submodule $X(\mfg_N,\mcG)^{tw}\eta$ is a highest weight representation, with highest weight vector $\eta$ and highest weight $(\mu_{i}(u))_{i\in\mcI_N}$. Moreover, since $V$ is irreducible, we must 
have $V=X(\mfg_N,\mcG)^{tw} \eta$. This proves that every finite-dimensional irreducible representation of $X(\mfg_N,\mcG)^{tw}$ is a highest weight representation. 

\medskip

\noindent \textit{Step 5:} Uniqueness of the highest weight vector.

\medskip

Let $\mu$ be the weight of the $\mfg_N^{\rho}$-module $V$ 
corresponding to $\eta$, i.e. $\mu(F_{ii}'^\rho)=\mu_i^{(1)}-\bar g_{ii}$ for all $1\leq i\leq n$. Since the central elements $w_i$, $i=2,4,6,\ldots$, must act
by scalar multiplication, Corollary \ref{X:PBW} implies that $V$ is spanned by elements of the form: 
\begin{equation}
 s_{j_1,i_1}^{(r_1)}\cdots s_{j_m,i_m}^{(r_m)}\eta \label{HWT:HWT.8}
\end{equation} 
with $j_a>i_a$, $j_a+i_a \ge \tfrac{1}{2} \pm \tfrac{1}{2}$, $r_a\geq 1$ for all $1\leq a\leq m$, and $m\geq 0$. It follows by \eqref{HWT:EmbeddingBracket.Cartan} that $v\in V$ can only belong to the $\mu$-weight space $V_\mu$ if 
$v\in \C\cdot \eta$, thus $V_\mu$ is one dimensional, and moreover any other weight of $V$ is $\prec \mu$.
\end{proof}


We now determine how the coefficients of the distinguished central series $w(u)$ (see \eqref{w(u)}) act on any highest weight representation of $X(\mfg_N,\mcG)^{tw}$. 
\begin{prop}\label{HWT:Prop.central_action}
 Let $V$ be a highest weight representation of $X(\mfg_N,\mcG)^{tw}$ with the highest weight $\mu(u)$.  Then the coefficients of the even series $w(u)$ act on $V$ 
 by scalar operators determined by: 
 $$w(u)|_V=\mu_n(-u)\mu_n(u).$$
\end{prop}

\begin{proof}
 Let $\eta\in V$ be the highest weight vector. Since the coefficients of $w(u)$ belong to the center $ZX(\mfg_N,\mcG)^{tw}$ and $V$ is spanned by elements of the form given in \eqref{HWT:HWT.8}, the action of the $2i$-th coefficient
 $w_{2i}$ of $w(u)$ on $V$ is completely determined by its action on $\eta$. By \eqref{w(u)} we have the relation $S(u)S(-u)=w(u)\cdot I$. Applying the $(n,n)^{th}$ entry of both sides to
 the highest weight vector $\eta$ we obtain
\begin{equation*}
 w(u)\eta=\sum_{\ell=-n}^n s_{n\ell}(u)s_{\ell n}(-u)\eta =s_{nn}(u)s_{nn}(-u)\eta=\mu_n(-u)\mu_n(u)\eta. \qedhere
\end{equation*}
\end{proof}

\medskip

\begin{defn}\label{HWT:Defn.Verma_Module}
Let $\mu(u)=(\mu_i(u))_{i\in\mcI_N}$ be any tuple of formal series such that $\mu_i(u) \in g_{ii}+u^{-1}\C[[u^{-1}]]$
and $\wt \mu_0(u)$ satisfies \eqref{HWT:mu_0(u).0} if $\mfg_N$ is of type B. We define the Verma module $M(\mu(u))$ over $X(\mfg_N,\mcG)^{tw}$  as the quotient
$$M(\mu(u))=X(\mfg_N,\mcG)^{tw}/J, $$
where $J$ is the left ideal in $X(\mfg_N,\mcG)^{tw}$ generated by all elements $s_{ij}^{(r)}$ with $i<j$ and $s_{kk}^{(r)}-\mu_k^{(r)}$ where $r\geq 1$ and $k\in \mathcal{I}_N$.
\end{defn}

We will soon see that, similarly to the Verma modules for $X(\mfg_N)$, some choices of $\mu(u)$ may result in $M(\mu(u))$ being trivial (see Proposition \ref{HWT:refl.Prop.2}).  If $M(\mu(u))$ is non-trivial, then it is a highest weight module with the highest weight $\mu(u)$ and the highest weight vector $1_{\mu(u)}$ equal to the image of the identity element $1\in X(\mfg_N,\mcG)^{tw}$ under the natural quotient map  $X(\mfg_N,\mcG)^{tw}\to M(\mu(u))$. As consequence of the Poincar\'{e}-Birkhoff-Witt theorem for $X(\mfg_N,\mcG)^{tw}$, $M(\mu(u))$ is spanned by elements of the form: 
\begin{equation*}
 s_{j_1,i_1}^{(r_1)}\cdots s_{j_m,i_m}^{(r_m)}1_{\mu(u)} 
\end{equation*} 
with $j_a>i_a$, $j_a + i_a \ge \tfrac{1}{2} \pm \tfrac{1}{2}$, $r_a\geq 1$ for all $1\leq a\leq m$, and $m\geq 0$. Using this fact, together with the commutator relation \eqref{HWT:EmbeddingBracket.Cartan}, one can prove the following standard proposition.

\begin{prop}\label{HWT:Prop.VM}
Suppose $\mu(u)=(\mu_i(u))_{i\in \mcI_N}$ is such that the Verma module $M(\mu(u))$ is non-trivial. Then:
\begin{enumerate}
\item If $K$ is a submodule of $M(\mu(u))$, then $K=\bigoplus_{\lambda} K_\lambda$ where \[ K_\lambda = \{v\in K:F_{ii}^{\prime \rho}v=\lambda_iv \enspace \forall \enspace 1\leq i\leq n\} = M(\mu(u))_\lambda\cap K. \]
\item If $K$ is a proper submodule of $M(\mu(u))$, then $K\subseteq \bigoplus_{\lambda\neq \mu}M(\mu(u))_\lambda$ where $\mu = (\mu_i^{(1)} - \bar{g}_{ii})_{i\in\mcI_N}$.
\item $M(\mu(u))$ admits a unique irreducible quotient $V(\mu(u))$.  
\item Any irreducible highest weight $X(\mfg_N,\mcG)^{tw}$-module with the highest weight $\mu(u)$ is isomorphic to $V(\mu(u))$.
\end{enumerate}
\end{prop}


Since $X(\mfg_N,\mcG)^{tw}$ is a subalgebra of $X(\mfg_N)$, we may view any $X(\mfg_N)$-module $L$ as a $X(\mfg_N,\mcG)^{tw}$-module by restricting the action of $X(\mfg_N)$.  Similarly, since $X(\mfg_N,\mcG)^{tw}$ is a left coideal subalgebra of $X(\mfg_N)$, the tensor product of an $X(\mfg_N)$-module $L$
and an $X(\mfg_N,\mcG)^{tw}$-module
$V$ inherits the structure of an $X(\mfg_N,\mcG)^{tw}$-module via the coproduct $\Delta$. More explicitly, for all $x\in X(\mfg_N,\mcG)^{tw}$, the action on $L\otimes V$ is given by 
$$x\cdot w\otimes v=\Delta(x)(w\otimes v) \text{ for all }w\in L \text{ and }v\in V.$$

In particular, we may take $L=L(\lambda(u))$ for some $N$-tuple $\lambda(u)=(\lambda_{-n}(u),\ldots, \lambda_{n}(u))$ satisfying \eqref{HWT:Ext.non-trivial}, and $V=V(\mu(u))$, where $\mu(u)$ is such that the Verma module $M(\mu(u))$ is non-trivial. If $L(\lambda(u))$ has the
highest weight vector $\xi$ and $V(\mu(u))$ has the highest weight vector $\eta$, then we may consider the $X(\mfg_N,\mcG)^{tw}$-modules $X(\mfg_N,\mcG)^{tw}\xi$ and $X(\mfg_N,\mcG)^{tw}(\xi\otimes\eta)$. Our 
present goal is to show that both these modules are of highest weight type, and to compute explicitly what the highest weights are. This will be achieved
in Proposition \ref{HWT:Prop.tensors} and Corollary \ref{HWT:Cor.restrictions}, however, first we need to prove a lemma concerning the extended Yangian $X(\mfg_N)$.

\medskip

In the extended twisted Yangian $X(\mfg_N)$ we have the relation $T(u)T^t(u+\ka)=z(u)I$, which immediately implies $T(u)^{-1}=z(u)^{-1}{T}^t(u+\ka)$. Let us denote
the $(i,j)^\text{th}$ entry of the inverse matrix $T(u)^{-1}$ by $t'_{ij}(u)$.

\begin{lemma}\label{HWT:Lem.restrictions.tensors}
 We have the following relations between the elements $t_{ij}(u)$ and $t'_{k\ell}(v)$ for all \mbox{$-n\leq i,j,k,\ell\leq n$:}
 \begin{align}
  [t_{ij}(u),t'_{k\ell}(v)]&=\frac{1}{u-v}\sum_{a=-n}^n\left(\delta_{kj}t_{ia}(u)t'_{a\ell}(v)-\delta_{i\ell}t'_{ka}(v)t_{aj}(u) \right)\nonumber\\
                                 &-\frac{1}{u-v-\ka}\left(\theta_{-k,j}t_{i,-k}(u)t'_{-j,\ell}(v)-\theta_{i,-\ell}t'_{k,-i}(v)t_{-\ell,j}(u) \right) \label{HWT:Inverse.1}.
 \end{align}
\end{lemma}

\begin{proof}
Multiplying both sides of  the relation $R(u-v)T_1(u)T_2(v)=T_2(v)T_1(u)R(u-v)$ by $T_2(v)^{-1}$ we obtain the equivalent relation
\begin{equation*}
T_2(v)^{-1}R(u-v)T_1(u)=T_1(u)R(u-v)T_2(v)^{-1}.
\end{equation*}
Expanding $R(u-v)$ as $R(u-v)=1-\frac{P}{u-v}+\frac{Q}{u-v-\ka}$ leads to \eqref{HWT:Inverse.1}.
\end{proof}

\medskip 


Recall that, given a tuple of series $\mu(u)$, $\wt \mu(u)$ is the corresponding tuple whose components have been defined in  \eqref{HWT:tilde_mu_i(u)}.

\begin{prop}\label{HWT:Prop.tensors}
 Let $\eta$ denote the highest weight vector of the irreducible $X(\mfg_N,\mcG)^{tw}$-module $V(\mu(u))$, and $\xi$ the highest weight vector of the irreducible $X(\mfg_N)$-module $L(\lambda(u))$.
 Then $X(\mfg_N,\mcG)^{tw}(\xi\otimes \eta)$ is a highest weight  $X(\mfg_N,\mcG)^{tw}$-module with the highest weight vector $ \xi\otimes \eta$, and the highest weight
 $\gamma(u)$ whose components are determined by the relations
 \begin{equation}
 \wt \gamma_i(u)=\wt \mu_{i}(u)\lambda_i(u-\ka/2)\lambda_{-i}(-u+\ka/2) \label{HWT:tensors}
 \end{equation}
 for all $i\in \mathcal{I}_N$.
\end{prop}

\begin{proof}
We will use the symbol $``\equiv"$ to denote equality of operators on the spaces $\C(\xi\otimes \eta)$ or $\C\xi$.
We begin by showing that $s_{ij}(u)\cdot (\xi\otimes \eta)=0$ for all $i<j$. By the symmetry relation \eqref{s=s}, it is enough to consider the cases
where $i<0<j$ or $0\leq i<j$. We have 
\begin{equation*}
\Delta(s_{ij}(u)) \equiv z(-u-\ka/2)\sum_{-n \le b\le a \le n} t_{ia}(u-\ka/2)t'_{bj}(-u-\ka/2)\otimes s_{ab}(u).
\end{equation*}
Moreover, we have $t'_{bj}(v)\xi=0$ whenever $b<j$, so we can assume $b\geq j$.
Since $i<j\leq b\leq a$, we have $t_{ia}(u-\ka/2)\xi=0$ and also $a,b>0$ since $j>0$. By Lemma \ref{HWT:Lem.restrictions.tensors} we have $t_{ia}(u)t'_{bj}(v)\equiv 0$
unless $a=b$. Therefore it suffices to show that $t_{ia}(u)t'_{aj}(v)\equiv 0$ for $i<j$ and $a\geq j>0$.  From Lemma \ref{HWT:Lem.restrictions.tensors} we arrive at 
\begin{equation*}
t_{ia}(u)t'_{aj}(v)\equiv \frac{1}{u-v}\sum_{b=j}^n t_{ib}(u)t'_{bj}(v)
\end{equation*}
for all $a\geq j$. This gives 
 \begin{equation}
  \sum_{a=j}^n t_{ia}(u)t'_{aj}(v)\equiv \frac{n-j+1}{u-v}\sum_{a=j}^n t_{ia}(u)t'_{aj}(v), 
 \end{equation}
and so $t_{ia}(u)t'_{aj}(v)=0$. This completes the proof that $\Delta(s_{ij}(u))(\xi\otimes \eta)=0$ for all  $i<j$. 

\medskip 

Next, we compute $\Delta(s_{ii}(u))(\xi\otimes \eta)$ for all $i\in \mathcal{I}_N$. Using computations similar to those above, we can show that $t_{ia}(u)t'_{bi}(v)\xi=0$ whenever $a>b$. Thus, 
\begin{equation}
\Delta(s_{ii}(u))(\xi\otimes \eta)= z(-u-\ka/2)\sum_{a=i}^nt_{ia}(u-\ka/2)t'_{ai}(-u-\ka/2)\xi\otimes s_{aa}(u)\eta=(\hat{s}_{ii}(u)\xi)\otimes \eta, \label{HWT:Rest.1V}
\end{equation}
where $\hat{s}_{ii}(u)$ is the operator defined by the formula 
\begin{equation*}
\hat{s}_{ii}(u)=z(-u-\ka/2)\sum_{a=i}^n\mu_a(u)t_{ia}(u-\ka/2)t'_{ai}(-u-\ka/2).
\end{equation*}
As a consequence of our work so far, it remains only to determine the eigenvalue $\gamma_i(u)$ of the operator $\hat{s}_{ii}(u)$ corresponding to the vector $\xi$. Define the operator
$A_i(u)$ by the formula 
\begin{equation}
 A_{i}(u)=\sum_{a=i}^nz(-u-\ka/2)t_{ia}(u-\ka/2)t'_{ai}(-u-\ka/2). \label{HWT:Rest.A_i(u)}
\end{equation}
We first show that $A_{i}(u)\xi=\mu_i^\bullet(u)\xi$ for some scalar series $\mu_i^\bullet(u)$. From Lemma \ref{HWT:Lem.restrictions.tensors} we obtain for all $a>i\geq0$: 
\begin{equation}
 t_{ia}(u-\ka/2)t'_{ai}(-u-\ka/2)\equiv \frac{1}{2u}\left(\sum_{r=i}^nt_{ir}(u-\ka/2)t'_{ri}(-u-\ka/2)-\sum_{r=a}^nt'_{ar}(-u-\ka/2)t_{ra}(u-\ka/2) \right).\label{eq0}
\end{equation}
This implies that
\begin{equation*}
 A_i(u)\equiv z\left(-u-\ka/2\right)t_{ii}\left(u-\ka/2\right)t'_{ii}\left(-u-\ka/2 \right)+\frac{n-i}{2u}A_i(u)-\frac{1}{2u}\sum_{a=i+1}^nB_a(u),
\end{equation*}
where 
$$B_a(u)= \sum_{r=a}^nz\left(-u-\ka/2\right)t'_{ar}\left(-u-\ka/2\right)t_{ra}\left(u-\ka/2\right).$$ 
Consequently, this proves that 
\begin{equation}
 \frac{2u-n+i}{2u}A_i(u)\equiv z\left(-u-\ka/2\right)t_{ii}\left(u-\ka/2\right)t'_{ii}\left(-u-\ka/2 \right)-\frac{1}{2u}\sum_{a=i+1}^nB_a(u). \label{HWT:Rest.1}
\end{equation}
Using the same method, one shows using Lemma \ref{HWT:Lem.restrictions.tensors} that 
\begin{equation}
 \frac{2u-n+i}{2u}B_i(u)\equiv z\left(-u-\ka/2\right)t_{ii}\left(u-\ka/2\right)t'_{ii}\left(-u-\ka/2 \right)-\frac{1}{2u}\sum_{a=i+1}^nA_{a}(u) \label{HWT:Rest.2}
\end{equation}
for all $i\in \mathcal{I}_N$. An easy downward induction then shows $B_i(u)\equiv A_i(u)$ for all such $i$. Substituting this result back into \eqref{HWT:Rest.1}
and using that $z(v)t'_{ii}(v)=t_{-i,-i}(v+\ka)$ for all $i$, we obtain: 
\begin{equation}
\frac{2u-n+i}{2u}A_i(u)\equiv t_{ii}\left(u-\ka/2\right)t_{-i,-i}(-u+\ka/2)-\frac{1}{2u}\sum_{a=i+1}^n A_{a}(u). \label{HWT:Rest.3}                         
\end{equation}
It follows from downward induction on $i\in \mathcal{I}_N$ that there is a tuple $\mu^\bullet(u)=(\mu_i^\bullet(u))$ such that  $A_{i}(u)\xi=\mu_i^\bullet(u)\xi$ for all $i\in \mathcal{I_N}$. Moreover, the components of $\mu^\bullet(u)$ are determined
by the relations
\begin{equation*}
 \wt \mu_i^\bullet(u)=2u\lambda_i(u-\ka/2)\lambda_{-i}(-u+\ka/2) \quad \text{ for all }i\in \mathcal{I}_N.
\end{equation*}
As $B_i(u)\equiv A_{i}(u)$ for all $i\in \mathcal{I}_N$, we may express \eqref{eq0} as 
\begin{equation*}
z(-u-\ka/2)t_{ia}(u-\ka/2)t'_{ai}(-u-\ka/2)\equiv \frac{1}{2u}\left(A_{i}(u)-A_{a}(u) \right)
\end{equation*}
for all $a>i$. This gives: 
\begin{equation*}
\hat{s}_{ii}(u)\equiv \frac{1}{2u}\sum_{a=i+1}^n\mu_{a}(u)\left(A_{i}(u)-A_{a}(u)\right)+\mu_i(u)t_{ii}(u-\ka/2)t_{-i,-i}(-u+\ka/2).
\end{equation*}
Since $A_{i}(u)\xi=\mu_i^\bullet(u)$ for each $i\in \mathcal{I}_N$, applying the above expression to $\xi$ we obtain the identity
\begin{equation}
\gamma_i(u)=\mu_{i}(u)\lambda_i(u-\ka/2)\lambda_{-i}(-u+\ka/2)+\frac{1}{2u}\sum_{a=i+1}^n\mu_a(u)\left(\mu_i^\bullet(u)-\mu_a^\bullet(u) \right). \label{gamma_i(u)}
\end{equation}
We now want to obtain the formula \eqref{HWT:tensors}. Since $\mu_i^\bullet(u)=\mfrac{1}{2u-n+i}\left(\wt \mu_i^\bullet(u)-\sum_{a\geq i+1}\mu_a^\bullet(u) \right)$, equation \eqref{gamma_i(u)} implies that 
\begin{align}
 (2u-n+i)\gamma_i(u)&=\frac{2u-n+i}{2u}\mu_i(u)\wt \mu_i^\bullet (u)+\frac{1}{2u}\sum_{j\geq i+1}\mu_j(u)\wt \mu_i^\bullet(u)\nonumber\\
                    &-\frac{1}{2u}\sum_{a,j\geq i+1} \mu_j(u)\mu_a^\bullet(u)-\frac{2u-n+i}{2u}\sum_{j\geq i+1}\mu_j(u)\mu_j^\bullet(u).\label{gamma_i(u).1}
\end{align}
A straightforward downward induction on $i\in \mathcal{I}_n$ then shows that
\begin{equation*}
 \sum_{j\geq i+1}\gamma_j(u)=\frac{1}{2u}\sum_{a,j\geq i+1} \mu_j(u)\mu_a^\bullet(u)+\frac{2u-n+i}{2u}\sum_{j\geq i+1}\mu_j(u)\mu_j^\bullet(u). 
\end{equation*}
Combining this with \eqref{gamma_i(u).1} proves that \eqref{HWT:tensors} holds for all $i\in \mathcal{I}_N$. 
\end{proof}


For each matrix $\mcG$, consider the corresponding finite sequence $(g_{ii})_{i\in \mathcal{I}_N}$. If $(g_{ii})_{i\in \mathcal{I}_N}$ is the sequence $(1,\ldots, 1)$, set $\mbfk=n$. Otherwise, let $\mbfk$ be the unique integer in $\mathcal{I}_N\setminus\{n\}$ with the property that $g_{\mbfk\mbfk}\neq g_{\mbfk+1,\mbfk+1}$. Set $\boldsymbol{\ell}=n-\mbfk$. In particular, if $\mfg_N=\mfso_{2n+1}$, this coincides with the constant $\boldsymbol{\ell}$ defined in Remark \ref{HWT:Rem.m_0(u)}.

\begin{crl}\label{HWT:Cor.restrictions}
 Assume $L(\lambda(u))$ exists, with the highest weight vector $\xi$. Then  $X(\mfg_N,\mcG)^{tw}\xi$ is a highest weight module with the highest weight $\mu(u)$ whose components are 
 determined by the relations 
 \begin{equation}
 \wt \mu_i(u)=\begin{cases}
           [\mp]2u\lambda_i(u-\ka/2)\lambda_{-i}(-u+\ka/2) &\text{ if } \quad i>\mbfk,\\
           [\mp](2\boldsymbol{\ell}-2u)\lambda_i(u-\ka/2)\lambda_{-i}(-u+\ka/2)  &\text{ if } \quad 0\leq i\leq \mbfk, 
          \end{cases} \label{HWT:Rest.FirstKind}
 \end{equation}
 if $\mcG$ is of the first kind, while 
  \begin{equation}
 \wt\mu_i(u)=\begin{cases}
              \left(\mfrac{1[\pm]\,cu}{1-cu} \right)2u\lambda_i(u-\ka/2)\lambda_{-i}(-u+\ka/2) &\text{ if } \quad i>\mbfk,\\[.5em]
              \left(\mfrac{1[\pm]\,c(\boldsymbol{\ell}-u)}{1-cu}\right)2u\lambda_i(u-\ka/2)\lambda_{-i}(-u+\ka/2)  &\text{ if } \quad 0\leq i\leq \mbfk.
             \end{cases} \label{HWT:Rest.SecondKind}
 \end{equation}
  if $\mcG$ is of the second kind.
\end{crl}

\begin{proof}
 By \eqref{RE} and Lemma \ref{L:K-RE-Sym}, the assignment $S(u)\mapsto \mcG(u)$ defines a one-dimensional representation of $X(\mfg_N,\mcG)^{tw}$, which we shall denote by
 $V(\mcG)$. As $X(\mfg_N,\mcG)^{tw}$-modules, $L(\lambda(u))$ and $L(\lambda(u))\otimes V(\mcG)$ are isomorphic (see \eqref{S=TGT} and \eqref{TX-cop}). Therefore, by Proposition 
 \ref{HWT:Prop.tensors} it suffices to observe that for each $i\in \mathcal{I}_N$ the following relations hold:
 \begin{equation*}
(2u-n+i)g_{ii}+\sum_{j\geq i+1}g_{jj}=
      \begin{cases}
           [\mp]2u &\text{ if } \quad i>\mbfk,\\
           [\mp](2\boldsymbol{\ell}-2u)  &\text{ if } \quad 0\leq i\leq \mbfk, 
          \end{cases}  
 \end{equation*}
 if $\mcG$ is of the first kind, while 
  \begin{equation*}
(2u-n+i)g_{ii}(u)+\sum_{j\geq i+1}g_{jj}(u)=
      \begin{cases}
           2u\left(\mfrac{1[\pm]\,cu}{1-cu} \right) &\text{ if } \quad i>\mbfk,\\[0.5em]
           2u\left(\mfrac{1[\pm]\,c(\boldsymbol{\ell}-u)}{1-cu}\right) &\text{ if } \quad 0\leq i\leq \mbfk,
          \end{cases}  
 \end{equation*}
 if $\mcG$ is of the second kind. 
\end{proof}



\subsection{Producing representations of lower rank twisted Yangians}\label{Subsection:Induction}

Let $\mfg_{N-2}$ denote the rank $n-1$ subalgebra of $\mfg_N$ which is of the same Dynkin type, that is, $\mfg_{N-2}=\mfso_{2n-1}$ if $\mfg_N=\mfso_{2n+1}$, $\mfg_{N-2}=\mfso_{2n-2}$ if $\mfg_N=\mfso_{2n}$ and $\mfg_{N-2}=\mfsp_{2n-2}$ if $\mfg_N=\mfsp_{2n}$.
Additionally, let $\mcG^\prime$ be the $(N-2)\times (N-2)$ matrix obtained from $\mcG$ by deleting the outermost rows and columns, i.e.  $\mcG ^\prime=\sum_{i,j=-n+1}^{n-1}g_{ij}E_{ij}$, and let $\kappa' = \kappa-1$. We will denote the standard generators of $\wt X(\mfg_{N-2},\mcG^\prime)^{tw}$ and $X(\mfg_{N-2},\mcG ^\prime)^{tw}$ by $\tilde{s}^\prime_{ij}(u)$ and $s^\prime_{ij}(u)$, respectively, where $-n+1\leq i,j\leq n-1$.

\medskip 

As a consequence of the summations which appear in the expansion \eqref{[s,s]}
of the defining reflection equation, there is no natural way of viewing $X(\mfg_{N-2},\mcG ^\prime)^{tw}$ as a subalgebra of $X(\mfg_N,\mcG)^{tw}$.
Our present goal is to show that, despite this fact, there is a systematic way of constructing a $X(\mfg_{N-2},\mcG ^\prime)^{tw}$
highest weight module from any $X(\mfg_N,\mcG)^{tw}$ highest weight module.

\medskip 

For the remainder of this subsection we fix an $X(\mfg_N,\mcG)^{tw}$-module $V$. 
Define the subspace $V_+$ of $V$ as: 
\[
V_+=\{ w \in V : \enspace  s_{kn}(u)\,w =0 \qu \text{for } k<n\}.
\]
Note that by the symmetry relation  \eqref{s=s}, if $w \in V_+$ then we also have $s_{-n,\ell}(u)w  =0$ for all $\ell>-n$. In addition, 
if $V$ contains a highest weight vector $\eta$, then $\eta$ belows to $V_+$, so in particular if $V$ is a highest weight module then $V_+$ is nonempty. 
Define, for all $-n+1\leq i,j\leq n-1$, the elements $s_{ij}^\circ(u)\in X(\mfg_N,\mcG)^{tw}$~by:
\begin{equation*}
s_{ij}^\circ(u)=s_{ij}(u+1/2)+\frac{\delta_{ij}}{2u}s_{nn}(u+1/2).
\end{equation*}
\begin{lemma}\label{CT:Lemma.induction}
 $V_+$ is stable under the action of all operators $s_{ij}^\circ(u)$ with $-n+1\leq i,j\leq n-1$. Moreover, the assignment $\tilde{s}^\prime_{ij}(u)\mapsto s_{ij}^\circ (u)$ defines a representation of $\wt X(\mfg_{N-2},\mcG ^\prime)^{tw}$ in the space $V_+$.
\end{lemma}

\begin{proof}
\noindent\textit{Step 1:} 

\medskip

 Let us begin by showing that $V_+$ is stable under the action of all operators $s_{ij}(u)$ with $-n+1\leq i,j\leq n-1$. As usual, we use $``\equiv"$ to denote equality of operators on the space
 $V_+$. Let $i,j$ be such that $-n+1\leq i,j\leq n-1$. We must show $s_{kn}(u)s_{ij}(v)\equiv 0$ for all $k<n$. Since $s_{kn}(u)s_{ij}(v)\equiv -[s_{ij}(v),s_{kn}(u)]$, 
 it is enough to show $[s_{ij}(v),s_{kn}(u)]\equiv 0$. By \eqref{[s,s]}, 
 \begin{align*}
  [s_{ij}(v),s_{kn}(u)]&\equiv \frac{\delta_{kj}}{v+u}s_{in}(v)s_{nn}(u)-\frac{\delta_{ij}}{v^2-u^2}\left(s_{kn}(v)s_{nn}(u)-s_{kn}(u)s_{nn}(v) \right)\\
                       &-\frac{\delta_{k,-i}}{v-u-\ka}\theta_{i,-n}s_{-n,j}(v)s_{nn}(u) +\frac{\delta_{k,-i}}{(v+u)(v-u-\ka)}\theta_{i,-j}s_{-j,n}(v)s_{nn}(u).
 \end{align*}
As a consequence of the above and the symmetry relation \eqref{s=s}, it remains only to see $s_{\ell n}(v)s_{nn}(u)\equiv 0$ for any $-n+1\leq\ell\leq n-1$. Using the expansion \eqref{[s,s]}, we compute:
\begin{equation}
 s_{\ell n}(v)s_{nn}(u)\equiv \frac{1}{v+u}s_{\ell n}(v)s_{nn}(u). 
\end{equation}
Therefore, $s_{\ell n}(v)s_{nn}(u)\equiv 0$ for all $-n+1\leq \ell\leq n-1$. This completes the proof that $V_+$ is stable under the action of all $s_{ij}(u)$ with 
$-n+1\leq i,j\leq n-1$. Moreover, it shows that $V_+$ is stable under the action of the operator $s_{nn}(u)$. Thus, by definition $V_+$ is also stable under the action of all operators
$s_{ij}^\circ (u)$.

\medskip

\noindent\textit{Step 2:}

\medskip

Let us now turn to proving the second statement of the lemma. 
As a consequence of the first part of the proof, we may view the operators $s_{ij}^\circ(u)$  with $-n+1\leq i,j\leq n-1$ as elements in $\End(V_+)[[u^{-1}]]$. We wish to show
the corresponding map $\wt X(\mfg_{N-2},\mcG ^\prime)^{tw}\to \End(V_+)$ is a homomorphism of algebras. 
\medskip 
First observe~that 
\begin{equation}
 [s_{-n,-n}(u),s_{nn}(v)]\equiv 0, \qu [s_{nn}(u),s_{nn}(v)]\equiv 0 \qu\text{and}\qu [s_{nn}(u),s_{ij}(v)]\equiv 0 \label{CT:induction.5}
\end{equation}
for all $-n+1\leq i,j\leq n-1$. To see this, note that by \eqref{[s,s]} we have 
\begin{equation*}
 [s_{nn}(u),s_{nn}(v)]\equiv \left(\frac{1}{u-v}+\frac{1}{u+v}-\frac{1}{u^2-v^2} \right)[s_{nn}(u),s_{nn}(v)], 
\end{equation*}
which implies $[s_{nn}(u),s_{nn}(v)]\equiv 0$. Furthermore, 
\begin{equation*}
 [s_{ij}(u),s_{nn}(v)]\equiv -\frac{\delta_{ij}}{u^2-v^2}[s_{nn}(u),s_{nn}(v)]\equiv 0
\end{equation*}
for all $-n+1\leq i,j\leq n-1$. The symmetry relation \eqref{s=s} together with the second and third equivalences of \eqref{CT:induction.5} then give
\begin{equation*}
 [s_{-n,-n}(u),s_{nn}(v)]\equiv -\frac{1}{2u-2\ka}\sum_{a=-n}^n[s_{aa}(u),s_{nn}(v)]\equiv -\frac{1}{2u-2\ka}[s_{-n,-n}(u),s_{nn}(v)].
\end{equation*}
Hence, $[s_{-n,-n}(u),s_{nn}(v)]\equiv 0$. 

Now let $i,j,k,\ell$ be such that $-n+1\leq i,j,k,\ell\leq n-1$. As a consequence of the second and third equivalences in \eqref{CT:induction.5}, we have 
\begin{equation*}
[s_{ij}^\circ(u),s_{k\ell}^\circ(v)]\equiv[s_{ij}(\wt{u}),s_{k\ell}(\wt{v})], 
\end{equation*}
where $\wt{u}=u+1/2$ and $\wt{v}=v+1/2$. Thus, appealing to \eqref{[s,s]} we have: 
\begin{align}
 [s_{ij}^\circ(u),s_{k\ell}^\circ(v)]&\equiv \frac{1}{u-v}\left(s_{kj}(\wt{u})s_{i\ell}(\wt{v})-s_{kj}(\wt{v})s_{i\ell}(\wt{u}) \right)+\frac{1}{\wt{u}+\wt{v}}\sum_{a=-n+1}^{n-1}\left(\delta_{kj}s_{ia}(\wt{u})s_{a\ell}(\wt{v})-\delta_{i\ell}s_{ka}(\wt{v})s_{aj}(\wt{u})\right) \nonumber\\
                                     &-\frac{\delta_{ij}}{\wt{u}^2-\wt{v}^2}\sum_{a=-n+1}^{n-1}\left(s_{ka}(\wt{u})s_{a\ell}(\wt{v})-s_{ka}(\wt{v})s_{a\ell}(\wt{u}) \right)\nonumber\\
                                     &-\frac{1}{\wt{u}-\wt{v}-\ka}\sum_{a=-n+1}^{n-1}\left(\delta_{k,-i}\theta_{ia}s_{aj}(\wt{u})s_{-a,\ell}(\wt{v})-\delta_{l,-j}\theta_{aj}s_{k,-a}(\wt{v})s_{ia}(\wt{u}) \right)\nonumber\\
                                     &-\frac{1}{\wt{u}+\wt{v}-\ka}\left(\theta_{j,-k}s_{i,-k}(\wt{u})s_{-j,\ell}(\wt{v})-\theta_{i,-\ell}s_{k,-i}(\wt{v})s_{-\ell,j}(\wt{u}) \right)\nonumber\\
                                     &+\frac{\theta_{i,-j}}{(\wt{u}+\wt{v})(\wt{u}-\wt{v}-\ka)}\sum_{a=-n+1}^{n-1}\left(\delta_{k,-i}s_{-j,a}(\wt{u})s_{a\ell}(\wt{v})-\delta_{\ell,-j}s_{ka}(\wt{v})s_{a,-i}(\wt{u}) \right)\nonumber\\
                                     &+\frac{\theta_{i,-j}}{(\wt{u}-\wt{v})(\wt{u}+\wt{v}-\ka)}\left(s_{k,-i}(\wt{u})s_{-j,\ell}(\wt{v})-s_{k,-i}(\wt{v})s_{-j,\ell}(\wt{u}) \right)\nonumber\\
                                     &-\frac{\theta_{ij}}{(\wt{u}-\wt{v}-\ka)(\wt{u}+\wt{v}-\ka)}\sum_{a=-n+1}^{n-1}\left(\delta_{k,-i}s_{aa}(\wt{u})s_{-j,\ell}(\wt{v})-\delta_{\ell,-j}s_{k,-i}(\wt{v})s_{aa}(\wt{u}) \right)\nonumber\\
                                     &+\frac{1}{\wt{u}+\wt{v}}\left(\delta_{kj}s_{in}(\wt{u})s_{n\ell}(\wt{v})-\delta_{i\ell}s_{kn}(\wt{v})s_{nj}(\wt{u})\right) \label{CT:induction.6}\\
                                     &-\frac{\delta_{ij}}{\wt{u}^2-\wt{v}^2}\left(s_{kn}(\wt{u})s_{n\ell}(\wt{v})-s_{kn}(\wt{v})s_{n\ell}(\wt{u}) \right) \label{CT:induction.7}\\
                                     &-\frac{1}{\wt{u}-\wt{v}-\ka}\left(\delta_{k,-i}\theta_{i,-n}s_{-n,j}(\wt{u})s_{n\ell}(\wt{v})-\delta_{l,-j}\theta_{-n,j}s_{kn}(\wt{v})s_{i,-n}(\wt{u}) \right) \label{CT:induction.8}\\
                                     &+\frac{\theta_{i,-j}}{(\wt{u}+\wt{v})(\wt{u}-\wt{v}-\ka)}\left(\delta_{k,-i}s_{-j,n}(\wt{u})s_{n\ell}(\wt{v})-\delta_{\ell,-j}s_{kn}(\wt{v})s_{n,-i}(\wt{u}) \right) \label{CT:induction.9}\\
                                     &-\frac{\theta_{ij}}{(\wt{u}-\wt{v}-\ka)(\wt{u}+\wt{v}-\ka)}\left(\delta_{k,-i}s_{-n,-n}(\wt{u})s_{-j,\ell}(\wt{v})-\delta_{\ell,-j}s_{k,-i}(\wt{v})s_{-n,-n}(\wt{u}) \right)\nonumber\\
                                     &-\frac{\theta_{ij}}{(\wt{u}-\wt{v}-\ka)(\wt{u}+\wt{v}-\ka)}\left(\delta_{k,-i}s_{nn}(\wt{u})s_{-j,\ell}(\wt{v})-\delta_{\ell,-j}s_{k,-i}(\wt{v})s_{nn}(\wt{u}) \right) \nonumber.
\end{align}
We now need to rewrite \eqref{CT:induction.6}-\eqref{CT:induction.9} in a way that will enable us to compare the right-hand side above with the right-hand side of the reflection equation \eqref{[s,s]} for $\wt{X}(\mfg_{N-2},\mcG')^{tw}$ with $\tilde{s}_{**}'(u)$ replaced by $s_{**}^\circ(u)$.

\medskip

\noindent\textit{Step 2.1:} Re-expressing \eqref{CT:induction.6}.

\medskip

For any $-n+1\leq i,\ell\leq n-1$, using the reflection equation \eqref{[s,s]} for $[s_{in}(\wt{u}),s_{n\ell}(\wt{v})]$ and rearranging the terms yields
\begin{align}
 \frac{1}{\wt{u}+\wt{v}}\, s_{in}(\wt{u})s_{n\ell}(\wt{v})&\equiv  \frac{1}{(\wt{u}-\wt{v})(\wt{u}+\wt{v}-1)}\left(s_{nn}(\wt{u})s_{i\ell}(\wt{v})-s_{nn}(\wt{v})s_{i\ell}(\wt{u})\right)\nonumber\\
                                       &+\frac{1}{(\wt{u}+\wt{v})(\wt{u}+\wt{v}-1)}\sum_{a=-n+1}^{n-1}s_{ia}(\wt{u})s_{a\ell}(\wt{v})-\frac{\delta_{i\ell}}{(\wt{u}+\wt{v})(\wt{u}+\wt{v}-1)}\,s_{nn}(\wt{v})s_{nn}(\wt{u})\label{CT:induction.10}.
\end{align}
This computation, together with \eqref{CT:induction.5}, implies that the expression \eqref{CT:induction.6} can be rewritten as follows:
\begin{align}
 \frac{1}{\wt{u}+\wt{v}}&\left(\delta_{kj}s_{in}(\wt{u})s_{n\ell}(\wt{v})-\delta_{i\ell}s_{kn}(\wt{v})s_{nj}(\wt{u})\right)\nonumber \\
                           &\equiv \frac{1}{(\wt{u}-\wt{v})(\wt{u}+\wt{v}-1)}\left(\delta_{kj}\left(s_{nn}(\wt{u})s_{i\ell}(\wt{v})-s_{nn}(\wt{v})s_{i\ell}(\wt{u})\right)+\delta_{i\ell}\left(s_{nn}(\wt{v})s_{kj}(\wt{u})-s_{nn}(\wt{u})s_{kj}(\wt{v})\right) \right)\nonumber\\
                           &+\frac{1}{(\wt{u}+\wt{v})(\wt{u}+\wt{v}-1)}\sum_{a=-n+1}^{n-1}\left(\delta_{kj}s_{ia}(\wt{u})s_{a\ell}(\wt{v})-\delta_{i\ell}s_{ka}(\wt{v})s_{aj}(\wt{u}) \right)\label{CT:induction.11}.
\end{align}

\medskip

\noindent\textit{Step 2.2:} Re-expressing \eqref{CT:induction.7}.

\medskip

Similarly, \eqref{CT:induction.10} and \eqref{CT:induction.5} imply that \eqref{CT:induction.7} can be expressed as:
\spl{
-\frac{\delta_{ij}}{\wt{u}^2-\wt{v}^2} & \left(s_{kn}(\wt{u})s_{n\ell}(\wt{v})-s_{kn}(\wt{v})s_{n\ell}(\wt{u}) \right) \\
 & \equiv -\frac{\delta_{ij}}{(\wt{u}^2-\wt{v}^2)(\wt{u}+\wt{v}-1)}\sum_{a=-n+1}^{n-1}\left(s_{ka}(\wt{u})s_{a\ell}(\wt{v})- s_{ka}(\wt{v})s_{a\ell}(\wt{u})\right). \label{CT:induction.12}
}

\medskip

\noindent\textit{Step 2.3:} Re-expressing \eqref{CT:induction.8}.

\medskip

We would like to obtain a similar expression for \eqref{CT:induction.8}. An analogous but more lengthy computation to that used in obtaining \eqref{CT:induction.10} gives the relation
\begin{align}
 \frac{\wt{u}-\wt{v}-\ka+1}{\wt{u}-\wt{v}-\ka}s_{-n,j}(\wt{u})s_{n\ell}(\wt{v})&\equiv -\frac{1}{\wt{u}-\wt{v}-\ka}\sum_{a=-n+1}^{n-1}\theta_{-n,a}s_{aj}(\wt{u})s_{-a,\ell}(\wt{v})+\frac{\delta_{\ell,-j}}{\wt{u}-\wt{v}-\ka}\theta_{-n,j}s_{nn}(\wt{v})s_{-n,-n}(\wt{u})\nn\\
                                                      &-\frac{1}{\wt{u}+\wt{v}-\ka}\left(\theta_{j,-n}s_{-n,-n}(\wt{u})s_{-j,\ell}(\wt{v})-\theta_{-n,-\ell}s_{nn}(\wt{v})s_{-\ell,j}(\wt{u}) \right)\nn \\
                                                      &+\frac{\theta_{-n,-j}}{(\wt{u}+\wt{v})(\wt{u}-\wt{v}-\ka)}\sum_{a=-n+1}^{n-1}s_{-j,a}(\wt{u})s_{a\ell}(\wt{v})\nn\\
                                                      &+\frac{\theta_{-n,-j}}{(\wt{u}+\wt{v})(\wt{u}-\wt{v}-\ka)}s_{-j,n}(\wt{u})s_{n\ell}(\wt{v})-\frac{\theta_{-n,-j}\delta_{\ell,-j}}{(\wt{u}+\wt{v})(\wt{u}-\wt{v}-\ka)}s_{nn}(\wt{v})s_{nn}(\wt{u})\nn\\
                                                      &+\frac{\theta_{-n,-j}}{(\wt{u}-\wt{v})(\wt{u}+\wt{v}-\ka)}\left(s_{nn}(\wt{u})s_{-j,\ell}(\wt{v})-s_{nn}(\wt{v})s_{-j,\ell}(\wt{u}) \right)\nn\\
                                                      &-\frac{\theta_{-n,j}}{(\wt{u}-\wt{v}-\ka)(\wt{u}+\wt{v}-\ka)}\sum_{a=-n}^n \left(s_{aa}(\wt{u})s_{-j,\ell}(\wt{v})-\delta_{\ell,-j}s_{nn}(\wt{v})s_{aa}(\wt{u}) \right)\label{CT:induction.13}.
\end{align}
Similarly, since $s_{kn}(\wt{v})s_{i,-n}(\wt{u})\equiv -[s_{i,-n}(\wt{u}),s_{kn}(\wt{v})]$, we have 
\begin{align}
  \frac{\wt{u}-\wt{v}-\ka+1}{\wt{u}-\wt{v}-\ka}s_{kn}(\wt{v})s_{i,-n}(\wt{u})&\equiv-\frac{1}{\wt{u}-\wt{v}-\ka}\sum_{a=-n+1}^{n-1}\theta_{a,-n}s_{k,-a}(\wt{v})s_{ia}(\wt{u}) +\frac{\delta_{k,-i}}{\wt{u}-\wt{v}-\ka}\theta_{i,-n}s_{-n,-n}(\wt{u})s_{nn}(\wt{v})\nn\\
                                                                             &+\frac{1}{\wt{u}+\wt{v}-\ka}\left(\theta_{-n,-k}s_{i,-k}(\wt{u})s_{nn}(\wt{v})-\theta_{i,-n}s_{k,-i}(\wt{v})s_{-n,-n}(\wt{u}) \right)\nn\\
                                                                             &+\frac{\theta_{in}}{(\wt{u}+\wt{v})(\wt{u}-\wt{v}-\ka)}\sum_{a=-n+1}^{n-1}s_{ka}(\wt{v})s_{a,-i}(\wt{u})\nn\\
                                                                             &+\frac{\theta_{in}}{(\wt{u}+\wt{v})(\wt{u}-\wt{v}-\ka)}s_{kn}(\wt{v})s_{n,-i}(\wt{u})-\frac{\theta_{in}\delta_{k,-i}}{(\wt{u}+\wt{v})(\wt{u}-\wt{v}-\ka)}s_{nn}(\wt{u})s_{nn}(\wt{v})\nn\\
                                                                             &-\frac{\theta_{in}}{(\wt{u}-\wt{v})(\wt{u}+\wt{v}-\ka)}\left(s_{k,-i}(\wt{u})s_{nn}(\wt{v})-s_{k,-i}(\wt{v})s_{nn}(\wt{u}) \right)\nn\\
                                                                             &+\frac{\theta_{i,-n}}{(\wt{u}-\wt{v}-\ka)(\wt{u}+\wt{v}-\ka)}\sum_{a=-n}^n\left(\delta_{k,-i}s_{aa}(\wt{u})s_{nn}(\wt{v})-s_{k,-i}(\wt{v})s_{aa}(\wt{u}) \right) . \label{CT:induction.14}
\end{align}
Using the two equivalences \eqref{CT:induction.13} and \eqref{CT:induction.14}, together with \eqref{CT:induction.5}, we obtain the following for \eqref{CT:induction.8}:
\begin{align}
 -\frac{1}{\wt{u}-\wt{v}-\ka}&\left(\delta_{k,-i}\theta_{i,-n}s_{-n,j}(\wt{u})s_{n\ell}(\wt{v})-\delta_{l,-j}\theta_{-n,j}s_{kn}(\wt{v})s_{i,-n}(\wt{u}) \right)\nonumber \\
                    &=\frac{1}{(\wt{u}-\wt{v}-\ka)(\wt{u}-\wt{v}-\ka+1)}\sum_{a=-n+1}^{n-1}\left(\delta_{k,-i}\theta_{ia}s_{aj}(\wt{u})s_{-a,\ell}(\wt{v})-\delta_{l,-j}\theta_{aj}s_{k,-a}(\wt{v})s_{ia}(\wt{u}) \right)\nonumber\\
                    &-\frac{\theta_{i,-j}}{(\wt{u}+\wt{v})(\wt{u}-\wt{v}-\ka)(\wt{u}-\wt{v}-\ka+1)}\sum_{a=-n+1}^{n-1}\left(\delta_{k,-i}s_{-j,a}(\wt{u})s_{a\ell}(\wt{v})-\delta_{\ell,-j}s_{ka}(\wt{v})s_{a,-i}(\wt{u}) \right)\nonumber\\
                    &+\frac{\theta_{ij}}{(\wt{u}-\wt{v}-\ka)(\wt{u}+\wt{v}-\ka)(\wt{u}-\wt{v}-\ka+1)}\sum_{a=-n}^{n}\left(\delta_{k,-i}s_{aa}(\wt{u})s_{-j,\ell}(\wt{v})-\delta_{\ell,-j}s_{k,-i}(\wt{v})s_{aa}(\wt{u}) \right)\nonumber\\
                    &+\frac{1}{(\wt{u}+\wt{v}-\ka)(\wt{u}-\wt{v}-\ka+1)}\delta_{k,-i}\left(\theta_{ij}s_{-n,-n}(\wt{u})s_{-j,\ell}(\wt{v})-\theta_{i,-\ell}s_{nn}(\wt{v})s_{-\ell,j}(\wt{u}) \right) \nonumber\\
                    &+\frac{1}{(\wt{u}+\wt{v}-\ka)(\wt{u}-\wt{v}-\ka+1)}\delta_{\ell,-j}\left(\theta_{j,-k}s_{i,-k}(\wt{u})s_{nn}(\wt{v})-\theta_{ij}s_{k,-i}(\wt{v})s_{-n,-n}(\wt{u}) \right)\nonumber\\
                    &-\frac{\theta_{i,-j}\delta_{k,-i}}{(\wt{u}-\wt{v})(\wt{u}+\wt{v}-\ka)(\wt{u}-\wt{v}-\ka+1)}\left(s_{nn}(\wt{u})s_{-j,\ell}(\wt{v})-s_{nn}(\wt{v})s_{-j,\ell}(\wt{u}) \right)\nonumber\\
                    &-\frac{\theta_{i,-j}\delta_{\ell,-j}}{(\wt{u}-\wt{v})(\wt{u}+\wt{v}-\ka)(\wt{u}-\wt{v}-\ka+1)}\left(s_{k,-i}(\wt{u})s_{nn}(\wt{v})-s_{k,-i}(\wt{v})s_{nn}(\wt{u}) \right)\nonumber\\
                    &-\frac{\theta_{i,-j}}{(\wt{u}+\wt{v})(\wt{u}-\wt{v}-\ka)(\wt{u}-\wt{v}-\ka+1)}\left(\delta_{k,-i}s_{-j,n}(\wt{u})s_{n,\ell}(\wt{v})-\delta_{\ell,-j}s_{kn}(\wt{v})s_{n,-i}(\wt{u}) \right)\label{CT:induction.15}.               
\end{align}

\medskip

\noindent\textit{Step 2.4:} Re-expressing \eqref{CT:induction.9}.

\medskip

If we add the last line of the above to \eqref{CT:induction.9}, we obtain: 

\begin{equation*}
 \frac{\theta_{i,-j}}{(\wt{u}+\wt{v})(\wt{u}-\wt{v}-\ka+1)}\left(\delta_{k,-i}s_{-j,n}(\wt{u})s_{n,\ell}(\wt{v})-\delta_{\ell,-j}s_{kn}(\wt{v})s_{n,-i}(\wt{u}) \right).
\end{equation*}
We can also re-express this using \eqref{CT:induction.10} and \eqref{CT:induction.5}. This yields:
\begin{align}
  \frac{\theta_{i,-j}}{(\wt{u}+\wt{v})(\wt{u}-\wt{v}-\ka+1)}&\left(\delta_{k,-i}s_{-j,n}(\wt{u})s_{n,\ell}(\wt{v})-\delta_{\ell,-j}s_{kn}(\wt{v})s_{n,-i}(\wt{u}) \right)\nonumber\\
                                           =&\frac{\theta_{i,-j}\delta_{k,-i}}{(\wt{u}-\wt{v})(\wt{u}+\wt{v}-1)(\wt{u}-\wt{v}-\ka+1)}\left(s_{nn}(\wt{u})s_{-j,\ell}(\wt{v})-s_{nn}(\wt{v})s_{-j,\ell}(\wt{u}) \right)\nonumber\\
                                           +&\frac{\theta_{i,-j}\delta_{\ell,-j}}{(\wt{u}-\wt{v})(\wt{u}+\wt{v}-1)(\wt{u}-\wt{v}-\ka+1)}\left(s_{nn}(\wt{v})s_{k,-i}(\wt{u})-s_{nn}(\wt{u})s_{k,-i}(\wt{v}) \right)\nonumber  \\
                                           +&\frac{\theta_{i,-j}}{(\wt{u}+\wt{v})(\wt{u}+\wt{v}-1)(\wt{u}-\wt{v}-\ka+1)}\sum_{a=-n+1}^{n-1}(\delta_{k,-i}s_{-j,a}(\wt{u})s_{a\ell}(\wt{v}) - \delta_{l,-j}s_{ka}(\wt{v})s_{a,-i}(\wt{u}))\label{CT:induction.16}.
\end{align}

\medskip

\noindent\textit{Step 2.5:} 

\medskip

Next, observe that the following identities hold: 
\begin{equation*}
 \frac{1}{\wt{u}+\wt{v}}+\frac{1}{(\wt{u}+\wt{v})(\wt{u}+\wt{v}-1)}=\frac{1}{u+v},
\end{equation*}
\begin{equation*}
 \frac{1}{(\wt{u}-\wt{v}-\ka)(\wt{u}-\wt{v}-\ka+1)}-\frac{1}{\wt{u}-\wt{v}-\ka}=-\frac{1}{u-v-\ka'},
\end{equation*}
\begin{align*}
 \frac{1}{(\wt{u}+\wt{v})(\wt{u}-\wt{v}-\ka)}-\frac{1}{(\wt{u}+\wt{v})(\wt{u}-\wt{v}-\ka)(\wt{u}-\wt{v}-\ka+1)}&+\frac{1}{(\wt{u}+\wt{v})(\wt{u}+\wt{v}-1)(\wt{u}-\wt{v}-\ka+1)}\\
                            &=\frac{1}{(u-v-\ka')(u+v)} .
\end{align*}
Therefore, combining the new expressions \eqref{CT:induction.11}, \eqref{CT:induction.12}, \eqref{CT:induction.15} and \eqref{CT:induction.16} and substituting them back into \eqref{CT:induction.6}-\eqref{CT:induction.9} gives:
\begin{align}
 [s_{ij}^\circ(u),s_{k\ell}^\circ(v)]&\equiv \frac{1}{u-v}\left(s_{kj}(\wt{u})s_{i\ell}(\wt{v})-s_{kj}(\wt{v})s_{i\ell}(\wt{u}) \right)+\frac{1}{u+v}\sum_{a=-n+1}^{n-1}\left(\delta_{kj}s_{ia}(\wt{u})s_{a\ell}(\wt{v})-\delta_{i\ell}s_{ka}(\wt{v})s_{aj}(\wt{u})\right) \nonumber\\
                                     &-\frac{\delta_{ij}}{u^2-v^2}\sum_{a=-n+1}^{n-1}\left(s_{ka}(\wt{u})s_{a\ell}(\wt{v})-s_{ka}(\wt{v})s_{a\ell}(\wt{u}) \right)\nonumber\\
                                     &-\frac{1}{u-v-\ka'}\sum_{a=-n+1}^{n-1}\left(\delta_{k,-i}\theta_{ia}s_{aj}(\wt{u})s_{-a,\ell}(\wt{v})-\delta_{l,-j}\theta_{aj}s_{k,-a}(\wt{v})s_{ia}(\wt{u}) \right)\nonumber\\
                                     &-\frac{1}{u+v-\ka'}\left(\theta_{j,-k}s_{i,-k}(\wt{u})s_{-j,\ell}(\wt{v})-\theta_{i,-\ell}s_{k,-i}(\wt{v})s_{-\ell,j}(\wt{u}) \right)\nonumber\\
                                     &+\frac{\theta_{i,-j}}{(u+v)(u-v-\ka')}\sum_{a=-n+1}^{n-1}\left(\delta_{k,-i}s_{-j,a}(\wt{u})s_{a\ell}(\wt{v})-\delta_{\ell,-j}s_{ka}(\wt{v})s_{a,-i}(\wt{u}) \right)\nonumber\\
                                     &+\frac{\theta_{i,-j}}{(u-v)(u+v-\ka')}\left(s_{k,-i}(\wt{u})s_{-j,\ell}(\wt{v})-s_{k,-i}(\wt{v})s_{-j,\ell}(\wt{u}) \right)\nonumber\\
                                     &-\frac{\theta_{ij}}{(u-v-\ka')(u+v-\ka')}\sum_{a=-n+1}^{n-1}\left(\delta_{k,-i}s_{aa}(\wt{u})s_{-j,\ell}(\wt{v})-\delta_{\ell,-j}s_{k,-i}(\wt{v})s_{aa}(\wt{u}) \right)+\mathcal{B}(u,v),\nonumber                                    
\end{align}
where $\mathcal{B}(u,v)$ is defined as the following operator on $V_+$: 
\begin{align*}
 \mathcal{B}(u,v)&=\frac{1}{u^2-v^2}\left(\delta_{kj}\left(s_{nn}(\wt{u})s_{i\ell}(\wt{v})-s_{nn}(\wt{v})s_{i\ell}(\wt{u})\right)+\delta_{i\ell}\left(s_{nn}(\wt{v})s_{kj}(\wt{u})-s_{nn}(\wt{u})s_{kj}(\wt{v})\right) \right)\\
                 &+\frac{1}{(u+v-\ka')(u-v-\ka')}\delta_{k,-i}\left(\theta_{ij}s_{-n,-n}(\wt{u})s_{-j,\ell}(\wt{v})-\theta_{i,-\ell}s_{nn}(\wt{v})s_{-\ell,j}(\wt{u}) \right)\\
                 &+\frac{1}{(u+v-\ka')(u-v-\ka')}\delta_{\ell,-j}\left(\theta_{j,-k}s_{i,-k}(\wt{u})s_{nn}(\wt{v})-\theta_{ij}s_{k,-i}(\wt{v})s_{-n,-n}(\wt{u}) \right)\\
                 &-\frac{\theta_{i,-j}\delta_{k,-i}}{(u-v)(u+v-\ka')(u-v-\ka')}\left(s_{nn}(\wt{u})s_{-j,\ell}(\wt{v})-s_{nn}(\wt{v})s_{-j,\ell}(\wt{u}) \right)\\
                 &-\frac{\theta_{i,-j}\delta_{\ell,-j}}{(u-v)(u+v-\ka')(u-v-\ka')}\left(s_{k,-i}(\wt{u})s_{nn}(\wt{v})-s_{k,-i}(\wt{v})s_{nn}(\wt{u}) \right)\\
                 &+\frac{\theta_{i,-j}\delta_{k,-i}}{(u-v)(u+v)(u-v-\ka')}\left(s_{nn}(\wt{u})s_{-j,\ell}(\wt{v})-s_{nn}(\wt{v})s_{-j,\ell}(\wt{u}) \right)\\
                 &+\frac{\theta_{i,-j}\delta_{\ell,-j}}{(u-v)(u+v)(u-v-\ka')}\left(s_{nn}(\wt{v})s_{k,-i}(\wt{u})-s_{nn}(\wt{u})s_{k,-i}(\wt{v}) \right)\\
                 &-\frac{\theta_{ij}}{(u-v-\ka')(u+v-\ka')}\left(\delta_{k,-i}s_{-n,-n}(\wt{u})s_{-j,\ell}(\wt{v})-\delta_{\ell,-j}s_{k,-i}(\wt{v})s_{-n,-n}(\wt{u}) \right)\\
                 &-\frac{\theta_{ij}}{(u-v-\ka')(u+v-\ka')}\left(\delta_{k,-i}s_{nn}(\wt{u})s_{-j,\ell}(\wt{v})-\delta_{\ell,-j}s_{k,-i}(\wt{v})s_{nn}(\wt{u}) \right).
\end{align*}
Adding terms together, and applying \eqref{CT:induction.5} where necessary, we obtain the equivalence of operators: \enlargethispage{1em}
\begin{align}
\mathcal{B}(u,v)&\equiv\frac{1}{u^2-v^2}\left(\delta_{kj}\left(s_{nn}(\wt{u})s_{i\ell}(\wt{v})-s_{nn}(\wt{v})s_{i\ell}(\wt{u})\right)+\delta_{i\ell}\left(s_{nn}(\wt{v})s_{kj}(\wt{u})-s_{nn}(\wt{u})s_{kj}(\wt{v})\right) \right)\nonumber\\
&+\frac{1}{(u+v-\ka')(u-v-\ka')}\delta_{\ell,-j}\left(\theta_{j,-k}s_{i,-k}(\wt{u})s_{nn}(\wt{v})+\theta_{ij}s_{k,-i}(\wt{v})s_{nn}(\wt{u})\right)\nonumber\\
&-\frac{1}{(u+v-\ka')(u-v-\ka')}\delta_{k,-i}\left(\theta_{i,-\ell}s_{nn}(\wt{v})s_{-\ell,j}(\wt{u})+\theta_{ij}s_{nn}(\wt{u})s_{-j,\ell}(\wt{v}) \right) \nonumber\\
&-\frac{\ka'\delta_{k,-i}\theta_{i,-j} }{(u^2-v^2)(u+v-\ka')(u-v-\ka')}\left(s_{nn}(\wt{u})s_{-j,\ell}(\wt{v})-s_{nn}(\wt{v})s_{-j,\ell}(\wt{u})\right)\nonumber\\
&-\frac{\ka'\delta_{\ell,-j}\theta_{i,-j} }{(u^2-v^2)(u+v-\ka')(u-v-\ka')}\left(s_{nn}(\wt{v})s_{k,-i}(\wt{u})-s_{nn}(\wt{u})s_{k,-i}(\wt{v}) \right)\label{CT:induction.B(u,v)}.
\end{align}

Conversely, let $\mathcal{D}(u,v)$ be the expression on the right-hand side of the reflection equation \eqref{[s,s]} for $\wt{X}(\mfg_{N-2},\mcG')^{tw}$ with $\wt{s}_{**}'(u)$ replaced by $s_{**}^\circ(u)$. Using the definition of the elements $s_{ij}^\circ(u)$ and again appealing to \eqref{CT:induction.5} where necessary, we obtain the following result after a lengthy computation: 
\begin{equation*}
 \mathcal{D}(u,v)\equiv [s_{ij}^\circ(u),s_{k\ell}^\circ(v)]-\mathcal{B}(u,v)+\mcA(u,v) ,
\end{equation*}
where $\mcA(u,v)$ is the operator defined by: 
\begin{align}
 \mcA(u,v)&=\frac{1}{u^2-v^2}\left(\delta_{kj}\left(s_{nn}(\wt{u})s_{i\ell}(\wt{v})-s_{nn}(\wt{v})s_{i\ell}(\wt{u})\right)+\delta_{i\ell}\left(s_{nn}(\wt{v})s_{kj}(\wt{u})-s_{nn}(\wt{u})s_{kj}(\wt{v})\right) \right)\nonumber\\
                 &+\frac{1}{(u+v-\ka')(u-v-\ka')}\delta_{\ell,-j}\left(\theta_{j,-k}s_{i,-k}(\wt{u})s_{nn}(\wt{v})+\theta_{ij}s_{k,-i}(\wt{v})s_{nn}(\wt{u})\right)\nonumber\\
                 &-\frac{1}{(u+v-\ka')(u-v-\ka')}\delta_{k,-i}\left(\theta_{i,-\ell}s_{nn}(\wt{v})s_{-\ell,j}(\wt{u})+\theta_{ij}s_{nn}(\wt{u})s_{-j,\ell}(\wt{v}) \right) \nonumber\\
                 &+\frac{\ka'\theta_{ij}+\theta_{i,-j}}{u(u+v-\ka')(u-v-\ka')}\left(\delta_{k,-i}s_{nn}(\wt{u})s_{-j,\ell}(\wt{v})-\delta_{\ell,-j}s_{nn}(\wt{u})s_{k,-i}(\wt{v}) \right)\nonumber\\
                 &-\frac{\ka'\delta_{k,-i}\theta_{i,-j} }{(u^2-v^2)(u+v-\ka')(u-v-\ka')}\left(s_{nn}(\wt{u})s_{-j,\ell}(\wt{v})-s_{nn}(\wt{v})s_{-j,\ell}(\wt{u})\right)\nonumber\\
                 &-\frac{\ka'\delta_{\ell,-j}\theta_{i,-j}  }{(u^2-v^2)(u+v-\ka')(u-v-\ka')}\left(s_{nn}(\wt{v})s_{k,-i}(\wt{u})-s_{nn}(\wt{u})s_{k,-i}(\wt{v}) \right)\nonumber\\
                 &-\frac{\theta_{ij}(N-2)}{2u(u-v-\ka')(u+v-\ka')}\left(\delta_{k,-i}s_{nn}(\wt{u})s_{-j,\ell}(\wt{v})-\delta_{\ell,-j}s_{k,-i}(\wt{v})s_{nn}(\wt{u}) \right) \label{CT:induction.A(u,v)}.
                 \end{align}
Therefore, to complete the proof of the lemma it remains only to see $\mcA(u,v)\equiv \mathcal{B}(u,v)$. Comparing \eqref{CT:induction.A(u,v)} with \eqref{CT:induction.B(u,v)}, we see that it is enough to show $ \ka'\theta_{ij}+\theta_{i,-j}-\theta_{ij}(\tfrac{N}{2}-1)=0$, which follows from $\ka=\frac{N}{2}\mp 1$, $\frac{N}{2}-1=\ka-1\pm 1=\ka'\pm1$ and  $\theta_{i,-j}=\pm\theta_{i,j}$.
Thus the family of operators $\{s_{ij}^\circ(u)\}_{-n+1\leq i,j\leq n-1}$ satisfy the defining reflection equation of the algebra $\wt X(\mfg_{N-2},\mcG^\prime)^{tw}$. 
\end{proof}

It is natural to ask whether the action of the extended reflection algebra $\wt X(\mfg_{N-2},\mcG^\prime)^{tw}$ on the space $V_+$ factors through the extended twisted Yangian $X(\mfg_{N-2},\mcG^\prime)^{tw}$.
We will soon see that if the symmetric pair $(\mfg_N,\mfg_N^\rho)$ is of type CI or DIII, then this is indeed the case. However, if the pair
$(\mfg_N,\mfg_N^\rho)$ is not of type CI or DIII, then the operators $s^\circ_{ij}(u)$ fail to satisfy the defining symmetry relation of the algebra $X(\mfg_{N-2},\mcG^\prime)^{tw}$. Our goal for the remainder of this subsection is to show that, in the general situation, the action of $\wt X(\mfg_{N-2},\mcG^\prime)^{tw}$ on the space $V_+$ can be twisted by a suitable automorphism in a way which results in the modified action of $\wt X(\mfg_{N-2},\mcG^\prime)^{tw}$ on $V_+$ factoring through the extended Yangian $X(\mfg_{N-2},\mcG^\prime)^{tw}$.

Recall the rational function $p(u)$ in $u$ associated to the algebra $X(\mfg_N,\mcG)^{tw}$ which has been defined in \eqref{p(u)}. Let $p^\prime(u)$ be the corresponding 
series for the algebra $X(\mfg_{N-2},\mcG^\prime)^{tw}$. That is, 
\begin{equation*}
 p(u)=(\pm)1\mp\frac{1}{2u-\ka}+\frac{\Tr\,(\mcG(u))}{2u-2\ka} \qu \text{ and } \qu p^\prime(u)=(\pm)1\mp\frac{1}{2u-\ka'}+\frac{\Tr\,(\mcG^\prime(u))}{2u-2\ka'}.
\end{equation*}

\begin{prop}\label{CT:Prop.induction}
 Let $h(u)$ be any series in $1+u^{-1}\C[[u^{-1}]]$ satisfying the relation
 \begin{equation}
  h(u)h(\ka'-u)^{-1}=p(u+1/2)^{-1}p^\prime(u). \label{h-p}
 \end{equation}
Then the assignment $s^\prime_{ij}(u)\mapsto h(u)s_{ij}^\circ(u)$ defines a representation of $X(\mfg_{N-2},\mcG^\prime)^{tw}$ in the space $V_+$. Moreover, if $V$ is a highest 
weight module with the highest weight $\mu(u)$ and the highest weight vector $\xi$, then the cyclic span $X(\mfg_{N-2},\mcG^\prime)^{tw}\xi$ is a highest weight module 
with the highest weight $h(u)\mu^\circ(u)=(h(u)\mu_i^\circ(u))_{i\in \mathcal{I}_{N-2}}$, where $\mu_i^\circ(u)=\mu_i(u+1/2)+\tfrac{1}{2u}\mu_n(u+1/2)$ for all $i\in \mathcal{I}_{N-2}$. 
\end{prop}

\begin{proof}
From Lemma \ref{CT:Lemma.induction}, we know that $V_+$ admits the structure of a $\wt X(\mfg_{N-2},\mcG^\prime)^{tw}$-module via the assignment $\tilde{s}^\prime_{ij}(u)\mapsto s_{ij}^{\circ}(u)$. We may consider the $\wt X(\mfg_{N-2},\mcG^\prime)^{tw}$-module $V_+^{\wt \nu_h}$ obtained by twisting $V_+$ by an automorphism $\wt \nu_h$ of the form \eqref{tnu_h}. Therefore, it suffices to show that the operators $h(u)\,s^\circ_{ij}(u)$ satisfy the defining symmetry relation of $X(\mfg_{N-2},\mcG^\prime)^{tw}$. That is, 
\begin{samepage}
\begin{align}
 \theta_{ij}h(u)s^\circ_{-j,-i}(u)\equiv & \,(\pm)h(\ka'-u)s_{ij}^\circ(\ka'-u)\pm \frac{h(u)s^\circ_{ij}(u)-h(\ka'-u)s^\circ_{ij}(\ka'-u)}{2u-\ka'}\nonumber\\
                              + & \frac{\Tr\, (\mcG'(u))h(\ka'-u)s^\circ_{ij}(\ka'-u)-\delta_{ij}h(u)\sum_{k=-n+1}^{n-1}s^\circ_{kk}(u)}{2u-2\ka'} \label{CT:Prop.induction.1}
\end{align}
\end{samepage}

\noindent for all $-n+1\leq i,j\leq n-1$, where $``\equiv"$ denotes equivalence of operators on the space $V_+$.

Suppose first that $i\neq j$ and set $\wt u=u+1/2$. Then $s^\circ_{ij}(u)=s_{ij}(\wt u)$ and $s^\circ_{-j,-i}(u)=s_{-j,-i}(\wt u)$, so equation	 \eqref{CT:Prop.induction.1} becomes equivalent to the relation 
\begin{equation*}
h(u)\left(\theta_{ij}s_{-j,-i}(\wt u)\mp\frac{s_{ij}(\wt u)}{2\wt{u}-\ka}\right)\equiv h(\ka'-u)s_{ij}(\ka-\wt{u})\left((\pm)1\mp\frac{1}{2u-\ka'}+\frac{\Tr\, (\mcG'(u))}{2u-2\ka'}\right).
\end{equation*}
By the defining symmetry relation in $X(\mfg_N,\mcG)^{tw}$, the left hand side is just $h(u)p(\wt u)s_{ij}(\ka-\wt{u})$, whereas by definition of $p'(u)$
the right hand side is $h(\ka'-u)s_{ij}(\ka-\wt{u})p^\prime(u)$. Therefore, since $h(u)$ satisfies the relation \eqref{h-p}, both sides are equal.

Now suppose instead that $i=j$. Then \eqref{CT:Prop.induction.1} is equivalent to the relation
\begin{equation*}
 h(u)\left(s^\circ_{-i,-i}(u)\mp\frac{s^\circ_{ii}(u)}{2u-\ka'}+\frac{\sum_{k=-n+1}^{n-1}s^\circ_{kk}(u)}{2u-2\ka'}\right)\equiv h(\ka'-u)s^\circ_{ii}(\ka'-u)p^\prime(u).
\end{equation*}
Therefore, by \eqref{h-p}, it suffices to show that 
\begin{equation}
 p(\wt u)s^\circ_{ii}(\ka'-u)\equiv s^\circ_{-i,-i}(u)\mp\frac{s^\circ_{ii}(u)}{2u-\ka'}+\frac{\sum_{k=-n+1}^{n-1}s^\circ_{kk}(u)}{2u-2\ka'}. \label{CT:Prop.induction.2}
\end{equation}
By definition of the operators $s^\circ_{ij}(u)$, we have 
\begin{equation}
 s^\circ_{-i,-i}(u)\mp\frac{s^\circ_{ii}(u)}{2u-\ka'}+\frac{\sum_{k=-n+1}^{n-1}s^\circ_{kk}(u)}{2u-2\ka'}=s_{-i,-i}(\wt u)\mp\frac{s_{ii}(\wt u)}{2u-\ka'}+\frac{\sum_{k=-n+1}^{n-1}s_{kk}(\wt u)}{2u-2\ka'}+\frac{s_{nn}(\wt{u})}{2u}p_0(u),\label{CT:Prop.induction.3}
\end{equation}
where 
\begin{equation}
 p_0(u)=1\mp\frac{1}{2u-\ka'}+\frac{N-2}{2u-2\ka'}=\frac{u(2u-\ka'\pm 1)}{(\ka'-2u)(\ka'-u)}. \label{p0(u)}
\end{equation}
By the symmetry relation in $X(\mfg_N,\mcG)^{tw}$, 
\begin{equation*}
p(\wt u)s_{ii}(\ka-\wt{u})=s_{-i,-i}(\wt{u})\mp\frac{s_{ii}(\wt{u})}{2u-\ka'}+\frac{\sum_{k=-n}^ns_{kk}(\wt u)}{2u-2\ka'-1},
\end{equation*}
and this formula also holds for $i=n$. This implies that the left hand side of \eqref{CT:Prop.induction.2} can be expressed as 
\begin{align*}
 p(\wt u)s^\circ_{ii}(\ka'-u)& =p(\wt u)s_{ii}(\ka-\wt{u})+\frac{p(\wt u)}{2\ka'-2u} \, s_{nn}(\ka-\wt{u})\\
                             & =s_{-i,-i}(\wt{u})\mp\frac{s_{ii}(\wt{u})}{2u-\ka'}+\frac{\sum_{k=-n+1}^{n-1}s_{kk}(\wt u)}{2u-2\ka'}\\
                             &+\left(\frac{1}{2u-2\ka'-1}+\frac{1}{2\ka'-2u}+\frac{1}{(2\ka'-2u)(2u-2\ka'-1)} \right)s_{-n,-n}(\wt u)\\
                             &+\left(\frac{1}{2u-2\ka'-1}\mp\frac{1}{(2u-\ka')(2\ka'-2u)}+\frac{1}{(2u-2\ka'-1)(2\ka'-2u)}\right)s_{nn}(\wt u)\\
                             & =s_{-i,-i}(\wt{u})\mp\frac{s_{ii}(\wt{u})}{2u-\ka'}+\frac{\sum_{k=-n+1}^{n-1}s_{kk}(\wt u)}{2u-2\ka'}+\frac{(2u-\ka'\pm 1)}{2(\ka'-2u)(\ka'-u)} \, s_{nn}(\wt u),                             
\end{align*}
which, by \eqref{CT:Prop.induction.3} and \eqref{p0(u)}, is precisely the right hand side of \eqref{CT:Prop.induction.2}. The second statement of the proposition is an immediate consequence of the definition of the action of $X(\mfg_{N-2},\mcG^\prime)^{tw}$ on the space $V_+$. 
\end{proof}

Observe that if $(\mfg_N,\mfg_N^\rho)$ is equal to $(\mfg_N,\mfgl_n)$ with $\mfg_N=\mfso_{2n}$ or $\mfsp_{2n}$, then $\Tr\,(\mcG(u))=\Tr\,(\mcG^\prime(u))=0$, and
so $p(u+1/2)=p'(u)$. An immediate corollary of Proposition \ref{CT:Prop.induction} is then that, if $(\mfg_N,\mfg_N^\rho)$ is of type CI or DIII, the action of $\wt X(\mfg_{N-2},\mcG^\prime)^{tw}$ on $V_+$ 
provided by Lemma \ref{CT:Lemma.induction} factors through $X(\mfg_{N-2},\mcG^\prime)^{tw}$.

We conclude this subsection with a few comments regarding the existence of a series $h(u)$ satisfying \eqref{h-p}, as well as an equivalent interpretation of Proposition \ref{CT:Prop.induction}. Let $c(u)\in 1+u^{-1}\C[[u^{-1}]]$ be the series defined by $c(u)=p(u+1/2)p'(u)^{-1}$. Since $p(u)$ and $p'(u)$ satisfy \eqref{p(u)p(k-u)}, we have that $c(u)=c(\ka'-u)^{-1}$. Let $h(u)\in 1+u^{-1}\C[[u^{-1}]]$ be such that $c(u)^{-1} = h(u)^2$; then $c(u)=c(\ka'-u)^{-1}$ implies that $h(u)^{-1} = h(\ka'-u)$, hence $c(u) = h^{-1}(u)h(\ka'-u)$. In particular, this implies that there always exists a series $h(u)\in 1+u^{-1}\C[[u^{-1}]]$ satisfying \eqref{h-p}.

Recall that we have the isomorphism $X(\mfg_{N-2},\mcG^\prime)^{tw}\cong \wt X(\mfg_{N-2},\mcG^\prime)^{tw}/(\msc(u)-1)$, where $\msc(u)$ is the central series defined in \eqref{sc(u)}.
The image of $\msc(u)$ under the automorphism $\wt\nu_h$ is $h(u)h(\ka'-u)^{-1}\msc(u)$. Since the action of $\wt X(\mfg_{N-2},\mcG^\prime)^{tw}$
on $V_+^{\wt \nu_h}$ factors through $X(\mfg_{N-2},\mcG^\prime)^{tw}$, the central series $h(u)h(\ka'-u)^{-1}\msc(u)$ must act as $1$ in $V_+$. Therefore, Proposition
\ref{CT:Prop.induction} is equivalent to the statement that $\msc(u)$ must act as the series $c(u)=p(u+1/2)p'(u)^{-1}$ in the representation $V_+$ from Lemma \ref{CT:Lemma.induction}.



\subsection{Connection with representation theory of Molev-Ragoucy reflection algebras}\label{Subsection:HWT.reflection}
In this subsection we introduce an important connection between the representation theory of the extended twisted Yangians $X(\mfg_N,\mcG)^{tw}$ and that of the 
Molev-Ragoucy reflection algebras $\mcB(\wt{N},\tilde q)$ for some appropriate choice of $\wt N$ and $\tilde q$. The definition of the latter algebras was recalled in Subsection \ref{sec:TY-MR}: they are the twisted Yangians associated to
the symmetric pairs $(\mfsl_{\wt N},(\mfgl_{\wt N-\tilde q}\oplus \mfgl_{\tilde q})\cap \mfsl_{\wt N})$ of type AIII (if $\tilde q>0$) and type A0 (if $\tilde q=0$). We begin by recalling the classification of finite-dimensional irreducible $\mcB(\wt{N},\tilde{q})$-modules
obtained in \cite{MR}. 

\medskip 

Fix $0\leq \tilde{q}\leq \wt N$. A representation $V$ of $\mcB(\wt{N},\tilde{q})$ is a \textit{highest weight representation} if there exists a nonzero vector $\xi\in V$ such that $V=\mcB(\wt{N},\tilde{q})\xi$ and the following 
conditions are satisfied: 
 \begin{alignat*}{4}
  &b_{ij}(u)\xi=0 \quad &&\text{ for } && 1\leq i<j\leq \wt{N}, \quad &&\text{ and } \nonumber\\
  &b_{ii}(u)\xi=\mu_i(u)\xi \quad  &&\text{ for } &&1\leq i\leq \wt{N}, && 
 \end{alignat*}
where for each $1\leq i\leq \wt{N}$, $\mu_i(u)$ is a formal power series in $u^{-1}$ with constant term equal to $\veps_i$ (see \eqref{B=TGT}): 
$$\mu_i(u)=\veps_i+\sum_{r=1}^{\infty}\mu_i^{(r)}u^{-r}, \quad \mu_i^{(r)}\in \mathbb{C}. $$

As usual, we call $\mu(u)=(\mu_1(u),\ldots,\mu_{\wt{N}}(u))$ the highest weight, and the vector $\xi$ the highest weight vector. 

\medskip 

Given an $\wt{N}$-tuple $(\mu_1(u),\ldots,\mu_{\wt{N}}(u))$, the Verma module $M(\mu(u))$ is defined the same way as for $X(\mfg_N)$ and  $X(\mfg_N,\mcG)^{tw}$ and, by Theorem 4.2 in \cite{MR}, is 
non-trivial if and only if the components of the highest weight satisfy the relations
\begin{gather}
\mu_{\wt{N}}(u)\mu_{\wt{N}}(-u)=1, \label{HWT:refl.nontrivial.1}\\
\wt\mu_i(u)\wt\mu_i(-u+\wt{N}-i)=\wt\mu_{i+1}(u)\wt\mu_{i+1}(-u+\wt{N}-i). \label{HWT:refl.nontrivial.2}
\end{gather}
for all $i=1,\ldots,\wt{N}-1$, where the components of $\wt \mu(u)$ are defined in \eqref{HWT:tilde_mu_i(u)}. 

\medskip 

For each $\wt{N}$-tuple $\mu(u)$ whose components satisfy \eqref{HWT:refl.nontrivial.1} and \eqref{HWT:refl.nontrivial.2}, there is a unique irreducible module $V(\mu(u))$ with the highest weight $\mu(u)$. Moreover, by Theorem 4.6 (i) in \cite{MR}, if $\tilde{q}=0$ or $\tilde{q}=\wt{N}$, then $V(\mu(u))$ is finite-dimensional if and only if there exist monic polynomials $P_2(u),\ldots, P_{\wt{N}}(u)$ in $u$ such that 
$P_i(-u+\wt{N}-i+2)=P_i(u)$ for each $i$, and 
\begin{equation}
 \frac{\wt\mu_{i-1}(u)}{\wt\mu_{i}(u)}=\frac{P_i(u+1)}{P_i(u)} \quad \text{ for all }\quad 2\leq i\leq \wt{N}. \label{HWT:refl.fd.1}
\end{equation}
If $0 < \tilde{q} < N$, set $\tilde{p}=\wt{N}-\tilde{q}$. By Theorem 4.6 (ii) in \cite{MR}, $V(\mu(u))$ is finite-dimensional if and only if there exists $\gamma\in \C$ and  monic polynomials $P_2(u),\ldots,P_{\wt{N}}(u)$ in $u$ such that 
$P_i(-u+\wt{N}-i+2)=P_i(u)$ for each $i$, $P_{\tilde{p}+1}(\gamma)\neq 0$, and 
\begin{flalign}
  &    &&\frac{\wt\mu_{i-1}(u)}{\wt\mu_{i}(u)}=\frac{P_i(u+1)}{P_i(u)} \quad \text{ for }\quad 2\leq i\leq \wt{N} \text{ with }i\neq \tilde{p}+1, &&& \label{HWT:refl.fd.2}\\
\textit{while}&  \hspace{20mm} &&  &&& \nonumber\\ 
  &     &&\frac{\wt\mu_{\tilde{p}}(u)}{\wt\mu_{\tilde{p}+1}(u)}=\frac{P_{\tilde{p}+1}(u+1)}{P_{\tilde{p}+1}(u)}\cdot \frac{\gamma-u}{\gamma+u-\tilde{q}}. &&& \label{HWT:refl.fd.3}
\end{flalign}


\medskip

Now let $V$ be a non-trivial highest weight $X(\mfg_N,\mcG)^{tw}$-module, and let $J$ be the left ideal in $X(\mfg_N,\mcG)^{tw}$ generated by the non-constant coefficients of the series $s_{-i,j}(u)$ and $s_{0j}(u)$ for all $1\leq i,j\leq n$. (Henceforth, all occurrences of $s_{ij}(u)$ with $i=0$ or $j=0$ should be ignored in types C and D.) We define
$V^J$ to be the subspace of $V$ annihilated by $J$:
\begin{equation}
V^J=\{\eta \in V:s_{-i,j}(u)\eta=s_{0j}(u)\eta=0 \text{ for all }1\leq i,j\leq n\}. \label{HWT:V^J}
\end{equation}
If $\xi\in V$ is the highest weight vector, then $\xi$ belongs to $V^J$, so in particular $V^J$ is nonzero. 
The idea of considering the subspace $V^J$ comes from the proof of Proposition 4.2.8 in \cite{Mo5}.  Let $\mbfk$ and $\boldsymbol{\ell}$ be as defined before Corollary \ref{HWT:Cor.restrictions}. The following proposition will be very important to prove our main classification theorems in Section \ref{sec:CT}.

\begin{prop}\label{HWT:refl.Prop.1}
$V^J$ is stable under the action of all operators $s_{ij}(u)$ with $1\leq i,j\leq n$. Moreover, the assignment $\wt{b}_{ij}(u)\mapsto [\pm]s_{ij}(u)$ defines a representation 
of the extended reflection algebra $ \wt{\mcB}(n,\boldsymbol{\ell})$ in $V^J$.
\end{prop}

\begin{proof}

 We begin by showing that $V^J$ is stable under the action of all operators $s_{ij}(u)$ for $1\leq i,j\leq n$. We must show that $s_{-i,j}(u)s_{k\ell}(v)=0 \mod J$ for 
 all $1\leq i,j,k,\ell\leq n$, or equivalently $[s_{-i,j}(u),s_{k\ell}(v)]\equiv 0$, where $``\equiv "$ is used to denote equality of operators on $V^J$. Let us first show $[s_{-i,j}(u),s_{k\ell}(v)]\equiv 0$. This is immediate if $k\neq i,j$ by \eqref{[s,s]}. Consider the case where $i=j=k$. As a consequence of 
relation \eqref{[s,s]} we have
\begin{equation}
 [s_{-i,i}(u),s_{i\ell}(v)]\equiv \left(\frac{1}{u+v}+\frac{1}{(u+v)(u-v-\ka)} \right)\sum_{a=1}^n s_{-i,a}(u)s_{a\ell}(v)-\frac{1}{u-v-\ka}\sum_{a=1}^ns_{-a,i}(u)s_{a\ell}(v). \label{HWT:refl.eq1}
\end{equation}
Computing $s_{-a,i}(u)s_{a\ell}(v)$ for $a\neq i$ we obtain
\begin{align*}
 [s_{-a,i}(u),s_{a\ell}(v)]&\equiv -\frac{1}{u-v-\ka}\sum_{b=1}^ns_{-b,i}(u)s_{b\ell}(v)+\frac{1}{(u+v)(u-v-\ka)}\sum_{b=1}^n s_{-i,b}(u)s_{b\ell}(v)\\
                           &\equiv [s_{-i,i}(u),s_{i\ell}(v)]-\frac{1}{u+v}\sum_{b=1}^n s_{-i,b}(u)s_{b\ell}(v),
\end{align*}
where the last equivalence is a direct consequence of equation \eqref{HWT:refl.eq1}. Substituting the above result back into \eqref{HWT:refl.eq1}, we get
\begin{align*}
 [s_{-i,i}(u),s_{i\ell}(v)] & \equiv \left(\frac{1}{u+v}+\frac{1}{(u+v)(u-v-\ka)} \right)\sum_{a=1}^n s_{-i,a}(u)s_{a\ell}(v)\\
                                   &-\frac{n}{u-v-\ka} [s_{-i,i}(u),s_{i\ell}(v)]+\frac{n-1}{(u+v)(u-v-\ka)}\sum_{b=1}^ns_{-i,b}(u)s_{b\ell}(v)\\
                            & \equiv \left(1+\frac{n}{u-v-\ka} \right)\frac{1}{u+v}\sum_{a=1}^ns_{-i,a}(u)s_{a\ell}(v)-\frac{n}{u-v-\ka}[s_{-i,i}(u),s_{i\ell}(v)],
\end{align*}
which implies that 
\begin{equation}
 [s_{-i,i}(u),s_{i\ell}(v)]\equiv\frac{1}{u+v}\sum_{a=1}^ns_{-i,a}(u)s_{a\ell}(v). \label{HWT:refl.eq2}
\end{equation}
By \eqref{[s,s]}, for all $a\neq i$ and $a\ge 1$, we have the relation 
\begin{equation*}
 s_{-i,a}(u)s_{a\ell}(v)\equiv\frac{1}{u+v}\sum_{b=1}^ns_{-i,b}(u)s_{b\ell}(v)\equiv [s_{-i,i}(u),s_{i\ell}(v)].
\end{equation*}
Substituting this into \eqref{HWT:refl.eq2}, we arrive at
\begin{equation*}
 [s_{-i,i}(u),s_{i\ell}(v)]\equiv\frac{n}{u+v}[s_{-i,i}(u),s_{i\ell}(v)],
\end{equation*}
which allows us to conclude that $[s_{-i,i}(u),s_{i\ell}(v)]\equiv 0$ for all $1\leq i\leq n$. 

Now, let us consider the case $i\neq j$. As a consequence of relation \eqref{s=s}, it is 
enough to consider the case where $j=k$. By \eqref{[s,s]} we have: 
\begin{equation*}
 [s_{-i,j}(u),s_{j\ell}(v)]\equiv \frac{1}{u+v}\sum_{a=1}^n s_{-i,a}(u)s_{a\ell}(v).
\end{equation*}
However, by \eqref{HWT:refl.eq2}, the right hand side of the above is equivalent to $0$, and so we obtain $[s_{-i,j}(u),s_{j\ell}(v)]\equiv 0$ for all $1\leq i,j,\ell,\leq n$ such that $i\neq j$. Thus,
we have shown that $[s_{-i,j}(u),s_{k\ell}(v)]\equiv 0$ for all $1\leq i,j,k,\ell\leq n$. If $N=2n+1$, then we must also show $[s_{0j}(u),s_{k\ell}(v)]\equiv 0$ for all 
$1\leq j,k,\ell\leq n$. This is immediate from \eqref{[s,s]} unless $j=k$, and in this case we obtain
\begin{equation}
 [s_{0j}(u),s_{j\ell}(v)]\equiv \frac{1}{u+v}\sum_{a=1}^n s_{0a}(u)s_{a\ell}(v). \label{HWT:refl.eq3}
\end{equation}
However, the same computation shows that
\begin{equation*}
 s_{0a}(u)s_{a\ell}(v)\equiv \frac{1}{u+v}\sum_{b=1}^n s_{0b}(u)s_{b\ell}(v)\equiv [s_{0j}(u),s_{j\ell}(v)],
\end{equation*}
and so \eqref{HWT:refl.eq3} is equivalent to 
\begin{equation*}
 \left(1-\frac{n}{u+v}\right)[s_{0j}(u),s_{j\ell}(v)]\equiv 0.
\end{equation*}
Therefore, we must have $s_{0j}(u)s_{j\ell}(v)\equiv [s_{0j}(u),s_{j\ell}(v)]\equiv 0$ for all $1\leq j, \ell\leq n$. 
This completes the proof that $V^J$ is stable under the action of all operators $s_{ij}(u)$ with $1\leq i,j\leq n$. 

\medskip 

Next, observe from relation \eqref{[s,s]} that for all $1\leq i,j,k,l\leq n$ we have the following equivalence of operators on $V^J$: 
\begin{align*}
 [s_{ij}(u),s_{kl}(v)]&\equiv \frac{1}{u-v}\left(s_{kj}(u)s_{i\ell}(v)-s_{kj}(v)s_{il}(u)\right)\\
                            &+\frac{1}{u+v}\sum_{a=1}^n\left(\delta_{kj}s_{ia}(u)s_{al}(v)-\delta_{il}s_{ka}(v)s_{aj}(u) \right)\\
                            &-\frac{1}{u^2-v^2}\sum_{a=1}^n\delta_{ij}\left(s_{ka}(u)s_{al}(v)-s_{ka}(v)s_{al}(u) \right).
\end{align*}
Comparing these relations with the defining relations of the reflection algebra $\wt{B}(n,\boldsymbol{\ell})$ implies the second part of the lemma. To explain the appearance of the sign 
$[\pm]$ in the statement of the lemma, notice that the matrix $\mcG^+=\sum_{i,j\geq 1} g_{ij}E_{ij}$ coincides with the constant matrix
$[\pm]{\rm diag}(\veps_1,\ldots,\veps_n)$ associated to $\wt{\mcB}(n,\boldsymbol{\ell})$.
 \end{proof}


\begin{rmk}\label{HWT:Rem.reflection}
 We may view $V$ as a $Y(\mfg_N,\mcG)^{tw}$-module by restricting the action of $X(\mfg_N,\mcG)^{tw}$ to the subalgebra $Y(\mfg_N,\mcG)^{tw}\subset X(\mfg_N,\mcG)^{tw}$. Set
 $\Si^+(u)=\sum_{i,j\geq 1} E_{ij}\otimes \si_{ij}(u)$. Since $\Si(u)\Si(-u)=I_N$, we also have the equivalence of operators $\Si^+(u)\Si^+(-u)\equiv I_n$ on $V^J$. As a consequence of this fact, the above proposition, and the definition of 
 $\mcB(n,\boldsymbol{\ell})$, the assignment $b_{ij}(u)\mapsto [\pm]\si_{ij}(u)$ defines a representation of $\mcB(n,\boldsymbol{\ell})$ in the space $V^J$. 
\end{rmk}


The following simple result will be instrumental in the proof of Proposition \ref{HWT:refl.Prop.2} and also in the proofs of the main results in Section \ref{sec:CT}.
\begin{lemma}\label{HWT:Ext}
 Let $(\lambda(u))_{i\in \mathcal{I}_N}$ be any tuple of formal series with $\lambda_i(u)\in 1+u^{-1}\C[[u^{-1}]]$, and let $\nu(u)$ be any series of the same form. Then:
 \begin{enumerate}
   \item If $N=2n+1$, then there is a unique $(2n+1)$-tuple $\lambda(u)$ extending $ (\lambda(u))_{i\in \mathcal{I}_N}$ with the property that the 
  $X(\mfg_N)$ Verma module $M(\lambda(u))$ is non-trivial. 
  \item If $N=2n$, then for each $k\in \mathcal{I}_N$ there exists a unique $2n$-tuple $\lambda(u)=(\lambda_{-n}(u),\ldots,\lambda_n(u))$ extending 
  $(\lambda(u))_{i\in \mathcal{I}_N}$ with the property that $\lambda_{-k}(u)=\nu(u)$ and the $X(\mfg_N)$ Verma module $M(\lambda(u))$ is non-trivial. 
 \end{enumerate}
\end{lemma}
\begin{proof}
 Suppose first that $N=2n+1$. Recall that the Verma module $M(\lambda(u))$ is non-trivial if and only if the components of $\lambda(u)$ satisfy \eqref{HWT:Ext.non-trivial}. 
 This forces us to define $\lambda_{-1}(u)=\frac{\lambda_0(u-\ka+n)}{\lambda_{1}(u-\ka+n)}\lambda_{0}(u)$, and recursively $\lambda_{-i-1}(u)=\frac{\lambda_i(u-\ka+n-i)}{\lambda_{i+1}(u-\ka+n-i)}\lambda_{-i}(u)$
for each $1\leq i \leq n-1 $. In this way we can associate a unique $(2n+1)$-tuple $\lambda(u)$ to $(\lambda(u))_{i\in \mathcal{I}_N}$ satisfying the claimed properties. 

If instead $N=2n$, then the condition \eqref{HWT:Ext.non-trivial} alone no longer uniquely determines an $N$-tuple $\lambda(u)$ from $(\lambda(u))_{i\in \mathcal{I}_N}$. However, fixing $k\in \mathcal{I}_N$ and 
setting $\lambda_{-k}(u)=\nu(u)$, a simple modification of the argument used in the $N=2n+1$ case shows that the condition \eqref{HWT:Ext.non-trivial} does produce a unique $2n$-tuple with 
the desired properties extending 
$(\lambda_{-k}(u),\lambda_{1}(u),\ldots,\lambda_n(u))$.
\end{proof}

We are now prepared to make precise the sufficient and necessary conditions on the tuple $\mu(u)$ which results in a non-trivial $X(\mfg_N,\mcG)^{tw}$ Verma module $M(\mu(u))$.

\begin{prop}\label{HWT:refl.Prop.2}
Let $\mu(u)=(\mu_1(u),\ldots, \mu_n(u))$ or $\mu(u)=(\mu_0(u),\ldots,\mu_n(u))$ be any tuple of formal series such that the $i$-th component belongs to $g_{ii}+u^{-1}\C[[u^{-1}]]$
and $ u\cdot \wt \mu_0(\ka-u)=(\ka-u)\cdot p_0(u)p(u)^{-1}\wt\mu_0(u)$ (see \eqref{HWT:mu_0(u).0}). Then 
the $X(\mfg_N,\mcG)^{tw}$ Verma module $M(\mu(u))$ is nontrivial if and only if
\begin{equation}
\wt\mu_i(u)\wt\mu_i(-u+n-i)=\wt{\mu}_{i+1}(u)\wt{\mu}_{i+1}(-u+n-i) \label{HWT:nontrivial}
\end{equation}
for all $i\in \mathcal{I}_N\setminus\{n\}$, and where the components of the series $\wt \mu_i(u)$ have been defined in \eqref{HWT:tilde_mu_i(u)}.
\end{prop}

\begin{proof}
Set $V=M(\mu(u))$, and denote the highest vector of $V$ by $1_{\mu(u)}$. Suppose
that $V$ is non-trivial. We will show the components of $\mu(u)$ satisfy \eqref{HWT:nontrivial}, first restricting ourselves to the case where $N=2n$. Since $V$ is a  non-trivial
highest weight module, the subspace $V^J$ is non-zero and, by Proposition \ref{HWT:refl.Prop.1}, admits the structure of a $\wt \mcB(n,\boldsymbol{\ell})$-module.
Consider the submodule $W=\wt \mcB(n,\boldsymbol{\ell})1_{\mu(u)}\subset V^J$. Recall the even central series $\ms{f}(u)$ defined by Proposition \ref{P:B=Z*SB}.
Choose $h(u)\in 1+u^{-1}\C[[u^{-1}]]$ such that $h(u)h(-u)=\ms{f}(u)$ as operators on $V^J$. Then the assignment $\wt{B}(u)\mapsto h^{-1}(u)\wt{B}(u)$ defines an automorphism 
$\wt \nu_{h^{-1}}$ of the algebra $\wt{\mcB}(n,\boldsymbol{\ell})$. We have: 
\begin{equation*}
\wt \nu_{h^{-1}}(\ms{f}(u))=h^{-1}(u)\,h^{-1}(-u)\,\ms{f}(u)\equiv \ms{f}(u)^{-1}\ms{f}(u)=1,
\end{equation*}
where $``\equiv"$ denotes equality of operators on $V^J$. Therefore, twisting the action of $\wt \mcB(n,\ell)$ on $V^J$ by $\wt \nu_{h^{-1}}$,
we get a non-trivial representation $W^{\wt \nu_{h^{-1}}}$ of $\mcB(n,\boldsymbol{\ell})$ with the highest weight $(h(u)^{-1}\mu_1(u),\ldots,h(u)^{-1}\mu_n(u))$. 
In particular, by \eqref{HWT:refl.nontrivial.2} we have 
\begin{equation*}
\wt{\mu}_i(u)\wt{\mu}_i(-u+n-i)=\wt{\mu}_{i+1}(u)\wt{\mu}_{i+1}(-u+n-i)
\end{equation*}
for all $1\leq i\leq n-1$, as desired. 

If $N=2n+1$, the above argument still shows that \eqref{HWT:nontrivial} holds for all $1\leq i\leq n-1$, but does not allow us to conclude that 
\begin{equation}
\wt{\mu}_0(u)\wt{\mu}_0(-u+n)=\wt{\mu}_{1}(u)\wt{\mu}_{1}(-u+n). \label{HWT:nontrivial.1}
\end{equation}
That being said, the same argument as used in \cite{MR} to establish \eqref{HWT:refl.nontrivial.2}
(see the proof of Theorem 4.2) can be applied to show that \eqref{HWT:nontrivial.1} does hold. Let us recall the main steps of this argument. 
For each $0\leq i \leq n$ define $\beta_i(u,v)=\sum_{a=i}^n s_{ia}(u)s_{ai}(v).$ Using $``\equiv"$ to denote equality of operators on $\C1_{\mu(u)}$, we have 
\begin{equation*}
 \beta_{i}(u,v)-\beta_{i}(v,u)=\sum_{a=i}^nA_{ia}(u,v)\equiv  0,
\end{equation*}
where the definition of $A_{ij}(u,v)$ has been given in \eqref{HWT:A_ij(u,v)}, and the second equivalence has been proven in Step 3.1 of the proof of Theorem \ref{HWT:Thm.HWT} for $i>0$, 
and in Step 3.2 of the same proof for $i=0$. As a consequence, we have $\beta_{i}(u,v)\equiv\beta_i(v,u)$ for all $i\geq 0$. 
From \eqref{[s,s]} we obtain 
\begin{equation*}
 \beta_i(u,v)\equiv s_{ii}(u)s_{ii}(v)+ \frac{1}{u-v}\sum_{a=i+1}^n(s_{aa}(u)s_{ii}(v)-s_{aa}(v)s_{ii}(u)) +\frac{1}{u+v}\sum_{a=i+1}^n\left( \beta_i(u,v)- \beta_a(v,u)\right),
\end{equation*}
which is equivalent to 
\begin{equation}
 \left(\frac{u+v-n+i}{u+v}\right)\beta_i(u,v)\equiv s_{ii}(u)s_{ii}(v)+ \frac{1}{u-v}\sum_{a=i+1}^n(s_{aa}(u)s_{ii}(v)-s_{aa}(v)s_{ii}(u))-\frac{1}{u+v}\sum_{a=i+1}^n\beta_a(v,u). \label{HWT:nontrivial.2}
\end{equation}
Subtracting \eqref{HWT:nontrivial.2} with $i=1$ from \eqref{HWT:nontrivial.2} with $i=0$ and rearranging, we obtain 
\begin{align*}
 \frac{u+v-n}{u+v}\left(\beta_0(u,v)-\beta_{1}(u,v)\right) & \equiv s_{00}(u)s_{00}(v)+\frac{1}{u-v}\sum_{a=1}^n(s_{aa}(u)s_{00}(v)-s_{aa}(v)s_{00}(u))\\
                                                                &-s_{11}(u)s_{11}(v)-\frac{1}{u-v}\sum_{a=2}^n(s_{aa}(u)s_{11}(v)-s_{aa}(v)s_{11}(u)).
\end{align*}
Substituting $v\mapsto n-u$, the left hand side becomes the zero operator and, after applying both sides to $1_{\mu(u)}$, we arrive at the relation 
\begin{equation*}
 \mu_{0}(u)\mu_{0}(v)+\frac{1}{2u-n}\sum_{a=1}^n(\mu_{a}(u)\mu_{0}(v)-\mu_{a}(v)\mu_{0}(u))=\mu_{1}(u)\mu_{1}(v)+\frac{1}{2u-n}\sum_{a=2}^n(\mu_{a}(u)\mu_{1}(v)-\mu_{a}(v)\mu_{1}(u)).
\end{equation*}
By expanding equation \eqref{HWT:nontrivial.1} (using the definition of $\wt \mu_i(u)$), we see that it is equivalent to the above relation. This completes the proof that the
components of $\mu(u)$ must satisfy \eqref{HWT:nontrivial.1}.

\medskip 

Conversely, suppose the components of $\mu(u)$ satisfy condition \eqref{HWT:nontrivial} as well as the condition \eqref{HWT:mu_0(u).0} if $N=2n+1$ . 
Let $h(u)$ be the rational function in $u$ defined by
\begin{equation}
 h(u)=\begin{cases}
        \mfrac{\boldsymbol{\ell}-u}{u} &\text{ if } \mcG \text{ is of the first kind},\\[.5em]
        \mfrac{1[\pm]\,c(\boldsymbol{\ell}-u)}{1[\pm]\,cu} &\text{ if } \mcG \text{ is of the second kind}.
       \end{cases} \label{HWT:h(u)}
\end{equation}
Note that $h(u)$ satisfies the relation $h(u)h(\boldsymbol{\ell}-u)=1$. For each $i\in \mathcal{I}_N\setminus\{n\}$ define 
$f_i(u)=(h(u))^{-\delta_{i\mbfk}}\frac{\wt \mu_i(u)}{\wt \mu_{i+1}(u)}$. Then, by \eqref{HWT:nontrivial} and the aforementioned property of $h(u)$, $f_i(u)=f_i(n-u-i)^{-1}$ for all
$i\in \mathcal{I}_{N}\setminus\{n\}$. Thus, for each $i$, there exists $g_i(u)\in 1+u^{-1}\C[[u^{-1}]]$ such that $f_i(u)=g_i(u)g_i(n-u-i)^{-1}$. We will 
use these series to construct a non-trivial $X(\mfg_N)$ Verma module $M(\lambda(u))$ containing an $X(\mfg_N,\mcG)^{tw}$ highest weight module with the highest weight $\mu(u)$. 

If $N=2n$, let $\lambda_n(u),\lambda_{-n}(u)\in 1+u^{-1}\C[[u^{-1}]]$ be any two series satisfying the relation 
\begin{equation}
\mu_n(u)=\begin{cases}
          g_{nn}\,\lambda_n(u-\ka/2)\lambda_{-n}(-u+\ka/2) &\text{ if } \mcG \text{ is of the first kind},\\[.25em]
          \left(\mfrac{1[\pm]\,cu}{1-cu}\right)\lambda_n(u-\ka/2)\lambda_{-n}(-u+\ka/2)  &\text{ if } \mcG \text{ is of the second kind}.
         \end{cases} \label{NT.mu_n}
\end{equation}
For each $1\leq i\leq n-1$ define $\lambda_i(u)\in 1+u^{-1}\C[[u^{-1}]]$ recursively in terms of $\lambda_{i+1}(u)$ by 
\begin{equation*}
 \lambda_i(u-\ka/2)=g_i(u)\lambda_{i+1}(-u-\ka/2+n-i)^{-1}.
\end{equation*}
By Lemma \ref{HWT:Ext}, there is a unique $2n$-tuple $\lambda(u)$ extending $(\lambda_{-n}(u),\lambda_1(u),\ldots,\lambda_n(u))$ with the property that the $X(\mfg_N)$ 
Verma module $M(\lambda(u))$ is non-trivial.

If instead $N=2n+1$, then by assumption $u\cdot g(u)\wt \mu_0(\ka-u)=(\ka-u)\cdot g(\ka-u)\wt\mu_0(u)$, where the series $g(u)$ has been defined in \eqref{HWT:Rem.m_0(u).1}. Therefore there
exists $\lambda_0(u)\in 1+u^{-1}\C[[u^{-1}]]$ such that  $\wt \mu_0(u)=2u\,g(u)\lambda_0(u-\ka/2)\lambda_0(-u+\ka/2).$ For each $0\leq i\leq n-1$, 
define $\lambda_{i+1}(u)\in 1+u^{-1}\C[[u^{-1}]]$ recursively in terms of $\lambda_{i}(u)$ by 
\begin{equation*}
\lambda_{i+1}(-u-\ka/2+n-i)=g_{i}(u)\lambda_{i}(u-\ka/2)^{-1}.
\end{equation*}
Then, by Lemma \ref{HWT:Ext}, there is a unique $(2n+1)$-tuple $\lambda(u)$ extending $(\lambda_i(u))_{i\in \mathcal{I}_N}$ so that $M(\lambda(u))$ is non-trivial.

In either case, we have produced a nontrivial $X(\mfg_N)$ Verma module $M(\lambda(u))$ with the highest weight $\lambda(u)$ whose components satisfy the relations 
\begin{equation*}
\frac{\widetilde{\mu}_i(u)}{\widetilde{\mu}_{i+1}(u)}=h(u)^{\delta_{i\mbfk}}\frac{\lambda_i(u-\ka/2)\lambda_{i+1}(-u-\ka/2+n-i)}{\lambda_{i+1}(u-\ka/2)\lambda_{i}(-u-\ka/2+n-i)}
=h(u)^{\delta_{i\mbfk}}\frac{\lambda_i(u-\ka/2)\lambda_{-i}(-u+\ka/2)}{\lambda_{i+1}(u-\ka/2)\lambda_{-i-1}(-u+\ka/2)}
\end{equation*}
for all $i\in \mathcal{I}_N\setminus\{n\}$, in addition to the relation \eqref{NT.mu_n} if $N=2n$ and $\wt \mu_0(u)=2u\,g(u)\lambda_0(u-\ka/2)\lambda_0(-u+\ka/2)$ if $N=2n+1$.

Hence, the module $X(\mfg_N,\mcG)^{tw}1_{\lambda(u)}\subset M(\lambda(u))$ is a non-trivial highest weight module, and the above relations together with Corollary \ref{HWT:Cor.restrictions}
 imply that the highest weight is equal to $\mu(u)$. Therefore the $X(\mfg_N,\mcG)$ Verma module $M(\mu(u))$ is non-trivial.
 \end{proof}


\bigskip

The following proposition is analogous to Proposition 4.2.8 in \cite{Mo5} and gives a necessary condition for any irreducible highest weight $X(\mfg_N,\mcG)^{tw}$-module to be finite-dimensional. 
\begin{prop}\label{HWT:refl.Prop.3}
 Let the components of $\mu(u)$ satisfy the conditions of Proposition \ref{HWT:refl.Prop.2}, so that the irreducible module $V(\mu(u))$ exists. Suppose further that $V(\mu(u))$ is finite-dimensional. Then:
 \begin{enumerate}
  \item If $(\mfg_N,\mfg_N^\rho)$ is of type BCD0, BI with $q=1$, CI or DIII, then there exist monic polynomials $P_2(u),\ldots,P_{n}(u)$ in $u$ such that $P_i(-u+n-i+2)=P_i(u)$ and 
	\begin{equation}
	   \frac{\wt \mu_{i-1}(u)}{\wt\mu_{i}(u)}=\frac{P_i(u+1)}{P_i(u)} \quad \text{ for all }\quad 2\leq i\leq n.\label{HWT:necessary.1} 
	 \end{equation}
 
  \item If $(\mfg_N,\mfg_N^\rho)$ is of type BDI (excluding the $q=1$ case) or CII, then there exists $\gamma \in \C$ and monic polynomials $P_2(u),\ldots,P_{n}(u)$ in $u$ such that $P_i(-u+n-i+2)=P_i(u)$, 
  $P_{\mbfk+1}(\gamma)\neq 0$ and 
  
  \begin{flalign}
              &   &&\frac{\widetilde{\mu}_{i-1}(u)}{\widetilde{\mu}_{i}(u)}=\frac{P_i(u+1)}{P_i(u)} \quad \text{ for all }\quad 2\leq i\leq n \quad\text{ with }\quad i\neq \mbfk+1,&&& \label{HWT:necessary.2}\\
\textit{while}& \hspace{20mm}  && &&& \nn\\
              &   &&\frac{\widetilde{\mu}_{\mbfk}(u)}{\widetilde{\mu}_{\mbfk+1}(u)}=\frac{P_{\mbfk+1}(u+1)}{P_{\mbfk+1}(u)}\cdot \frac{\gamma-u}{\gamma+u-\boldsymbol{\ell}},\label{HWT:necessary.3}&&&
  \end{flalign}
  where $\mbfk\in \mathcal{I}_N\setminus\{n\}$ is the unique integer such that $g_{\mbfk\mbfk}\neq g_{\mbfk+1,\mbfk+1}$, and $\boldsymbol{\ell}=n-\mbfk$. 
 \end{enumerate}
\end{prop}

\begin{proof}
Denote the highest weight vector of $V(\mu(u))$ by $\xi$ and let $V(\mu(u))^J$ be as in \eqref{HWT:V^J}. Then, by 
Proposition \ref{HWT:refl.Prop.1}, allowing $\wt b_{ij}(u)$ to operate as $[\pm]s_{ij}(u)$ for all $1\leq i,j\leq n$ makes $V(\mu(u))^J$ into a 
$\wt \mcB(n,\boldsymbol{\ell})$-module. Choose $h(u)\in 1+u^{-1}\C[[u^{-1}]]$ such that $h(u)h(-u)=\ms{f}(u)$ as operators on $V^J$ (see \eqref{P:B=Z*SB}). 
Twisting the action of $\wt \mcB(n,\boldsymbol{\ell})$ on $V^J$ by the automorphism $\wt \nu_{h^{-1}}$, we obtain a non-trivial representation of $\mcB(n,\boldsymbol{\ell})$ such that 
the cyclic span $\mcB(n,\boldsymbol{\ell})\xi$ is a finite-dimensional highest weight module with the highest weight $[\pm](h(u)^{-1}\mu_1(u),\ldots,h(u)^{-1}\mu_n(u))$. 
The proposition now follows directly from the equations \eqref{HWT:refl.fd.1} through \eqref{HWT:refl.fd.3} (see also \cite[Thm.~4.6]{MR}).
\end{proof}



\section{Twisted Yangians of small rank and their representations} \label{sec:SR}  
In this section, we use the classification results for finite-dimensional irreducible representations of $Y^{\pm}(2)$ \cite{Mo2} together with the isomorphisms from \cite{GRW} to classify all finite-dimensional irreducible representations of low rank extended twisted Yangians of type BCD0, CI, and DIII.
These low rank classification results will play a crucial role in the proofs of the main results in Section \ref{sec:CT}. We also obtain explicit formulas for evaluation morphisms $\mathrm{ev}:X(\mfg_N,\mfg_N^\rho)^{tw}\onto \mathfrak{U}\mfg_N^\rho$, where 
$\mfg_N=\mfsp_2,\mfso_3$ or $\mfso_4$, and study the corresponding evaluation modules.

\medskip 

The definitions of the twisted Yangians $Y^\pm(N)$ were recalled in Subsection \ref{sec:TY-Ols}. They are the twisted Yangians associated to the symmetric pairs $(\mfgl_N,\mfg_N)$, where 
$\mfg_N=\mfso_{2n}$ or $\mfsp_{2n}$ if $N=2n$, and $\mfg_N=\mfso_{2n+1}$ if $N=2n+1$. In this section we will only be concerned with $Y^{\pm}(N)$ where $N=2$. 
In order to distinguish between the generators of $X(\mfg_N,\mfg_N^\rho)^{tw}$ and $Y^\pm(2)$, we shall follow the convention established in \cite{GRW} and denote the 
generators of $Y^\pm(2)$ by $s^{\circ (r)}_{ij}$, where $i,j\in \{\pm1\}$ and $r\geq0$. These generators are then arranged as the coefficients of the various series' $s_{ij}^\circ(u)$, which 
in turn form the $(i,j)^{th}$ entry of the matrix $S^\circ(u)$. Similarly, the generators of the special twisted Yangian $SY^\pm(2)$ are denoted by $\si^{\circ(r)}_{ij}$, and the
corresponding series and matrix are denoted by $\si^\circ_{ij}(u)$ and $\Si^\circ(u)$, respectively. We shall also denote the $R$-matrix $I-u^{-1}P\in \End\,(\C^2\otimes \C^2) [[u^{-1}]]$ from  $a)$ of \eqref{R(u)}
by $R^\circ(u)$, and the evaluation morphism $s^\circ_{ij}(u)\mapsto \delta_{ij}+F_{ij}\left(u\pm \frac{1}{2}\right)^{-1}$ from Proposition \ref{P:Y-Ols->U} by $\mathrm{ev}^\circ_{\pm}$. We will make use of the following explicit formulas for the Sklyanin determinant ${\rm sdet}\,S^\circ(u)$ (see \cite[Sec.~4]{MNO}):
\begin{align}
 {\rm sdet}\,S^\circ(u)&=\frac{2u+1}{2u\pm1}\left(s^\circ_{-1,-1}(u-1)s^\circ_{-1,-1}(-u)\mp s^\circ_{-1,1}(u-1)s^\circ_{1,-1}(-u) \right)\nonumber\\
                    &=\frac{2u+1}{2u\pm1}\left(s^\circ_{11}(-u)s^\circ_{11}(u-1)\mp s^\circ_{1,-1}(-u)s^\circ_{-1,1}(u-1)\right) . \label{LRR:sdet}
\end{align}

We now recall the classification results for finite-dimensional irreducible representations of the twisted Yangians $Y^\pm(2)$.

\medskip 

 A representation $V$ of $Y^\pm(2)$ is called a \textit{highest weight representation} if there exists a nonzero vector $\xi\in V$ such that $V=Y^\pm(2)\xi$, 
 $s_{-1,1}(u)\xi=0$, and $s^\circ_{11}(u)\xi=\mu^\circ(u)\xi$ for some scalar series $\mu^\circ(u)\in 1+u^{-1}\C[[u^{-1}]]$. As usual, we call $\mu^\circ(u)$ the highest weight of $V$, and the vector $\xi$ the highest weight vector. 

\medskip 

Given $\mu^\circ(u)\in 1+u^{-1}\C[[u^{-1}]]$, the Verma module $M(\mu^\circ(u))$ is defined the same way as for $X(\mfg_N)$ and $X(\mfg_N,\mcG)^{tw}$, and is always non-trivial.
It admits a unique irreducible quotient $V(\mu^\circ(u))$, and any irreducible highest weight module with the highest weight $\mu^\circ(u)$ is isomorphic to $V(\mu^\circ(u))$.
The following classification results are restatements of Theorems 4.4 and 5.4 of \cite{Mo2} (see also Theorems 4.3.3 and 4.4.3 of \cite{Mo5}): 

\medskip

The irreducible $Y^-(2)$-module $V(\mu^\circ(u))$ is finite-dimensional if and only if there exists a monic polynomial $P(u)$ in $u$ such that $P(u)=P(-u+1)$ and 
\begin{equation}
 \frac{\mu^\circ(-u)}{\mu^\circ(u)}=\frac{P(u+1)}{P(u)}. \label{LRR:AII-class}
\end{equation}
In this case, the monic polynomial $P(u)$ is unique. 

On the other hand, the irreducible $Y^+(2)$-module $V(\mu^\circ(u))$ is finite-dimensional if and only if there exists a scalar $\alpha\in \C$ together with a monic 
polynomial $Q(u)$ such that $Q(u)=Q(-u+1)$, $Q(\alpha)\neq 0$, and 
\begin{equation}
 \frac{(1-\frac{1}{2u})\mu^\circ(-u)}{(1+\frac{1}{2u})\mu^\circ(u)}=\frac{Q(u+1)}{Q(u)}\cdot \frac{u-\alpha}{u+\alpha}. \label{LRR:AI-class}
\end{equation}
In this case, the pair $(Q(u),\alpha)$ is unique. 

Let us now briefly recall the isomorphisms from Section 4 of \cite{GRW} which are relevant to our present study, beginning with those concerning the twisted Yangians associated to the pairs 
$(\mfsp_2,\mfsp_2)$ and $(\mfsp_2,\mfgl_1)$. Let $K=E_{11}-E_{-1,-1}\in \End\, \C^2$. Then the mappings 
\begin{alignat}{2}
 &\varphi^\prime_0: X(\mfsp_2,\mfsp_2)^{tw}\to Y^-(2), \quad &&S(u)\mapsto S^\circ(u/2-1/2), \label{LRR:C0->AII}\\
 &\varphi^\prime_1: X(\mfsp_2,\mfgl_1)^{tw}\to Y^+(2), \quad &&S(u)\mapsto S^\circ(u/2-1/2)K, \label{LRR:CI->AI}
\end{alignat}
are isomorphisms of algebras. Moreover, they induce algebra isomorphisms $Y(\mfsp_2,\mfsp_2)^{tw}\cong SY^-(2)$ and $Y(\mfsp_2,\mfgl_1)^{tw}\cong SY^+(2)$. 

Consider now the symmetric pair $(\mfso_3,\mfso_3)$ of type B0. Let the standard basis of $\C^2$ be given by the vectors $e_{-1}$ and $e_1$ and let $V$ be the three-dimensional subspace of $\C^2\otimes \C^2$ spanned by the elements 
$v_{-1}=e_{-1}\otimes e_{-1}$, $v_0=\tfrac{1}{\sqrt{2}}(e_{-1}\otimes e_1+e_1\otimes e_{-1})$, and $v_1=-e_1\otimes e_1$. 
We may identify $V$ with $\C^3$ by regarding $\{v_{-1},v_0,v_1\}$ as the canonical basis of $\C^3$. In this way we may consider $S(u)$ as an element of  
$\End\,V\otimes X(\mfso_3,\mfso_3)^{tw}[[u^{-1}]]$. Moreover, the operator $\tfrac{1}{2}R^\circ(-1)\in \End\,(\C^2\otimes \C^2)$ is a projection of $\C^2\otimes \C^2$ onto the subspace $V$ and the mapping 
\begin{equation}
 \varphi_0: X(\mfso_3,\mfso_3)^{tw}\to Y^-(2), \quad S(u)\mapsto \tfrac{1}{2}R^\circ(-1)S_1^\circ(2u-1)R^\circ(-4u+1)^{t_-}S_2^\circ(2u) \label{LRR:B0->AII}
\end{equation}
is an algebra isomorphism whose restriction to the subalgebra $Y(\mfso_3,\mfso_3)^{tw}$ induces an isomorphism between $Y(\mfso_3,\mfso_3)^{tw}$  and $SY^-(2)$.

Lastly, we recall the isomorphisms for the twisted Yangians associated to the symmetric pairs $(\mfso_4,\mfgl_2)$ and $(\mfso_4,\mfso_4)$. These isomorphisms involve
the tensor products $SY^+(2)\otimes Y^-(2)$ and $SY^-(2)\otimes Y^-(2)$. We shall denote the corresponding generating series of $SY^\pm(2)$ by $\sigma^\circ_{ij}(u)$ and those of $Y^-(2)$
by $s^\bullet_{ij}(u)$, where in both cases $i,j\in \{\pm1\}$. These are then arranged into the matrices $\Si^\circ(u)$ and $S^\bullet(u)$, respectively. Let 
$V=\C^2\otimes \C^2$ with ordered basis given by $v_{-2}=e_{-1}\otimes e_{-1}$, $v_{-1}=e_{-1}\otimes e_{1}$, $v_1=e_1\otimes e_{-1}$ and $v_2=-e_1\otimes e_1$. 
By identifying $V$ with $\C^4$ equipped with canonical basis $\{v_{-2},v_{-1},v_1,v_2\}$, we can consider $S(u)$ as an element of $\End\,V\otimes X(\mfso_4,\mfso_4^\rho)^{tw}[[u^{-1}]]$, where $\mfso_4^\rho$ is either $\mfgl_2$ or $\mfso_4$. The following
maps are isomorphisms of algebras: 
\begin{alignat}{2}
 &\chi^{(1)}_0: X(\mfso_4,\mfgl_2)^{tw}\to SY^+(2)\otimes Y^-(2), \quad &&S(u)\mapsto \Si^\circ(u-1/2)K_1S^\bullet(u-1/2),\label{LRR:DIII->AII+AI}\\
 &\chi^{(1)}_1: X(\mfso_4,\mfso_4)^{tw}\to SY^-(2)\otimes Y^-(2), \quad &&S(u)\mapsto \Si^\circ(u-1/2)S^\bullet(u-1/2).\label{LRR:D0->AII+AII}
\end{alignat}
Their restrictions to the subalgebras $Y(\mfso_4,\mfgl_2)^{tw}$ and $Y(\mfso_4,\mfso_4)^{tw}$ yield isomorphisms $Y(\mfso_4,\mfgl_2)^{tw}\cong SY^+(2)\otimes SY^-(2)$ and 
$Y(\mfso_4,\mfso_4)^{tw}\cong SY^-(2)\otimes SY^-(2)$, respectively. The isomorphisms $\chi^{(1)}_0$ and 
$\chi^{(1)}_1$ can be obtained from the embeddings 
\begin{alignat}{2}
 &\wt \chi^{(1)}_0: X(\mfso_4,\mfgl_2)^{tw}\to Y^+(2)\otimes Y^-(2), \quad &&S(u)\mapsto S^\circ(u-1/2)K_1S^\bullet(u-1/2),\label{LRR:DIII.into.AII+AI}\\
 &\wt \chi^{(1)}_1: X(\mfso_4,\mfso_4)^{tw}\to Y^-(2)\otimes Y^-(2), \quad &&S(u)\mapsto S^\circ(u-1/2)S^\bullet(u-1/2),\label{LRR:D0.into.AII+AII}
\end{alignat}
by composing with the epimorphisms $\mathrm{Pr}_+\otimes 1$ or $\mathrm{Pr}_-\otimes 1$, respectively, where $\mathrm{Pr}_\pm$ is the natural projection $\mathrm{Pr}_\pm:Y^\pm(2)\onto SY^\pm(2)$. 

\medskip 


We now turn our attention to the finite-dimensional representation theory of the low rank twisted Yangians of type B-C-D.

\subsection{Twisted Yangians for the symmetric pairs \texorpdfstring{$(\mfsp_2,\mfsp_2^\rho)$}{}}
We begin with the classification of the finite-dimensional irreducible representations of the extended twisted Yangian $X(\mfsp_2,\mfsp_2^\rho)^{tw}$, where~$\mfsp_2^\rho=\mfgl_1$~or~$\mfsp_2$.
\begin{prop}\label{LRR:CI-C0.Class}
Let $\mu(u) \in 1 + u^{-1}\C[[u^{-1}]]$. The irreducible $X(\mathfrak{sp}_2,\mfsp_2^\rho)^{tw}$-module $V(\mu(u))$ is finite-dimensional if and only if there exists a monic polynomial $P(u)$ in $u$, in addition to a scalar $\gamma\in \C$ with $P(\gamma)\neq 0$ if $\mfsp_2^\rho=\mfgl_1$, such that $P(u)=P(-u+4)$ and  
 \begin{alignat}{2}
  &\frac{\wt\mu(2-u)}{\wt\mu(u)}=\frac{P(u+2)}{P(u)}\cdot \frac{2-u}{u}\quad &&\textit{if }\quad \mfsp_2^\rho=\mfsp_2,\label{LRR:C0.1}\\
  &\frac{\wt \mu(2-u)}{\wt\mu(u)}=\frac{P(u+2)}{P(u)}\cdot \frac{\gamma-u}{\gamma+u-2}\quad &&\textit{if }\quad \mfsp_2^\rho=\mfgl_1.\label{LRR:CI.Class.eq}
  \end{alignat}
 Moreover, when they exist, the polynomial $P(u)$ and the scalar $\gamma$ are uniquely determined. 
\end{prop}
\begin{proof}
 
The proposition follows from the classification results \eqref{LRR:AII-class} and \eqref{LRR:AI-class}, together with the existence of the explicit isomorphisms
$\varphi_0^\prime$ and  $\varphi_1^\prime$ given by \eqref{LRR:C0->AII} and \eqref{LRR:CI->AI}, respectively. Due to the similarities between these isomorphisms, 
we will only include a detailed proof for the case $\mfsp_2^\rho=\mfgl_1$.

\medskip 

The isomorphism $\varphi_1^\prime$ from \eqref{LRR:CI->AI} defines an equivalence between the highest weight representations 
of $X(\mfsp_2,\mfgl_1)^{tw}$ and those of $Y^+(2)$. To see this, given a series $\mu^\circ(u)\in 1+u^{-1}\C[[u^{-1}]]$, let $V(\mu^\circ(u))$ denote the irreducible $Y^+(2)$-module with the highest weight $\mu^\circ(u)$. 
Then, viewed as a $Y^+(2)$-module via $\varphi_1^\prime$, the irreducible $X(\mathfrak{sp}_2,\mathfrak{gl}_1)^{tw}$-module $V(\mu(u))$ is isomorphic to $V(\mu^\circ(u))$ with $\mu^\circ(u)=\mu(2u+1)$.  Indeed, if $\xi\in V(\mu(u))$ 
is the highest weight vector, then we have $s^\circ_{-1,1}(u)\cdot \xi={\varphi_1^\prime}^{-1}(s^\circ_{-1,1}(u))\xi=s_{-1,1}(2u+1)\xi=0$ and 
\begin{equation*}
s^\circ_{11}(u)\cdot \xi=\varphi^{-1}(s^\circ_{11}(u))\,\xi=s_{11}(2u+1)\,\xi=\mu(2u+1)\,\xi.
\end{equation*}
 The $Y^+(2)$-module $V(\mu^\circ(u))$ is finite-dimensional if and only there exists a (unique) pair 
 $(Q(u),\alpha)$, where $Q(u)$ is a monic polynomial in $u$ with $Q(u)=Q(-u+1)$, $\alpha\in \C$ is such that $Q(\alpha)\neq 0$, and \eqref{LRR:AI-class} holds.

Rewriting \eqref{LRR:AI-class} using $\mu(u)$ and substituting $u\mapsto \frac{u-1}{2}$ we obtain the expression 
\begin{equation}
 \frac{\left(2-u\right)\mu(2-u)}{u\mu(u)}=\frac{Q(\frac{u+1}{2})}{Q(\frac{u-1}{2})}\cdot\frac{2\alpha-(u-1)}{(u-1)+2\alpha}. \label{LRR:CI.1}
\end{equation}
Set $P(u)=2^{\deg  Q(u)}Q\left(\frac{u-1}{2}\right)$ and $\gamma=2\alpha+1$, so that $P(u)$ is a monic polynomial with $P(u)= P(-u+4)$ (since $Q(u) = Q(-u+1)$) 
and $P(\gamma)=2^{\deg  Q(u)}Q(\alpha)\neq 0$. Then by \eqref{LRR:CI.1}, we have shown that $V(\mu(u))$ is finite-dimensional if and only if there exists
a pair $(P(u),\gamma)$ as in the statement of the proposition, satisfying 
\begin{equation*}
\frac{\wt \mu(2-u)}{\wt \mu(u)}=\frac{P(u+2)}{P(u)}\cdot \frac{\gamma-u}{\gamma+u-2}.
\end{equation*}
The uniqueness of the pair $(P(u),\gamma)$ follows immediately from the uniqueness of $(Q(u),\alpha)$. 
\end{proof}

Composing the isomorphisms \eqref{LRR:C0->AII} and \eqref{LRR:CI->AI} with the evaluation morphisms $\mathrm{ev}^\circ_\pm$ from Proposition \ref{P:Y-Ols->U}, we obtain evaluation morphisms for $X(\mfsp_2,\mfsp_2^\rho)^{tw}$. 

\begin{prop}\label{LRR:Prop.ev_CI-C0}
 Let $F^{\prime\rho}=\sum_{i,j=\pm1} E_{ij}\otimes F^{\prime \rho}_{ij}$ where $F^{\prime \rho}_{ij}=(g_{ii}+g_{jj})F_{ij}$  (see \eqref{P:F->TY}). The mappings
 \begin{align}
 & \mathrm{ev}_0:X(\mfsp_2,\mfsp_2)^{tw}\to \mathfrak{U}\mfsp_2,\qu S(u)\mapsto I+F^{\prime\rho} (u-2)^{-1}, \label{LRR:ev.C0}\\
 & \mathrm{ev}_1:X(\mfsp_2,\mfgl_1)^{tw}\to \mathfrak{U}\mfgl_1,\qu\; S(u)\mapsto \mcG+F^{\prime\rho} u^{-1}, \label{LRR:ev.CI}
 \end{align}
are surjective algebra homomorphisms. 
\end{prop}

For any $\mu\in \C$, let $V_\rho(\mu)$ denote the irreducible highest weight representation of $\mfsp_2^\rho$ with the highest weight $\mu$. The pull-back of $V_\rho(\mu)$ via $\mathrm{ev}_0$ and $\mathrm{ev}_1$ is an irreducible $X(\mfsp_2,\mfsp_2^\rho)^{tw}$-module. 

\begin{crl}
 Given $\mu\in \C$, $V_\rho(\mu)$ is isomorphic to the irreducible $X(\mfsp_2,\mfsp_2^\rho)^{tw}$-module $V(\mu(u))$ with  
 \begin{equation}
  \mu(u)=1+(2\mu)u^{-1}\qu\text{ if }\qu \mfsp_2^\rho=\mfgl_1, \qu\text{ and }\qu \mu(u)=1+2\mu(u-2)^{-1} \qu\text{ if }\qu \mfsp_2^\rho=\mfsp_2.
 \end{equation}
\end{crl}

\subsection{Twisted Yangians for the symmetric pairs \texorpdfstring{$(\mfso_4,\mfso_4^\rho)$}{}}
We aim to establish results for $X(\mfso_4,\mfso_4^\rho)^{tw}$, where $\mfso_4^\rho=\mfgl_2$ or $\mfso_4$, which 
are analogous to those obtained in the previous subsection for  $X(\mfsp_2,\mfsp_2^\rho)^{tw}$. 

\begin{prop}\label{LRR:DIII-D0.Class}
Let the components of $\mu(u)=(\mu_1(u),\mu_2(u))$ satisfy the condition \eqref{HWT:nontrivial} so that the irreducible $X(\mfso_4,\mfso_4^\rho)^{tw}$-module 
$V(\mu(u))$ exists. Then $V(\mu(u))$ is finite-dimensional if and only if there exist monic polynomials $P(u)$ and $Q(u)$ in $u$, 
in addition to a scalar $\gamma\in \C$ with $Q(\gamma)\neq 0$ if $\mfso_4^\rho=\mfgl_2$, such that $P(u)=P(-u+2)$, $Q(u)=Q(-u+2)$ and 
\begin{equation}
 \frac{\widetilde{\mu}_1(u)}{\widetilde{\mu}_2(u)}=\frac{P(u+1)}{P(u)}, \label{LRR:DIII.2}
\end{equation}
while
\begin{alignat}{2}
  &\frac{\wt{\mu}_1(1-u)}{\wt{\mu}_2(u)}=\frac{Q(u+1)}{Q(u)}\cdot\frac{1-u}{u}  \quad &&\textit{if }\quad \mfso_4^\rho=\mfso_4,\label{LRR:D0.3}\\
  &\frac{\widetilde{\mu}_1(1-u)}{\widetilde{\mu}_2(u)}=\frac{Q(u+1)}{Q(u)}\cdot \frac{\gamma-u}{\gamma+u-1} \quad &&\textit{if }\quad \mfso_4^\rho=\mfgl_2. \label{LRR:DIII.3}
\end{alignat}
Moreover, when they exist, the pair $(Q(u),P(u))$ and the scalar $\gamma$ are uniquely determined. 
\end{prop}
\begin{proof}
The proposition follows from the existence of the explicit isomorphisms $\chi_0^{(1)}$ and $\chi_1^{(1)}$ given in \eqref{LRR:DIII->AII+AI} and \eqref{LRR:D0->AII+AII}, respectively, together with the classification results \eqref{LRR:AII-class} and \eqref{LRR:AI-class}.
We will give details of the proof only for the case $\mfso_4^\rho=\mfgl_2$. 

\medskip

It is a general fact that any simple finite-dimensional module over a tensor product $A \otimes B$  of two associative unital $\C$-algebras $A$ and $B$ is of the form $M_A \otimes M_B$, where $M_A$ (resp.~$M_B$) is a simple, finite-dimensional module over $A$ (resp. over $B$): see Theorem 3.10.2 in \cite{EGH+}. We show more precisely that the $X(\mathfrak{so}_4,\mathfrak{gl}_2)^{tw}$-module $V(\mu(u))$, viewed as a $SY^+(2)\otimes Y^-(2)$-module via the isomorphism $\chi^{(1)}_0$, is isomorphic to $V(\lambda^\circ(u))\otimes V(\lambda^\bullet(u))$, where $V(\lambda^\circ(u))$ denotes the irreducible highest weight $SY^+(2)$-module
 of weight $\lambda^\circ(u)$, $V(\lambda^\bullet(u))$ denotes the irreducible highest weight $Y^-(2)$-module of weight $\lambda^\bullet(u)$, and where the series $\lambda^\circ(u)$ and $\lambda^\bullet(u)$ are 
 completely determined by the two relations
 \begin{align}
  \lambda^\bullet(u)\lambda^\bullet(u-1)-\tfrac{1}{2u}&\left(\lambda^\bullet(u)-\lambda^\bullet(-u) \right)\lambda^\bullet(u-1)=\mu_1(-u+1/2)\mu_2(u-1/2), \label{LRR:DIII.4}\\
  &\lambda^\circ(u)=\mu_2(u+1/2)\lambda^\bullet(u)^{-1} \label{LRR:DIII.5}.
\end{align}
Equivalently, $\lambda^\circ(u)$ and $\lambda^\bullet(u)$ are completely determined by the two equations
\begin{equation}
 \widetilde{\mu}_2(u)=2u\cdot\lambda^\circ(u-1/2)\lambda^\bullet(u-1/2) \quad \mathrm{and}\quad \widetilde{\mu}_1(u)=2u\cdot\lambda^\circ(u-1/2)\lambda^\bullet(-u+1/2) \label{LRR:DIII.6}.
\end{equation}
We will need the following explicit formulas for the images of the generators $s_{ij}(u)$ under the isomorphism $\chi^{(1)}_0$ from \eqref{LRR:DIII->AII+AI}:
\begin{equation}
\begin{aligned}
 s_{-2,-2}(u)&\mapsto -\si^\circ_{-1,-1}(\wt{u})s^\bullet_{-1,-1}(\wt{u}), \\
 s_{-2,-1}(u)&\mapsto -\si^\circ_{-1,-1}(\wt{u})s^\bullet_{-1,1}(\wt{u}), \\
 s_{-2,1}(u)&\mapsto \si^\circ_{-1,1}(\wt{u})s^\bullet_{-1,-1}(\wt{u}), \\
 s_{-2,2}(u)&\mapsto -\si^\circ_{-1,1}(\wt{u})s^\bullet_{-1,1}(\wt{u}),  \\
 s_{-1,-2}(u)&\mapsto -\si^\circ_{-1,-1}(\wt{u})s^\bullet_{1,-1}(\wt{u}), \\
 s_{-1,-1}(u)&\mapsto -\si^\circ_{-1,-1}(\wt{u})s^\bullet_{11}(\wt{u}), \\
 s_{-1,1}(u)&\mapsto \si^\circ_{-1,1}(\wt{u})s^\bullet_{1,-1}(\wt{u}), \\
 s_{-1,2}(u)&\mapsto -\si^\circ_{-1,1}(\wt{u})s^\bullet_{11}(\wt{u}),
\end{aligned}
\quad
\begin{aligned}
 s_{1,-2}(u)&\mapsto -\si^\circ_{1,-1}(\wt{u})s^\bullet_{-1,-1}(\wt{u}), \\
 s_{1,-1}(u)&\mapsto -\si^\circ_{1,-1}(\wt{u})s^\bullet_{-1,1}(\wt{u}),\\
 s_{11}(u)&\mapsto \si^\circ_{11}(\wt{u})s^\bullet_{-1,-1}(\wt{u}), \\
 s_{12}(u)&\mapsto -\si^\circ_{11}(\wt{u})s^\bullet_{-1,1}(\wt{u}), \\
 s_{2,-2}(u)&\mapsto \si^\circ_{1,-1}(\wt{u})s^\bullet_{1,-1}(\wt{u}),\\
 s_{2,-1}(u)&\mapsto \si^\circ_{1,-1}(\wt{u})s^\bullet_{11}(\wt{u}), \\
 s_{21}(u)&\mapsto -\si^\circ_{11}(\wt{u})s^\bullet_{1,-1}(\wt{u}),\\
 s_{22}(u)&\mapsto \si^\circ_{11}(\wt{u})s^\bullet_{11}(\wt{u}),
\end{aligned} \label{LRR:DIII.explicit}
\end{equation}
where $\wt{u}=u-1/2$. It is explained how to obtain these formulas from the assignment \eqref{LRR:DIII->AII+AI} at the end of the proof of Proposition 4.8 in \cite{GRW}.
These formulas together with the expression \eqref{LRR:sdet} and the fact that ${\rm sdet}\,\Si^\circ(u)=1$ give
\begin{align*}
 \chi^{(1)}_0&\left(s_{11}(-\wt{u})s_{22}(\wt{u})-s_{1,-2}(-\wt{u})s_{-1,2}(\wt{u}) \right)\\
                   &=\left(\si^\circ_{11}(-u)\si^\circ_{11}(u-1)-\si^\circ_{1,-1}(-u)\si^\circ_{-1,1}(u-1)\right)s^\bullet_{-1,-1}(-u)s^\bullet_{11}(u-1)\\
                   &=s^\bullet_{-1,-1}(-u)s^\bullet_{11}(u-1).
\end{align*}
Letting $\xi\in V(\mu(u))$ denote the highest weight vector, this gives $s^\bullet_{-1,-1}(-u)s^\bullet_{11}(u-1)\xi=\mu_1(-\wt{u})\mu_2(\wt{u})\xi.$
Employing the defining symmetry relation of $Y^-(2)$, we can rewrite this as 
\begin{equation*}
 \left(s^\bullet_{11}(u)-\tfrac{1}{2u}\left(s^\bullet_{11}(u)-s^\bullet_{11}(-u)\right) \right)s^\bullet_{11}(u-1)\xi=\mu_1(-\wt{u})\mu_2(\wt{u})\xi.
\end{equation*}
By induction on the coefficients $s^{\bullet(r)}_{11}$ of $s^\bullet_{11}(u)$, this implies that there exists $\lambda^\bullet(u)\in 1+u^{-1}\C[[u^{-1}]]$ such that
$s^\bullet_{11}(u)\xi=\lambda^\bullet(u)\xi,$
and $\lambda^\bullet(u)$ is determined by \eqref{LRR:DIII.4}. Again appealing to the formulas \eqref{LRR:DIII.explicit}, we have $\chi^{(1)}_0(s_{22}(u))=\si^\circ_{11}(\wt{u})s^\bullet_{11}(\wt{u})$, which implies that $\xi$ is an eigenvector for the action of $\si^\circ_{11}(u)$ with weight $\lambda^\circ(u)$ determined by the relation \eqref{LRR:DIII.5}. Notice that it now follows
immediately from the explicit formulas \eqref{LRR:DIII.explicit}  that $\si^\circ_{-1,1}(u)\xi=s^\bullet_{-1,1}(u)\xi=0$. Conversely, any vector $\eta$ with the property that $\si^\circ_{-1,1}(u)\eta=s^\bullet_{-1,1}(u)\eta=0$ and which is a weight vector for $s^\bullet_{ii}(u)$ must be a highest weight vector of the $X(\mathfrak{so}_4,\mathfrak{gl}_2)^{tw}$-module $V(\mu(u))$ by \eqref{LRR:DIII.explicit}, hence a scalar multiple of $\xi$. Thus, by the irreducibility of 
$V(\mu(u))$ we can conclude that 
$$V(\mu(u))\cong V(\lambda^\circ(u))\otimes V(\lambda^\bullet(u)).$$
To see that \eqref{LRR:DIII.4} and \eqref{LRR:DIII.5} are equivalent to the relations given in equation \eqref{LRR:DIII.6}, we observe first that relation \eqref{LRR:DIII.5} is clearly equivalent
to $\widetilde{\mu}_2(u)=2u\cdot\lambda^\circ(u-1/2)\lambda^\bullet(u-1/2)$. Notice also that we may rewrite \eqref{LRR:DIII.4} as 
\begin{equation}
 \lambda^\bullet(u)\lambda^\bullet(u-1)-\tfrac{1}{2u}\left(\lambda^\bullet(u)-\lambda^\bullet(-u) \right)\lambda^\bullet(u-1)=\mu_1(-u+1/2)\lambda^\circ(u-1)\lambda^\bullet(u-1).\label{LRR:DIII.7}
\end{equation}
Since ${\rm sdet}\, \Si^\circ(u)=1$, by formula \eqref{LRR:sdet}, we also have
\begin{equation*}
 1\cdot\xi=\left(\si^\circ_{11}(-u)\si^\circ_{11}(u-1)-\si^\circ_{1,-1}(-u)\si^\circ_{-1,1}(u-1)\right)\xi=\lambda^\circ(-u)\lambda^\circ(u-1)\xi,
\end{equation*}
and thus $\lambda^\circ(-u)^{-1}=\lambda^\circ(u-1)$. Using this, we may rewrite \eqref{LRR:DIII.7} as 
\begin{equation*}
 \lambda^\bullet(-u+1/2)\lambda^\circ(u-1/2)-\tfrac{1}{1-2u}\left(\lambda^\bullet(-u+1/2)\lambda^\circ(u-1/2)-\lambda^\bullet(u-1/2)\lambda^\circ(u-1/2)\right)=\mu_1(u),
\end{equation*}
which is equivalent to 
\begin{equation*}
 2u\cdot \lambda^\bullet(-u+1/2)\lambda^\circ(u-1/2)=(2u-1)\mu_1(u)+\mu_2(u)=\widetilde{\mu}_1(u).
\end{equation*}

As a consequence of the isomorphism of $SY^+(2)\otimes Y^-(2)$-modules $V(\mu(u))\cong V(\lambda^\circ(u))\otimes V(\lambda^\bullet(u))$, we can deduce exactly when 
$V(\mu(u))$ is finite-dimensional. Indeed, due to the classification results \eqref{LRR:AII-class} and \eqref{LRR:AI-class}, $V(\lambda^\circ(u))\otimes V(\lambda^\bullet(u))$
is finite-dimensional if and only if there exists $\alpha\in \C$ together with monic polynomials $P'(u)$, $Q'(u)$ such that $P'(u)=P'(-u+1)$, $Q'(u)=Q'(-u+1)$, $Q'(\alpha)\neq 0$ and the following equations hold: 
\begin{equation}
 \frac{\left(1-\tfrac{1}{2u}\right)\lambda^\circ(-u)}{\left(1+\tfrac{1}{2u}\right)\lambda^\circ(u)}=\frac{Q'(u+1)}{Q'(u)}\cdot \frac{u-\alpha}{u+\alpha}, \quad
 \frac{\lambda^\bullet(-u)}{\lambda^\bullet(u)}=\frac{P'(u+1)}{P'(u)}. \label{LRR:DIII.8}
\end{equation}
In this case the triple $(Q'(u),P'(u),\gamma)$ is unique. Since 
\begin{equation*}
 \frac{\widetilde{\mu}_1(u)}{\widetilde{\mu}_2(u)}=\frac{2u\cdot\lambda^\circ(u-1/2)\lambda^\bullet(-u+1/2)}{2u\cdot\lambda^\circ(u-1/2)\lambda^\bullet(u-1/2)}=\frac{\lambda^\bullet(-u+1/2)}{\lambda^\bullet(u-1/2)},
\end{equation*}
the second equation in \eqref{LRR:DIII.8} is equivalent to 
\begin{equation*}
 \frac{\widetilde{\mu}_1(u)}{\widetilde{\mu}_2(u)}=\frac{P'(u+1/2)}{P'(u-1/2)}.
\end{equation*}
Setting $P(u)=P^\prime(u-1/2)$, we have $P(u)=P(-u+2)$ and the above equality becomes that given in \eqref{LRR:DIII.2}. Similarly, since 
\begin{equation*}
 \frac{\widetilde{\mu}_1(1-u)}{\widetilde{\mu}_2(u)}=\frac{2(1-u)\lambda^\circ(-u+1/2)\lambda^\bullet(u-1/2)}{2u\lambda^\circ(u-1/2)\lambda^\bullet(u-1/2)}=\frac{1-u}{u}\cdot \frac{\lambda^\circ(-u+1/2)}{\lambda^\circ(u-1/2)},
\end{equation*}
we may rewrite the first equation in \eqref{LRR:DIII.8} as
\begin{align*}
 \frac{\widetilde{\mu}_1(1-u)}{\widetilde{\mu}_2(u)}&=\frac{1-u}{u}\cdot \frac{Q'(u-1/2+1)}{Q'(u-1/2)}\cdot \frac{u-\frac{1}{2}-\alpha}{u-\frac{1}{2}+\alpha}\cdot\frac{u}{u-1}\\
                                                            &=\frac{Q'(u-1/2+1)}{Q'(u-1/2)}\cdot \frac{\alpha+\frac{1}{2}-u}{u-\frac{1}{2}+\alpha}.
\end{align*}

Setting $\gamma=\alpha+1/2$ and $Q(u)=Q'(u-1/2)$, we have $Q(\gamma)\neq 0$, $Q(u)=Q(-u+2)$, and the above equality becomes equivalent to \eqref{LRR:DIII.3}. Finally,
we note that the uniqueness of the triple $(Q(u),P(u),\gamma)$ is immediate from the uniqueness of $(Q'(u),P'(u),\alpha)$. 
\end{proof}

We now turn to the construction of evaluation morphisms $X(\mfso_4,\mfso_4^\rho)^{tw}\onto \mathfrak{U}\mfso_4^\rho$. The enveloping algebra
$\mathfrak{U}\mfso_4^\rho$ is generated by the elements $F^{\prime\rho}_{ij}=(g_{ii}+g_{jj})F_{ij}$ for $-2\leq i,j\leq 2$ (see Proposition \ref{P:F->TY}). Let 
$\Omega_\rho$ be the Casimir element of $\mathfrak{U}\mfso_4$ if $\mfso_4^\rho=\mfso_4$, or of $\mathfrak{U}\mfsl_2$ if $\mfso_4^\rho=\mfgl_2$, defined by
\begin{equation*}
\Omega_\rho=\begin{cases}
              F_{11}^2+F_{22}^2-2F_{22}+2F_{21}F_{12}+2F_{2,-1}F_{-1,2} &\text{ if }\mfso_4^\rho=\mfso_4,\\
              \tfrac{1}{2}(F_{22}-F_{11})^2+F_{12}F_{21}+F_{21}F_{12} &\text{ if }\mfso_4^\rho=\mfgl_2. 
             \end{cases}
\end{equation*}
If $\mfso_4^\rho=\mfgl_2$, define the auxiliary central element $z\in \mathfrak{U}\mfgl_2$ by 
\begin{equation*}
 z=F^2_{11}+F^2_{22}+F_{12}F_{21}+F_{21}F_{12}=\Omega_\rho+\tfrac{1}{2}(F_{11}+F_{22})^2. 
\end{equation*}
In the following proposition it will be convenient to denote the Casimir element $\Omega_\rho$ corresponding to the case $\mfso_4^\rho=\mfso_4$ simply by $\Omega$. 

\begin{prop}\label{LRR:Prop.ev_DIII-D0}
 Let $F^{\prime\rho}=\sum_{i,j=-2}^2 E_{ij}\otimes F^{\prime \rho}_{ij}\in \End\,\C^4\otimes \mathfrak{U}\mfso_4^\rho$. The mappings
 \begin{align}
 & \mathrm{ev}_0:X(\mfso_4,\mfso_4)^{tw}\to \mathfrak{U}\mfso_4,\qu S(u)\mapsto I+\frac{F^{\prime\rho}}{u-1}+\frac{(F^{\prime\rho})^2-2F^{\prime\rho}-2\Omega\cdot I}{2(u-1)^2}, \label{LRR:ev.D0}\\
 & \mathrm{ev}_1:X(\mfso_4,\mfgl_2)^{tw}\to \mathfrak{U}\mfgl_2,\qu\; S(u)\mapsto \mcG+\frac{F^{\prime\rho}}{u}+\mcG\frac{(F^{\prime\rho})^2-2z\cdot I}{2u(u-1)}, \label{LRR:ev.DIII}
 \end{align}
are surjective algebra homomorphisms. 
\end{prop}

\begin{proof}
 Suppose first that $\mfso_4^\rho=\mfgl_2$. 
 Consider the Lie algebra $\mfso_2\oplus\mfsp_2$. Denote the generators of $\mfso_2$ in this direct sum by $F^\circ_{ij}$, and those of $\mfsp_2$ by $F^\bullet_{ij}$, 
 where $i,j\in \{\pm 1\}$. The Lie algebra
 $\mfso_2$ is one-dimensional with basis $F^\circ_{11}$, while $\mfsp_2$ is three-dimensional with basis $\{F^\bullet_{1,1},F^\bullet_{-1,1},F^\bullet_{1,-1}\}$. 
 Let $\Phi$ be the isomorphism $\mfso_2\oplus \mfsp_2\iso \mfgl_2$ given by
 \begin{equation*}
  F^\circ_{11}\mapsto F_{11}+F_{22},\quad
 F^\bullet_{11}\mapsto F_{22}-F_{11}, \quad F^\bullet_{-1,1}\mapsto -2F_{12},\quad F^\bullet_{1,-1}\mapsto -2F_{21}.
\end{equation*}
$\Phi$ induces an isomorphism $\wh\Phi:\mathfrak{U}\mfso_2\otimes \mathfrak{U}\mfsp_2\iso \mathfrak{U}\mfgl_2$. Therefore, the composition  $\wh\Phi\circ (\mathrm{ev}_+^\circ\otimes \mathrm{ev}^\circ_-)$
yields a surjective homomorphism $ Y^+(2)\otimes Y^-(2)\onto \mathfrak{U}\mfgl_2$. Writing this map explicitly and using Proposition \ref{P:Y-Ols->U}, we have $s^\circ_{ij}(u)\mapsto 0$ for $i\neq j$, and
\begin{equation*}
\begin{aligned}
 &s^\circ_{-1,-1}(u)\mapsto 1-\frac{F_{11}+F_{22}}{u+1/2},\\
 &s^\circ_{11}(u)\mapsto 1+\frac{F_{11}+F_{22}}{u+1/2},
 \end{aligned}
 \quad
\begin{aligned}
 &s^\bullet_{-1,-1}(u)\mapsto 1+\frac{F_{11}-F_{22}}{u-1/2}, \quad && s^\bullet_{-1,1}(u)\mapsto -\frac{2F_{12}}{u-1/2},\\
 &s^\bullet_{11}(u)\mapsto 1+\frac{F_{22}-F_{11}}{u-1/2}, \quad && s^\bullet_{1,-1}(u)\mapsto -\frac{2F_{21}}{u-1/2}.
\end{aligned}
\end{equation*}
The proof is now completed as follows: Composing $\wh\Phi\circ (\mathrm{ev}_+^\circ\otimes \mathrm{ev}^\circ_-)$ with the embedding $\wt\chi^{(1)}_0$
from \eqref{LRR:DIII.into.AII+AI} gives a homomorphism $\mathrm{ev}_1: X(\mfso_4,\mfgl_2)^{tw}\to \mathfrak{U}\mfgl_2$. It remains to see that it is given by the assignment
\eqref{LRR:ev.DIII} and that it is surjective. However, if it is indeed given by \eqref{LRR:ev.DIII} then it must be surjective, so it remains only to check
the former claim. This can be shown by a direct calculation using the formulas \eqref{LRR:DIII.explicit}. For instance, since 
\begin{equation*}
\wt\chi^{(1)}_0(s_{-2,-2}(u))=-s^\circ_{-1,-1}(u-1/2)s^\bullet_{-1,-1}(u-1/2),
\end{equation*}
we have 
\begin{equation}
 \mathrm{ev}_1(s_{-2,-2}(u))=-\left(1-\frac{F_{11}+F_{22}}{u} \right)\left(1+\frac{F_{11}-F_{22}}{u-1} \right) =-1+\frac{2F_{22}}{u}+\frac{F^2_{11}-F^2_{22}-F_{11}+F_{22}}{u(u-1)}. \label{LRR.DIII.ev.1}
\end{equation}
On the other hand, since $F^{\prime \rho}_{ij}=(g_{ii}+g_{jj})F_{ij}$ and $F_{ij}=-F_{-j,-i}$ in $\mfso_4$, the $(-2,-2)$-th entry of the right hand side of \eqref{LRR:ev.DIII} is given by
\begin{equation*}
 -1+\frac{2F_{22}}{u}+\frac{-2F_{22}^2-2F_{12}F_{21}+F^2_{11}+F^2_{22}+F_{12}F_{21}+F_{21}F_{12}}{u(u-1)}=-1+\frac{2F_{22}}{u}+\frac{F^2_{11}-F^2_{22}+F_{22} - F_{11}}{u(u-1)},
\end{equation*}
which coincides with \eqref{LRR.DIII.ev.1}. The images of the remaining generators can be verified similarly. 

\medskip 

If instead $\mfso_4^\rho=\mfso_4$, the argument is similar.  Denote the generators corresponding to the first copy of $\mfsp_2$ in the direct sum $\mfsp_2\oplus\mfsp_2$ by $F^\circ_{ij}$, and those corresponding to the second copy of $\mfsp_2$  by $F^\bullet_{ij}$, where in both cases $i,j\in \{\pm 1\}$.  A basis for $\mfsp_2\oplus \mfsp_2$ is then given by the union of $\{F^\circ_{1,1},F^\circ_{-1,1},F^\circ_{1,-1}\}$ 
 and $\{F^\bullet_{1,1},F^\bullet_{-1,1},F^\bullet_{1,-1}\}$.  
 Let $\Phi$ be the isomorphism $\mfsp_2\oplus \mfsp_2\iso \mfso_4$ given by
 \begin{alignat*}{3}
 & F^\circ_{11}\mapsto F_{11}+F_{22}, \quad &&F^\circ_{-1,1}\mapsto 2F_{-2,1},\quad  &&F^\circ_{1,-1}\mapsto 2F_{1,-2}\\
 &F^\bullet_{11}\mapsto F_{22}-F_{11}, \quad &&F^\bullet_{-1,1}\mapsto -2F_{12},\quad &&F^\bullet_{1,-1}\mapsto -2F_{21}.
\end{alignat*}
$\Phi$ induces an isomorphism $\wh\Phi:\mathfrak{U}\mfsp_2\otimes \mathfrak{U}\mfsp_2\iso \mathfrak{U}\mfso_4$, and so the composition  $\wh\Phi\circ (\mathrm{ev}_-^\circ\otimes \mathrm{ev}^\circ_-)$
is a surjective homomorphism $ Y^-(2)\otimes Y^-(2)\onto \mathfrak{U}\mfso_4$. The composition of this map with the embedding $\wt\chi^{(1)}_1$
from \eqref{LRR:D0.into.AII+AII} gives a homomorphism $\mathrm{ev}_0: X(\mfso_4,\mfso_4)^{tw}\to \mathfrak{U}\mfso_4$. If it is indeed given by 
the assignment \eqref{LRR:ev.D0} then it is surjective, so we need only verify that this is the case. This can be shown directly
by first computing the explicit images $\wt\chi^{(1)}_1(s_{ij}(u))$ (as in \eqref{LRR:DIII.explicit}), and then performing computations similar to those carried out in 
the $\mfso_4^\rho=\mfgl_2$ case.
\end{proof}

Given $\mu_1,\mu_2\in \C$, let $V_\rho(\mu_1,\mu_2)$ denote the irreducible $\mfso_4^\rho$-module with the highest weight $(\mu_1,\mu_2)$ (so, in particular, $F_{ii}\xi=\mu_i\xi$ for $i=1,2$).  The pull-back of $V_\rho(\mu_1,\mu_2)$ via $\mathrm{ev}_0$ and $\mathrm{ev}_1$ is an irreducible $X(\mfso_4,\mfso_4^\rho)^{tw}$-module which we call an evaluation module.

\begin{crl}\label{LRR:Cor.ev_DIIID0}
 Given $\mu_1,\mu_2\in\C$, the evaluation module $V_\rho(\mu_1,\mu_2)$ is isomorphic to the $X(\mfso_4,\mfso_4^\rho)^{tw}$-module $V(\mu_1(u),\mu_2(u))$ where
 \begin{alignat*}{3}
  &\mu_1(u)=1+\frac{2\mu_1}{u}+\frac{\mu_1^2-\mu_2^2+\mu_1-\mu_2}{u(u-1)},&&\quad \mu_2(u)=1+\frac{2\mu_2}{u}+\frac{\mu_2^2-\mu_1^2+\mu_2-\mu_1}{u(u-1)} &&\quad\text{ if }\quad\mfso_4^\rho=\mfgl_2,\\
  &\mu_1(u)=1+\frac{2\mu_1}{u-1}+\frac{\mu_1^2-\mu_2^2}{(u-1)^2},&&\quad \mu_2(u)=1+\frac{2\mu_2}{u-1}+\frac{\mu_2^2-\mu_1^2}{(u-1)^2} &&\quad \text{ if }\quad\mfso_4^\rho=\mfso_4.
  \end{alignat*}
\end{crl}
\begin{proof}
 Consider first the case where $\mfso_4^\rho=\mfgl_2$. We first show that $z$ acts on $V_\rho(\mu_1,\mu_2)$ as the scalar $\mu_1^2+\mu_2^2+\mu_1-\mu_2$. Since $z$ belongs to the center of $\mathfrak{U}\mfgl_2$ and 
 $V_\rho(\mu_1,\mu_2)$ is a highest weight module, $z$ acts by scalar multiplication. Therefore, it suffices to determine how $z$ operates on the highest weight vector $\xi$. We have
 \begin{equation*}
(F^2_{11}+F^2_{22}+F_{12}F_{21}+F_{21}F_{12})\xi=(F^2_{11}+F^2_{22}+F_{11}-F_{22})\xi=(\mu_1^2+\mu_2^2+\mu_1-\mu_2)\xi.
 \end{equation*}
Hence, $z$ acts on $V_\rho(\mu_1,\mu_2)$ as the scalar $\mu_1^2+\mu_2^2+\mu_1-\mu_2$.  Finally, applying \eqref{LRR:ev.DIII} yields the desired formula. 

\medskip 

 If instead $\mfso_4^\rho=\mfso_4$, observe that the Casimir element $\Omega$ operates on $V_\rho(\mu_1,\mu_2)$ as multiplication by the scalar $\mu_1^2+\mu_2^2-2\mu_2$. The corollary now follows from the formula \eqref{LRR:ev.D0}. 
\end{proof}

\subsection{Twisted Yangians for the symmetric pair \texorpdfstring{$(\mfso_3,\mfso_3)$}{}}

The isomorphism $\varphi_0$ given in \eqref{LRR:B0->AII} allows us to use the representation theory of the twisted Yangian
$Y^-(2)$ to study the representation theory of $X(\mfso_3,\mfso_3)^{tw}$.

We will use below the following observation, which can be seen by expanding in powers of $u^{-1}$: if $h(u)\in 1+u^{-1}\C[[u^{-1}]]$ and $a\in \C$, then there exists a unique series $k(u)\in 1+u^{-1}\C[[u^{-1}]]$ such that $h(u)=k(u)k(u+a)$.

Let $\mu(u)=(\mu_0(u),\mu_1(u))$ with $\mu_i(u)\in 1+u^{-1}\C[[u^{-1}]]$. The next lemma will be used in the proof of Proposition \ref{LRR:B0.Class}.

\begin{lemma}\label{lem:mufact}
Suppose the components of $\mu(u)$ satisfy the relations $u\cdot\wt \mu_0(1/2-u)=(1/2-u)\cdot \wt \mu_0(u)$ and $\wt \mu_0(u)\wt \mu_0(-u+1)=\wt \mu_1(u)\wt\mu_1(-u+1)$. Then there exists
$\mu^\circ(u)\in 1+u^{-1}\C[[u^{-1}]]$ such that
\begin{equation}
 \wt \mu_1(u)=2u\,\mu^\circ(2u)\,\mu^\circ(2u-1)\quad \text{ and }\quad \wt \mu_0(u)=2u\,\mu^\circ(2u)\,\mu^\circ(1-2u).
\end{equation}
\end{lemma}
\begin{proof}
 Set $f(u)=\wt \mu_0(u)/\wt \mu_1(u)$. By the observation at the beginning of this subsection, we can find a unique series $\lambda(u)\in 1+u^{-1}\C[[u^{-1}]]$ such that $\mu_1(u)=\lambda(u)\,\lambda(u-1/2)$. Set $\mu^\circ(u)=\lambda(u/2)$. Then
 \begin{equation*}
  \wt \mu_1(u)=2u\, \mu^\circ(2u)\,\mu^\circ(2u-1). 
 \end{equation*}
Since $f(u)=f(-u+1)^{-1}$, there exists a series $\alpha(u)\in 1+u^{-1}\C[[u^{-1}]]$ such that $f(u)=\alpha(1-2u)\,\alpha(2u-1)^{-1}$. Set $g(u)=\alpha(2u-1)\,\mu^\circ(2u-1)^{-1}$. Then: 
\begin{equation*}
 f(u)=\frac{\mu^\circ(1-2u)g(1-u)}{\mu^\circ(2u-1)g(u)}=\frac{2u\mu^\circ(2u)\mu^\circ(1-2u)g(1-u)g(u)^{-1}}{\wt \mu_1(u)},
\end{equation*}
which implies that $\wt \mu_0(u)=2u\,\mu^\circ(2u)\,\mu^\circ(1-2u)\,g(1-u)\,g(u)^{-1}$. Since $u\cdot\wt \mu_0(1/2-u)=(1/2-u)\cdot \wt \mu_0(u)$, we obtain 
\begin{equation*}
 g(1-u)\,g(u)^{-1}=g(u+1/2)\,g(1/2-u)^{-1}\implies k(u+1/2)\,k(u)=g(u+1/2)\,g(u) \text{ where }k(u)=g(1-u).
\end{equation*}
Setting $h(u)=g(u+1/2)g(u)$ and applying the uniqueness of the decomposition in the observation preceding this lemma, we obtain that $g(u)=k(u)=g(1-u)$. Hence
\begin{equation*}
 \wt \mu_0(u)=2u\,\mu^\circ(2u)\,\mu^\circ(1-2u)\,g(1-u)\,g(u)^{-1}=2u\,\mu^\circ(2u)\,\mu^\circ(1-2u). \qedhere
\end{equation*}
\end{proof}

\begin{prop}\label{LRR:B0.Class}
 Let $\mu(u)=(\mu_0(u),\mu_1(u))$ satisfy the conditions of Proposition \ref{HWT:refl.Prop.2} so that the irreducible 
 $X(\mfso_3,\mfso_3)^{tw}$ module $V(\mu_0(u),\mu_1(u))$ exists. Then $V(\mu(u))$ is finite-dimensional if and only if there exists a monic polynomial 
$P(u)$ in $u$ with $P(u)=P(-u+3/2)$, and
\begin{equation}
 \frac{\wt{\mu}_0(u)}{\wt{\mu}_1(u)}=\frac{P(u+1/2)}{P(u)}. \label{LRR:B0.2}
\end{equation}
In this case the polynomial $P(u)$ is unique.
\end{prop}
\begin{proof}
Let us begin by listing the explicit images of the generating series $ s_{ij}(u)$ of $X(\mfso_3,\mfso_3)^{tw}$ under $\varphi_0$. Setting $\wt u=u-1/2$, we have
 \begin{equation}
 \begin{aligned}
  s_{-1,-1}(u) \mapsto& s_{-1,-1}^\circ(2\wt u)s_{-1,-1}^\circ(2u)-\tfrac{1}{4u-1}s_{-1,1}^\circ(2\wt u)s_{1,-1}^\circ(2u) \,,\\
  {s}_{-1,0}(u)  \mapsto& \tfrac{1}{\sqrt{2}}s_{-1,-1}^\circ(2\wt u)s_{-1,1}^\circ(2u)+\tfrac{1}{\sqrt{2}(4u-1)}\left(4u\,s_{-1,1}^\circ(2\wt u)s_{-1,-1}^\circ(2u)-s_{-1,1}^\circ(2\wt u)s_{11}^\circ(2u) \right), \\
  {s}_{-1,1}(u)  \mapsto& -\tfrac{4u}{4u-1}s_{-1,1}^\circ(2\wt u)s_{-1,1}^\circ(2u),\\
   {s}_{0,-1}(u)  \mapsto& \tfrac{1}{\sqrt{2}}s_{1,-1}^\circ(2\wt u)s_{-1,-1}^\circ(2u)+\tfrac{1}{\sqrt{2}(4u-1)}\left(4u\,s_{-1,-1}^\circ(2\wt u)s_{1,-1}^\circ(2u)-s_{11}^\circ(2\wt u)s_{1,-1}^\circ(2u) \right) , \\
  {s}_{00}(u)    \mapsto& \tfrac{1}{8u-2}\left((4u\,s_{-1,-1}^\circ(2\wt u)-s_{11}^\circ(2\wt u))s^\circ_{11}(2u)+(4u\,s^\circ_{11}(2\wt u)-s^\circ_{-1,-1}(2\wt u))s_{-1,-1}^\circ(2u) \right)\\
                        &+\tfrac{1}{2}\left(s_{1,-1}^\circ(2\wt u)s_{-1,1}^\circ(2u)+s_{-1,1}^\circ(2\wt u)s_{1,-1}^\circ(2u) \right) ,\\
  s_{01}(u) \mapsto& -\tfrac{1}{\sqrt{2}}s_{-1,1}^\circ(2\wt u)s_{11}^\circ(2u)-\tfrac{1}{\sqrt{2}(4u-1)}\left(4u\,s_{11}^\circ(2\wt u)s_{-1,1}^\circ(2u)-s_{-1,-1}^\circ(2\wt u)s_{-1,1}^\circ(2u) \right),\\      
  s_{1,-1}(u) \mapsto& -\tfrac{4u}{4u-1}s_{1,-1}^\circ(2\wt u)s_{1,-1}^\circ(2u) \,,\\
   s_{10}(u) \mapsto&-\tfrac{1}{\sqrt{2}}s_{11}^\circ(2\wt u)s_{1,-1}^\circ(2u)-\tfrac{1}{\sqrt{2}(4u-1)}\left(4u\,s_{1,-1}^\circ(2\wt u)s_{11}^\circ(2u)-s_{1,-1}^\circ(2\wt u)s_{-1,-1}^\circ(2u) \right),\\
   s_{11}(u) \mapsto& s_{11}^\circ(2\wt u)s_{11}^\circ(2u)-\tfrac{1}{4u-1}s_{1,-1}^\circ(2\wt u)s_{-1,1}^\circ(2u).                   
 \end{aligned}
 \label{LRR:BO.explicit}
\end{equation} 
A brief explanation for how to obtain these formulas was given in Subsection 4.4 of \cite{GRW}, see in particular the proof of Proposition 4.7 therein. For the sake of the reader, we recall this process and provide a detailed proof for a few of the above relations. 
 
 Recall the vector space $V\cong \C^3$ and its basis $\{v_{-1},v_0,v_1\}$ defined in the paragraph immediately preceding \eqref{LRR:B0->AII}. 
 The matrix $S(u)$ is an element of $\End\,V\otimes X(\mfso_3,\mfso_3)^{tw}[[u^{-1}]]$, while the image $\varphi_0(S(u))=\tfrac{1}{2}R^\circ(-1)S_1^\circ(2u-1)R^\circ(-4u+1)^{t_-}S_2^\circ(2u)$ belongs to $\End\,V\otimes Y^-(2)[[u^{-1}]]$. Therefore, to obtain the image of each generating series $s_{ij}(u)$ under $\varphi_0$ from the assignment given in (5.6), we expand
 $S(u)v_k$ and $\varphi_0(S(u))v_k$, for each $-1\leq k\leq 1$, as linear combinations of $v_{-1}, v_0$ and $v_1$ and then compare coefficients. 
 
 As an example, we consider the case where $k=1$. Since $S(u)v_1=\sum_{i=-1}^1v_i\otimes s_{i1}(u)$, this computation will allow us to compute the images of $s_{-1,1}(u),s_{01}(u)$ and $s_{11}(u)$. Since
 $v_1=-e_1\otimes e_1$, a straightforward computation shows that 
\begin{align}
\varphi_0  (S(u))v_1 = {} & -\tfrac{1}{2}R^\circ(-1)S_1^\circ(2u-1)R^\circ(-4u+1)^{t_-}S_2^\circ(2u)(e_1\otimes e_1)\nonumber\\
 = {} & \tfrac{1}{8u-2}\hspace{-.5em}\sum_{i,k=\pm 1}(e_k\otimes e_1+e_1\otimes e_k)\otimes s_{ki}^\circ(2u-1)s_{i1}^\circ(2u) \nonumber \\
 &  -\tfrac{2u}{4u-1}\hspace{-.5em}\sum_{i,k=\pm 1}(e_k\otimes e_i+e_i\otimes e_k)\otimes s_{k1}^\circ(2u-1)s_{i1}^\circ(2u)\label{so3-ex}. 
\end{align}
From this expression we can easily compute the $Y^-(2)$-coefficients of $v_{-1}=e_{-1}\otimes e_{-1}$ and $v_1=-e_{1}\otimes e_{1}$, which must coincide with the images 
of $s_{-1,1}(u)$ and $s_{11}(u)$, respectively. We obtain
\begin{equation*}
 \varphi_0(s_{-1,1}(u))=-\tfrac{4u}{4u-1}s_{-1,1}^\circ(2u-1)s_{-1,1}^\circ(2u), \quad \varphi_0(s_{11}(u))=s_{11}^\circ(2u-1)s_{11}^\circ(2u)-\tfrac{1}{4u-1}s_{1,-1}^\circ(2u-1)s_{-1,1}^\circ(2u).
\end{equation*}
Similarly, the coefficients of $e_{-1}\otimes e_1$ and $e_1\otimes e_{-1}$ in \eqref{so3-ex} are both equal to 
\begin{equation*}
 \tfrac{1}{8u-2}(s_{-1,-1}^\circ(2u-1)s_{-1,1}^\circ(2u)+s_{-1,1}^\circ(2u-1)s_{11}^\circ(2u))-\tfrac{4u}{8u-2}(s_{-1,1}^\circ(2u-1)s_{11}^\circ(2u)+s_{11}^\circ(2u-1)s_{-1,1}^\circ(2u)).
\end{equation*}
Since $v_0=\frac{1}{\sqrt{2}}(e_{-1}\otimes e_1+e_1\otimes e_{-1})$, the image of $s_{01}(u)$ must coincide with the above expression multiplied by $\sqrt{2}$. After rearranging, this gives
\begin{equation*}
 \varphi_0(s_{01}(u))=-\tfrac{1}{\sqrt{2}}s_{-1,1}^\circ(2u-1)s_{11}^\circ(2u)-\tfrac{1}{\sqrt{2}(4u-1)}(4us_{11}^\circ(2u-1)s_{-1,1}^\circ(2u)-s_{-1-1}^\circ(2u-1)s_{-1,1}^\circ(2u)).
\end{equation*}
The remaining images can all be computed by repeating this procedure with $k=0$ and $k=1$.

\medskip

Let us now return to the core of the proof of Proposition \ref{LRR:B0.Class}. Since the components of $\mu(u)$ satisfy the conditions of Lemma \ref{lem:mufact} (by Proposition \ref{HWT:refl.Prop.2}), we may choose $\mu^\circ(u)\in 1+u^{-1}\C[[u^{-1}]]$ such that
\begin{equation}
 \wt \mu_1(u)=2u\mu^\circ(2u)\mu^\circ(2u-1)\quad \text{ and }\quad \wt \mu_0(u)=2u\mu^\circ(2u)\mu^\circ(1-2u). \label{LRR:B0.4}
\end{equation}
Let $V(\mu^\circ(u))$ denote the irreducible highest weight $Y^-(2)$-module with the highest weight $\mu^\circ(u)$, and let $\xi$ denote its highest weight vector. $V(\mu^\circ(u))$ may
be viewed as a $X(\mfso_3,\mfso_3)^{tw}$-module via the isomorphism $\varphi_0$. 
It is immediate from the explicit formulas \eqref{LRR:BO.explicit} that we must have $s_{ij}(u)\xi=0$ for all $i<j$. Moreover, we have 
\begin{equation*}
 s_{11}(u)\xi=\varphi_0(s_{11}(u))\xi=\mu^\circ(2u)\mu^\circ(2u-1)\xi=\mu_1(u)\xi. 
\end{equation*}
Computing $s_{00}(u)\xi$ requires substantially more effort. First note that we have $s_{-1,1}^\circ(2u-1)s_{1,-1}^\circ(2u)\equiv [s_{-1,1}^\circ(2u-1),s_{1,-1}^\circ(2u)]$ on $\C\xi$, 
and by the explicit form of the defining reflection equation \eqref{[s,s]-Ols} we have 
\begin{align*}
[s_{-1,1}^\circ(2u-1),s_{1,-1}^\circ(2u)] & \equiv \tfrac{4u}{4u-1}\left(s^\circ_{11}(2u)s^\circ_{-1,-1}(2u-1)-s^\circ_{11}(2u-1)s^\circ_{-1,-1}(2u)\right)\\
                                                &+\tfrac{1}{4u-1}\left(s^\circ_{-1,-1}(2u-1)s^\circ_{-1,-1}(2u)-s^\circ_{11}(2u)s^\circ_{11}(2u-1)\right).
\end{align*}
Substituting this into the formula for $\varphi_0(s_{00}(u))$ (obtained from \eqref{LRR:BO.explicit}) and using that $[s^\circ_{ii}(u),s^\circ_{jj}(v)]\equiv 0$
for all $i,j\in \{\pm 1\}$, we obtain 
\begin{equation}
 \varphi_0(s_{00}(u))\equiv \tfrac{1}{4u-1}\left(4u\,s^\circ_{11}(2u)s^\circ_{-1,-1}(2u-1)-s^\circ_{11}(2u)s^\circ_{11}(2u-1)\right). \label{LRR:B0.3}
\end{equation}
By the symmetry relation \eqref{Sym-Ols}, $s^\circ_{-1,-1}(2u-1)=s^\circ_{11}(1-2u)-\tfrac{1}{2-4u}(s^\circ_{11}(1-2u)-s^\circ_{11}(2u-1))$. Substituting this into \eqref{LRR:B0.3}
and appealing to \eqref{LRR:B0.4} we obtain 
\begin{equation*}
 s_{00}(u)\xi=\tfrac{1}{2u-1}\mu^\circ(2u)\left(2u\,\mu^\circ(1-2u)-\mu^\circ(2u-1)\right)\xi=\mu_0(u)\xi.
\end{equation*}
In particular, this shows that, as an $X(\mfso_3,\mfso_3)^{tw}$-module, $V(\mu^\circ(u))$ is isomorphic to $V(\mu(u))$. 
\medskip 

We may now employ the classification results for finite-dimensional irreducible $Y^-(2)$-modules to determine precisely when  $V(\mu(u))$
is finite-dimensional.
As recalled in \eqref{LRR:AII-class}, the $Y^-(2)$-module $V(\mu^\circ(u))$ is finite-dimensional if and only if there exists a monic polynomial $Q(u)$ in $u$ with 
$Q(u)=Q(-u+1)$ and 
\begin{equation}
 \frac{\mu^\circ(-u)}{\mu^\circ(u)}=\frac{Q(u+1)}{Q(u)}. \label{LRR:B0.5}
\end{equation}
Moreover, if it exists (in other words, if $V(\mu^\circ(u))$ is indeed finite-dimensional), the monic polynomial $Q(u)$ is unique.

By \eqref{LRR:B0.4}, the condition \eqref{LRR:B0.5} is equivalent to 
\begin{equation*}
 \frac{\wt \mu_0(u)}{\wt \mu_1(u)}=\frac{\mu^\circ(1-2u)}{\mu^\circ(2u-1)}=\frac{Q(2u)}{Q(2u-1)}=\frac{P(u+1/2)}{P(u)},
\end{equation*}
where $P(u)=2^{-\deg Q(u)}Q(2u-1)$. With this definition of $P(u)$ we have $P(u)=P(-u+3/2)$, 
and the uniqueness of $P(u)$ is guaranteed by the uniqueness of $Q(u)$. Therefore, we have shown that $V(\mu(u))$ is finite-dimensional if and only if 
there exists a (uniquely determined) monic polynomial $P(u)$ such that $P(u)=P(-u+3/2)$ and \eqref{LRR:B0.2} holds. 
\end{proof}

Let $\Omega$ denote the Casimir element $\Omega=F_{11}^2-F_{11}+2F_{10}F_{01}$ of the Lie algebra $\mfso_3$, and set $\Omega(u)=\tfrac{4u+1}{4u}\cdot \Omega$. We have the following analogue of Propositions \ref{LRR:Prop.ev_CI-C0} and \ref{LRR:Prop.ev_DIII-D0}: 
\begin{prop}\label{LRR:Prop.ev_B0}
 Let $F^{\prime\rho}=\sum_{i,j=-1}^1 E_{ij}\otimes F^{\prime \rho}_{ij}\in \End\,\C^3\otimes \mathfrak{U}\mfso_3$. Then the assignment
 \begin{equation}
  S(u)\mapsto I+\frac{u}{u-3/4}\left(\frac{F^{\prime\rho}}{u-1/4}+\frac{(F^{\prime\rho})^2-2F^{\prime\rho}-2\Omega(u)\cdot I}{2(u-1/4)^2}\right) \label{LRR:ev.B0}
 \end{equation}
defines a surjective algebra homomorphism $\mathrm{ev}:X(\mfso_3,\mfso_3)^{tw}\onto \mathfrak{U}\mfso_3$. 
\end{prop}

\begin{proof}
 Let $\Phi:\mfsp_2\iso \mfso_3$ be the isomorphism of Lie algebras given by the assignment
 \begin{equation*}
  F^\circ_{11}\mapsto 2F_{11},\quad F^\circ_{-1,1}\mapsto 2\sqrt{2}F_{-1,0},\quad F^\circ_{1,-1}\mapsto 2\sqrt{2}F_{0,-1}.
 \end{equation*}
Let $\wh\Phi$ denote the corresponding isomorphism of enveloping algebras $\mathfrak{U}\mfsp_2\iso\mathfrak{U}\mfso_3$. Then the composition 
$\wh\Phi\circ\mathrm{ev}^\circ_-\circ \varphi_0$ is a surjective homomorphism $\mathrm{ev}:X(\mfso_3,\mfso_3)^{tw}\onto \mathfrak{U}\mfso_3$. To complete the proof 
of the first statement of the proposition, it remains only to see that 
this map agrees with that given in \eqref{LRR:ev.B0}. This can be checked directly using the explicit images \eqref{LRR:BO.explicit} and the formula $s^\circ_{ij}(u)\mapsto \delta_{ij}+F^\circ_{ij}\left(u- \frac{1}{2}\right)^{-1}$ for
the map $\mathrm{ev}^\circ_-$. For example, from these formulas we obtain 
\begin{align*}
 s_{11}(u)\mapsto &\left(1+\frac{F_{11}}{u-3/4}\right)\left(1+\frac{F_{11}}{u-1/4}\right)-\frac{1}{2u-1/2}\left(\frac{F_{0,-1}}{u-3/4}\right)\left(\frac{F_{-1,0}}{u-1/4}\right)\\
                  &=1+\frac{u}{u-3/4}\left(\frac{2F_{11}}{u-1/4}+\frac{F_{11}^2-F_{11}-\tfrac{1}{4u}(F_{11}^2-F_{11}+2F_{0,-1}F_{-1,0})}{(u-1/4)^2} \right).
\end{align*}
Conversely, since $F^\rho_{ij}=2F_{ij}$ for all $-1\leq i,j\leq 1$, the $(1,1)$-th entry of the matrix on the right hand side of \eqref{LRR:ev.B0} is 
\begin{align*}
 1+&\frac{u}{u-3/4}\left(\frac{2F_{11}}{u-1/4}+\frac{2F_{10}F_{01}+2F_{11}^2-2F_{11}-\tfrac{4u+1}{4u}(F_{11}^2-F_{11}+2F_{10}F_{01})}{(u-1/4)^2} \right)\\
   &=1+\frac{u}{u-3/4}\left(\frac{2F_{11}}{u-1/4}+\frac{F_{11}^2-F_{11}-\tfrac{1}{4u}(F_{11}^2-F_{11}+2F_{10}F_{01})}{(u-1/4)^2} \right).
 \end{align*}
As $F_{10}F_{01}=F_{0,-1}F_{-1,0}$, this shows that $\mathrm{ev}(s_{11}(u))$ is indeed given as claimed in \eqref{LRR:ev.B0}. The images of the other generators can 
be checked similarly, or one can use \eqref{HWT:EmbeddingBracket.2} and the fact that $X(\mfso_3,\mfso_3)^{tw}$ is generated by the coefficients of $s_{11}(u)$ and the elements $F_{ij} \in \mfso_3$.
\end{proof}

The morphism $\mathrm{ev}$ allows us to extend $\mfso_3$-modules to $X(\mfso_3,\mfso_3)^{tw}$-modules. As usual, modules obtained this way are called
evaluation modules. Let $V(\mu)$ denote the irreducible $\mfso_3$-module with the highest 
weight $\mu\in \C$. 

\begin{crl}\label{LRR:Cor.ev_B0}
 Let $\mu\in \C$. Then as an $X(\mfso_3,\mfso_3)$-module $V(\mu)$ is isomorphic to $V(\mu_0(u),\mu_1(u))$ where 
 \begin{equation}
 \mu_0(u)=1+\frac{\mu^2(-u-1/4)-\mu(u-1/4)}{(u-3/4)(u-1/4)^2},\quad \mu_1(u)=1+\frac{\mu^2+(2u-1)\mu}{(u-3/4)(u-1/4)}.\label{LRR:Cor.ev_B0.1}
 \end{equation}
\end{crl}
\begin{proof}
 Observe that the Casimir element $\Omega$ operates on $V(\mu)$ as scalar multiplication by $\mu^2-\mu$. One then obtains the formulas for $\mu_0(u)$ and 
 $\mu_1(u)$ given in \eqref{LRR:Cor.ev_B0.1} directly from the formula \eqref{LRR:ev.B0}, as in the proof of Corollary \ref{LRR:Cor.ev_DIIID0}. 
\end{proof}


\section{Classification of finite dimensional irreducible representations} \label{sec:CT}

With the classification results for low rank twisted Yangians obtained
in Section \ref{sec:SR} and the machinery of Section \ref{sec:HWT} at our disposal, we are now in a position to obtain the main results of this paper: classifications of finite-dimensional irreducible 
modules for twisted Yangians of types CI, DIII and BCD0. It will be convenient to adapt the convention used in Section \ref{sec:SR} and employ the notation 
$X(\mfg_N,\mfg_N^\rho)^{tw}$ and $Y(\mfg_N,\mfg_N^\rho)^{tw}$ for the twisted Yangians $X(\mfg_N,\mcG)^{tw}$ and $Y(\mfg_N,\mcG)^{tw}$, respectively. 

\subsection{Twisted Yangians for symmetric pairs of type CI and DIII}\label{Subsection:CT.CI-DIII}

We begin by focusing on the extended twisted Yangians of types CI and DIII. That is, we have $\mfg_N=\mfso_{2n}$ or $\mfg_N=\mfsp_{2n}$, and $\mfg_N^\rho=\mfgl_n$ in both cases, with
$\mcG=\sum_{i=1}^n(E_{ii}-E_{-i,-i})$. 
The following lemma produces a family of one-dimensional representations of $X(\mfg_N,\mfgl_n)^{tw}$: 
 \begin{lemma}\label{CT:Lemma.tracereps}
  Let $a\in \C$. Then the assignment 
  \begin{equation}
   S(u)\mapsto \mathcal{G}+au^{-1}I \label{CT:Lemma.tracereps.1}
  \end{equation}
  yields a one-dimensional representation $V(a)$ of $X(\mfg_N,\mfgl_n)^{tw}$ with the highest weight $(1+au^{-1},\ldots, 1+au^{-1})$.
 \end{lemma}
\begin{proof}
This follows from Lemma \ref{L:K-1dim} where it was shown that the matrix $\mathcal{G}+au^{-1}I$ satisfies the defining reflection equation \eqref{TX-RE} and symmetry relation
\eqref{TX-Sym} of the extended twisted Yangian $X(\mfg_N,\mfgl_n)^{tw}$. 
\end{proof}



\bigskip 

\noindent We are now prepared to prove our main results concerning the twisted Yangians of types CI and DIII. 

\begin{thrm}\label{CT:Thm.DIII-CI}
 Let $\mu(u)=(\mu_1(u),\ldots,\mu_n(u))$ satisfy \eqref{HWT:nontrivial} so that the irreducible $X(\mfg_{N},\mfgl_n)^{tw}$-module $V(\mu(u))$ exists. Then $V(\mu(u))$ is finite-dimensional if and only if there exists
a scalar $\gamma\in \C$ together with monic polynomials $P_1(u),\ldots,P_n(u)$ in $u$ with $P_i(u)=P_i(-u+n-i+2)$ for each $i>1$, $P_1(\gamma)\neq 0$, and
 \begin{flalign}
 &   &&\frac{\wt{\mu}_{i-1}(u)}{\wt{\mu}_i(u)}=\frac{P_i(u+1)}{P_i(u)}\quad \text{ for }\quad  i=2,\ldots,n, \label{CT:DIII-CI.Drinfeld.1} &&&\\
\text{while} & \hspace{20mm} && &&& \nn\\
 &   &&\frac{\wt{\mu}_1(\ka-u)}{\wt{\mu}_2(u)}=\frac{P_1(u+1)}{P_1(u)}\cdot \frac{\gamma-u}{\gamma+u-\ka} \; \text{ and } \; P_1(u)=P_1(-u+n) \; \text{ if }\; \mfg_{N}=\mfso_{2n} \label{CT:DIII-CI.Drinfeld.2}, &&&\\
 &   &&\frac{\wt{\mu}_1(\ka-u)}{\wt{\mu}_1(u)}=\frac{P_1(u+2)}{P_1(u)}\cdot \frac{\gamma-u}{\gamma+u-\ka} \; \text{ and } \; P_1(u)=P_1(-u+n+3) \; \text{ if }\; \mfg_{N}=\mfsp_{2n}. \label{CT:DIII-CI.Drinfeld.3} &&&
\end{flalign}
 Moreover, when $V(\mu(u))$ is finite-dimensional, the associated tuple $(P_1(u),\ldots,P_n(u),\gamma)$ is unique. 
\end{thrm}

\begin{proof} $(\Longrightarrow)$
We begin by showing that if the $X(\mfg_N,\mfgl_{2n})^{tw}$-module $V=V(\mu_1(u),\ldots,\mu_n(u))$ is finite-dimensional, then there exists $\gamma\in \C$ 
and monic polynomials $P_1(u),\ldots,P_n(u)$ satisfying the conditions in the statement of the theorem. We obtain immediately from Proposition \ref{HWT:refl.Prop.3} that
 there exists monic polynomials $P_2(u),\ldots,P_n(u)$ in $u$ such that $P_i(u)=P_i(-u+n-i+2)$ and  \eqref{CT:DIII-CI.Drinfeld.1} holds for each $i\geq 2$. Therefore, to complete this direction of the proof it remains to see that there is also a scalar $\gamma\in \C$ and a monic polynomial $P_1(u)$ with $P_1(\gamma)\neq 0$, such that \eqref{CT:DIII-CI.Drinfeld.2} and \eqref{CT:DIII-CI.Drinfeld.3} hold. We will prove this by induction on the rank $n$, taking Propositions \ref{LRR:CI-C0.Class} and \ref{LRR:DIII-D0.Class} as the base for the  induction.
 Suppose that the statement holds whenever the rank $n$ of $\mfg_N$ satisfies $n<m$, where $m\in \N$ is some fixed
 integer with $m>2$ if $\mfg_{2m}=\mfso_{2m}$, and $m>1$ if $\mfg_{2m}=\mfsp_{2m}$.  Let $V=V(\mu_1(u),\ldots,\mu_m(u))$ be a nontrivial finite-dimensional irreducible $X(\mfg_{2m},\mfgl_{m})^{tw}$-module, and 
 denote its highest weight vector by $\xi$. Recall the subspace $V_+$ of $V$ containing $\xi$ defined by 
 \begin{equation*}
 V_+=\{w\in V:s_{kn}(u)w=0 \text{ for }k<n\}.
 \end{equation*}
By Proposition \ref{CT:Prop.induction}, $V_+$ inherits the structure of an $X(\mfg_{2m-2},\mathfrak{gl}_{m-1})^{tw}$
 module by letting $s^\prime_{ij}(u)$ operate as $s_{ij}^\circ(u)=s_{ij}(u+1/2)+\frac{\delta_{ij}}{2u}s_{mm}(u+1/2)$ for all $-m+1\leq i,j\leq m-1$. In particular, the highest weight module $V(\mu^\circ(u))$ (see Proposition \ref{CT:Prop.induction} for $\mu^{\circ}(u)$) must also be finite-dimensional, so by the induction hypothesis there is $\alpha\in \C$ and a monic polynomial $Q(u)$ in $u$ such that $Q(\alpha)\neq 0$,
 $Q(u)=Q(-u+m+1/2\mp 3/2)$, and
 \begin{equation}
 \frac{\wt{\mu}^\circ_1(\ka'-u)}{\wt{\mu}^\circ_{3/2\pm1/2}(u)}=\frac{Q(u+3/2\mp 1/2)}{Q(u)}\cdot \frac{\alpha-u}{\alpha+u-\ka'} \; \text{ where } \ka'=\ka-1.\label{CT:DIII.inductionstep}
 \end{equation}

Next observe that for any $1\leq i<m$ we have $\wt\mu_i^\circ(u)=\wt{\mu_i}(u+1/2)$. Using this, making the substitution $u\mapsto u-1/2$, setting $\gamma=\alpha+1/2$ and $P_1(u)=Q(u-1/2)$, we obtain:
\begin{equation}
 \frac{\wt{\mu}_1(\ka-u)}{\wt{\mu}_{3/2\pm 1/2}(u)}=\frac{Q(u+1\mp 1/2)}{Q(u-1/2)}\cdot \frac{\alpha+1/2-u}{\alpha+u-\ka+1/2}=\frac{P_1(u+3/2\mp 1/2)}{P_1(u)}\cdot \frac{\gamma-u}{\gamma+u-\ka}. \label{CT:DIII.unique}
\end{equation}
Moreover, the relation $Q(u)=Q(-u+m+1/2\mp 3/2)$ implies that $P_1(u)=P_1(-u+m)$ if $\mfg_{2m}=\mfso_{2m}$ and $P_1(u) = P_1(-u+m+3)$ if $\mfg_{2m}=\mfsp_{2m}$, and since $Q(\alpha)\neq 0$, $P_1(\gamma)=Q(\gamma-1/2)=Q(\alpha)\neq 0$. Thus, by induction we have established
that if $V=V(\mu(u))$ is a finite-dimensional irreducible $X(\mfg_{N},\mfgl_n)^{tw}$-module, then there exists a scalar $\gamma\in \C$ and monic polynomials $P_1(u),\ldots,P_n(u)$ satisfying the conditions 
in the statement of the theorem. 
 
 \medskip 
 
$(\Longleftarrow)$  Conversely, suppose $\mu(u)=(\mu_1(u),\ldots,\mu_n(u))$ is such that the irreducible $X(\mfg_{N},\mfgl_n)^{tw}$-module $V(\mu(u))$ exists, and in addition there exists $\gamma\in \C$ and monic polynomials 
$P_1(u),\ldots,P_n(u)$ satisfying all the conditions outlined in the statement of the theorem. We wish to show that $V(\mu(u))$ is finite-dimensional. Suppose first that $\mfg_N=\mfso_{2n}$.
Then, since for each $i\in \mathcal{I}_N$ we have $P_i(u)=P_i(-u+n-i+2-\delta_{i1})$, there exists monic polynomials $Q_1(u),\ldots, Q_n(u)$ such that 
\begin{equation*}
 P_i(u)=(-1)^{\deg{Q_i(u)}}Q_i(u)Q_i(-u+n-i+2-\delta_{i1}) 
\end{equation*}
for each $i\in \mathcal{I}_N$. Now, for each $i\in \mathcal{I}_N$, define $\wh{Q}_i(u)=Q_i(u+\ka/2)$ and $\lambda_i(u)\in 1+u^{-1}\C[[u^{-1}]]$ by 
\begin{equation*}
 \lambda_i(u)=u^{-a}\wh Q_2(u)\cdots \wh Q_i(u)\wh Q_{i+1}(u+1)\cdots \wh Q_n(u+1),
\end{equation*}
where $a=\sum_{i=2}^n\deg \wh Q_i$, and let $\lambda_{-1}(u)\in 1+u^{-1}\C[[u^{-1}]]$ be given by the formula
\begin{equation*}
 \lambda_{-1}(u)=\frac{1}{\wh {Q}_1(u)}u^{-a}\wh Q_1(u+1) \wh Q_2(u)\wh Q_{3}(u+1)\cdots \wh Q_n(u+1).
\end{equation*}
Then the $\lambda_i(u)$ satisfy the relations 
\begin{equation}
\frac{\lambda_{-1}(u)}{\lambda_2(u)}=\frac{\wh{Q}_1(u+1)}{\wh{Q}_1(u)} \quad \text{ and }\quad \frac{\lambda_{i-1}(u)}{\lambda_i(u)}=\frac{\wh{Q}_i(u+1)}{\wh{Q}_i(u)} \label{CT:DIII.Lambda}
\end{equation}
for all $i\geq 2$. Next, by Lemma \ref{HWT:Ext} there exists a unique $2n$-tuple $\lambda(u)$ extending $(\lambda_{-1}(u),\lambda_1(u),\ldots,\lambda_n(u))$ with the property that the irreducible $X(\mfso_{2n})$-module
$L(\lambda(u))$ exists. By \eqref{HWT:Ext.All.P_n} and \eqref{HWT:Ext.D.P_1} $L(\lambda(u))$ is also finite-dimensional. Now consider the $X(\mfso_{2n},\mfgl_n)^{tw}$-module $L(\lambda(u))\otimes V(\gamma-\ka)$. Let $\xi\in L(\lambda(u))$ be the highest weight vector, and
let $\eta$ be any nonzero vector in $V(\gamma-\ka)$. Then by Proposition \ref{HWT:Prop.tensors}, 
the module $X(\mfso_{2n},\mfgl_n)^{tw}\cdot(\xi\otimes \eta)$ is a finite-dimensional highest weight $X(\mfso_{2n},\mfgl_n)^{tw}$-module of weight $\mu^\sharp(u)=(\mu^\sharp_1(u),\ldots,\mu^\sharp_n(u))$ whose components are given by
\begin{equation}
\wt \mu^\sharp_i(u)=2u(1+(\gamma-\ka)u^{-1})\lambda_i(u-\ka/2)\lambda_{-i}(-u+\ka/2). \label{CT:DIII.tensorweight}
\end{equation}
By the relations \eqref{CT:DIII.Lambda}, \eqref{HWT:Ext.non-trivial} and \eqref{CT:DIII.tensorweight}, for all $i\geq 2$ we have:
\begin{align*}
 \frac{\wt \mu^\sharp_{i-1}(u)}{\wt \mu^\sharp_{i}(u)}&=\frac{\lambda_{i-1}(u-\ka/2)\lambda_{i}(-u-\ka/2+n-i+1)}{\lambda_i(u-\ka/2)\lambda_{i-1}(-u-\ka/2+n-i+1)}\\
                                                      &=\frac{\wh Q_i(u-\ka/2+1)}{\wh Q_i(u-\ka/2) }\frac{\wh Q_i(-u-\ka/2+n-i+1)}{\wh Q_i(-u-\ka/2+n-i+2)} =\frac{P_i(u+1)}{P_i(u)}.                                                     
\end{align*}
Similarly, relations \eqref{CT:DIII.Lambda}, \eqref{HWT:Ext.non-trivial} and \eqref{CT:DIII.tensorweight} imply that
\begin{align*}
 \frac{\wt \mu^\sharp_{1}(\ka-u)}{\wt \mu^\sharp_{2}(u)}&=\frac{\lambda_{-1}(u-\ka/2)\lambda_2(\ka/2-u)}{\lambda_2(u-\ka/2)\lambda_{-1}(-u+\ka/2)}\cdot \frac{(\ka-u)(1+(\gamma-\ka)(\ka-u)^{-1})}{u(1+(\gamma-\ka)(u)^{-1})}\\                                        
                                                        &=\frac{\wh Q_1(u-\ka/2+1)\wh Q_1(-u+\ka/2)}{\wh Q_1(u-\ka/2)\wh Q_1(-u+\ka/2+1)}\cdot \frac{\gamma-u}{\gamma+u-\ka} = \frac{P_1(u+1)}{P_1(u)}\cdot \frac{\gamma-u}{\gamma+u-\ka}.
\end{align*}
By assumption, the components of the $n$-tuple $\mu(u)$ also satisfy the relations \eqref{CT:DIII-CI.Drinfeld.1} and \eqref{CT:DIII-CI.Drinfeld.2} for the same scalar $\gamma$
and monic polynomials $P_1(u),\ldots,P_n(u)$. This yields that 
\begin{equation}
 \frac{\wt \mu_1(\ka-u)}{\wt \mu_2(u)}=\frac{\wt \mu^\sharp_1(\ka-u)}{\wt \mu^\sharp_2(u)} \quad \text{ and }\quad \frac{\wt \mu_{i-1}(u)}{\wt \mu_i(u)}=\frac{\wt \mu^\sharp_{i-1}(u)}{\wt \mu^\sharp_i(u)} \label{CT:DIII.mu<->sharp}
\end{equation}
for all $i\geq 2$. From the second set of equalities above, we deduce that, setting $g(u)=\mu_n(u)\mu^\sharp_n(u)^{-1}$, we have $\mu_i(u)=g(u)\mu^\sharp_i(u)$ for all $i\in \mathcal{I}_N$. From this and the first equality in \eqref{CT:DIII.mu<->sharp}, we deduce that $g(u) = g(\kappa-u)$. 

Set $h(u)=g(u+\ka/2)$.  Using $g(u)=g(\ka-u)$ we deduce that $h(-u)=g(\ka/2-u)=g(u+\ka/2)=h(u)$, and so $h(u)\in 1+u^{-2}\C[[u^{-2}]]$. Let $V(\mu^\sharp(u))^{\nu_h}$ denote the (irreducible) module obtained by twisting $V(\mu^\sharp(u))$
by the automorphism $\nu_h$ (see \eqref{nu_g}). Since $V(\mu(u))$ and $V(\mu^\sharp(u))^{\nu_h}$ are both irreducible and share the same highest weight, they are isomorphic. 
Moreover, since the module $V(\mu^\sharp(u))$ is an irreducible quotient of the finite-dimensional module $X(\mfso_{2n},\mfgl_n)^{tw}(\xi\otimes \eta)$,
it is itself finite-dimensional. Therefore the module $V(\mu^\sharp(u))^{\nu_h}$, and thus $V(\mu(u))$, is also finite-dimensional.

\medskip 

If instead $\mfg_N=\mfsp_{2n}$, then we need only make minor adjustments in the above proof to account for the subtle differences between \eqref{CT:DIII-CI.Drinfeld.2}
and \eqref{CT:DIII-CI.Drinfeld.3} and the symmetry $P_1(u)=P_1(-u+n+3)$ of $P_1(u)$. Since for each $i\in \mathcal{I}_N$ we have $P_i(u)=P_i(-u+n-i+2+2\delta_{i1})$, there exists monic polynomials $Q_1(u),\ldots, Q_n(u)$ with 
 \begin{equation*} 
  P_i(u)=(-1)^{\deg Q_i}Q_i(u)Q_i(-u+n-i+2+2\delta_{i1})
 \end{equation*}
for all $i\in \mathcal{I}_N$. As before, we define the shifted polynomials $\wh{Q}_1(u),\ldots,\wh{Q}_n(u)$ by the formula $\wh{Q}_i(u)=Q_i(u+\ka/2)$. We then define $\lambda_{-1}(u),\lambda_1(u),\ldots,\lambda_n(u)\in 1+u^{-1}\C[[u^{-1}]]$
by the formulas
\begin{equation*}
  \lambda_{-1}(u)=\frac{1}{\wh {Q}_1(u)}u^{-a}\wh Q_1(u+2)\wh Q_2(u+1)\wh Q_{3}(u+1)\cdots \wh Q_n(u+1) ,
\end{equation*}
and 
\begin{equation*}
 \lambda_i(u)=u^{-a}\wh Q_2(u)\cdots \wh Q_i(u)\wh Q_{i+1}(u+1)\cdots \wh Q_n(u+1)
\end{equation*}
for all $i\in \mathcal{I}_N$, where $a=\sum_{i=2}^n\deg \wh Q_i$.
As in the proof for $\mfg_N=\mfso_{2n}$, Lemma \ref{HWT:Ext} implies that there is a unique $2n$-tuple $\lambda(u)$ extending $(\lambda_{-1}(u),\ldots,\lambda_n(u))$ in such a way that the  irreducible $X(\mfsp_{2n})$-module $L(\lambda(u))$ exists. By
construction $L(\lambda(u))$ is also finite dimensional. Let $\xi\in L(\lambda(u))$ be the highest weight vector, and let $\eta$ be any nonzero vector in the one dimensional $X(\mfsp_{2n},\mfgl_n)^{tw}$-module $V(\gamma-\ka)$. 
The $X(\mfsp_{2n},\mfgl_n)^{tw}$-module $X(\mfsp_{2n},\mfgl_n)^{tw}(\xi\otimes \eta)\subset L(\lambda(u))\otimes V(\gamma-\ka)$ is then a highest weight module with the highest 
weight $\mu^\sharp(u)$ whose components are determined by \eqref{CT:DIII.tensorweight}. The irreducible module $V(\mu^\sharp(u))$ is then isomorphic to a
quotient of the cyclic span $X(\mfsp_{2n},\mfgl_n)^{tw}(\xi\otimes \eta)$, and so in particular is finite-dimensional. Similar computations to those carried
out for the case $\mfg_N=\mfso_{2n}$ then show that the components of $\mu^\sharp(u)$ satisfy the relations \eqref{CT:DIII-CI.Drinfeld.1} and \eqref{CT:DIII-CI.Drinfeld.3}. 
It follows that there exists an even series $h(u)\in 1+u^{-2}\C[[u^{-2}]]$ such that $V(\mu(u))$ is isomorphic to the module $V(\mu^\sharp(u))^{\nu_h}$ obtained
by twisting the module $V(\mu^\sharp(u))$ by the automorphism $\nu_h$. Therefore, since $V(\mu^\sharp(u))$ is finite-dimensional, so is $V(\mu^\sharp(u))^{\nu_h}$, and thus $V(\mu(u))$ is finite-dimensional.

\medskip 

 It remains to show the uniqueness of the tuple $(P_1(u),\ldots,P_n(u),\gamma)$. Suppose that $(Q_1(u),\ldots,Q_n(u),\alpha)$ and $(P_1(u),\ldots,P_n(u),\gamma)$ are two tuples as in the statement of the theorem, both associated to the same
finite-dimensional module $V(\mu(u))$. Define rational functions $f_2(u),\ldots, f_n(u)$ in $u$ by $f_i(u)=\frac{Q_i(u)}{P_i(u)}$ for each $i\geq 2$.  Then, by \eqref{CT:DIII-CI.Drinfeld.1}, each $f_i(u)$ satisfies $f_i(u)=f_i(u+1)$ and
so is periodic with period $1$, which is impossible unless $f_i(u)\in \C$. Since $P_i(u)$ and $Q_i(u)$ are assumed to be monic, this forces $f_i(u)=1$ giving $(Q_2(u),\ldots,Q_n(u))=(P_2(u),\ldots,P_n(u))$. The equality $(Q_1(u),\alpha)=(P_1(u),\gamma)$ can now be proven by induction on the rank $n$ of $\mfg_N$ taking the uniqueness statements of Propositions \ref{LRR:CI-C0.Class} and \ref{LRR:DIII-D0.Class} for $X(\mfsp_2,\mfgl_1)^{tw}$ and $X(\mfso_4,\mfgl_2)^{tw}$, respectively, as the base for induction. Indeed, by the arguments given at the beginning of the proof (see \eqref{CT:DIII.inductionstep} and \eqref{CT:DIII.unique}), if $(P(u),\gamma)$ and $(Q(u),\alpha)$ are both associated to the $X(\mfg_{2n},\mfgl_n)^{tw}$ module 
$V(\mu(u))$, then $(P_1(u+1/2),\gamma-1/2)$ and $(Q_1(u+1/2),\alpha-1/2)$ are both associated to the $X(\mfg_{2n-2},\mfgl_{n-1})^{tw}$ module $V(\mu^\circ(u))$. Therefore,
by induction $(P_1(u),\gamma)=(Q_1(u),\alpha)$. 
\end{proof}


The decomposition $X(\mfg_N,\mfgl_n)^{tw}\cong ZX(\mfg_N,\mfgl_n)^{tw}\otimes Y(\mfg_N,\mfgl_n)^{tw}$ recalled in \eqref{TX=ZTX*TY} allows us to deduce the following
corollary of Theorem \ref{CT:Thm.DIII-CI}:

\begin{crl}\label{CT:Cor.DIII-CI.Yangians}
 The isomorphism classes of finite-dimensional irreducible representations of the twisted Yangians $Y(\mfg_N,\mfgl_n)^{tw}$ are parametrized by families $(P_1(u),\ldots,P_n(u),\gamma)$ where $\gamma\in \C$
 and the $P_i(u)$ are monic polynomials in $u$ such that $P_i(u)=P_i(-u+n-i+2)$ for all $i>1$, while
 \begin{equation}
 P_1(\gamma)\neq 0 \quad \text{ and }\quad P_1(u)=\begin{cases}
                                                   P_1(-u+n) &\quad \text{ if }\quad \mfg_N=\mfso_{2n},\\
                                                   P_1(-u+n+3) &\quad\text{ if }\quad \mfg_N=\mfsp_{2n}.
                                                  \end{cases}
 \end{equation}
\end{crl}
\begin{proof} The proof is similar to that of Corollary 5.19 in \cite{AMR}.  Let $\mathcal{P}(\mfg_N,\mfgl_n)$ be the subset of $\C[u]^{n}\times \C$ consisting of all tuples $(P_1(u),\ldots,P_n(u),\gamma)$ satisfying
 the conditions of the corollary. Suppose first that $V$ is a finite-dimensional irreducible $Y(\mfg_N,\mfgl_n)^{tw}$-module.
 Recall that the center $ZX(\mfg_N,\mfgl_n)^{tw}$ is generated by the coefficients of the even series $w(u)$ (see \eqref{w(u)}). By \eqref{TX=ZTX*TY} we have the decomposition
 $ X(\mfg_N,\mfgl_n)^{tw}\cong ZX(\mfg_N,\mfgl_n)^{tw}\otimes Y(\mfg_N,\mfgl_n)^{tw}$
 and therefore $V$ can be extended to an irreducible representation of $X(\mfg_N,\mfgl_n)^{tw}$ where the central elements act as scalars. In particular, we may let $w(u)$ operate as $1$. By Theorem \ref{HWT:Thm.HWT} and Proposition
 \ref{HWT:Prop.VM}, $V\cong V(\mu(u))$ for some highest 
 weight $\mu(u)$. Since $V(\mu(u))$ is finite-dimensional, Theorem \ref{CT:Thm.DIII-CI} allows us to associate an element of  $\mathcal{P}(\mfg_N,\mfgl_n)$ to 
 $V(\mu(u))$, and thus to $V$. The uniqueness statement of Theorem \ref{CT:Thm.DIII-CI} ensures that this gives a well-defined function $F^\circ$ from the space of isomorphism classes of finite-dimensional irreducible 
 $Y(\mfg_N,\mfgl_n)^{tw}$-modules to $\mathcal{P}(\mfg_N,\mfgl_n)$.

 Conversely, if  $(P_1(u),\ldots,P_n(u),\gamma)\in \mathcal{P}(\mfg_N,\mfgl_n)$, then we can find $\mu(u)=(\mu_1(u),\ldots,\mu_n(u))$
 such that the conditions of Theorem \ref{CT:Thm.DIII-CI} are satisfied. The proof
 of the theorem shows that the corresponding module $V(\mu(u))$ is determined uniquely up to twisting by an automorphism $\nu_h$: see \eqref{CT:DIII.mu<->sharp} and the lines below it. However, the elements of the subalgebra
 $Y(\mfg_N,\mfgl_n)^{tw}$ are all stable under automorphisms of the form $\nu_h$, so the $Y(\mfg_N,\mfgl_n)^{tw}$-module $V$ obtained by restriction is uniquely (up to isomorphism) determined by the $n+1$
 tuple $(P_1(u),\ldots,P_n(u),\gamma)$. Since the elements of the center $ZX(\mfgl_N,\mfgl_n)^{tw}$ must operate by scalar multiplication, the decomposition \eqref{TX=ZTX*TY} implies 
that $V$ is irreducible.  
 Hence, we obtain a well-defined function $F^\bullet$ from $\mathcal{P}(\mfg_N,\mfgl_n)$ to the space of isomorphism classes of finite-dimensional irreducible $Y(\mfg_N,\mfgl_n)^{tw}$-modules, 
 and moreover $F^\bullet$ and $F^\circ$ are mutual inverses.
\end{proof}

Let $\alpha\in \C$, and let $L(i:\alpha)$ denote the irreducible highest weight representation of $Y(\mfg_N)$ with Drinfeld polynomials $(P_1(u),\ldots,P_n(u))$ determined by $P_j(u)=1$ if $j\neq i$, and $P_i(u)=u-\alpha$. These are the so-called \textit{fundamental representations} of $Y(\mfg_N)$. The module $L(i:\alpha)$ can be obtained 
from any $X(\mfg_N)$ module $L(\lambda(u))$ with the Drinfeld polynomials $P_1(u),\ldots,P_n(u)$ by restricting to the subalgebra $Y(\mfg_N)$ (see Definition \ref{def:YgN}). For an explanation
of why this procedure is well-defined, see the proof of Corollary 5.19 in \cite{AMR}. The problem of explicitly constructing fundamental representations for $Y(\mfg_N)$ in the RTT-presentation was treated in Subsection 5.4 of \cite{AMR} (see also Subsection 12.1.D in \cite{CP}). The significance of the fundamental representations can be summarized by the following fact 
which is a restatement of Lemma 5.17 in \cite{AMR}: If $L$ and $L^\circ$ are two finite-dimensional highest weight $X(\mfg_N)$-modules with the highest weight vectors $\xi$ and $\xi^\circ$, respectively, then the cyclic span of $\xi\otimes \xi^\circ$ in $L\otimes L^\circ$ is a highest weight module with the Drinfeld polynomials $(Q_1(u)Q_1^\circ(u),\ldots, Q_n(u)Q_n^\circ(u))$, where $(Q_1(u),\ldots,Q_n(u))$ is the Drinfeld tuple associated to $L$ and $(Q_1^\circ(u),\ldots,Q_n^\circ(u))$ is the Drinfeld tuple associated to $L^\circ$. In particular, this implies that any finite-dimensional irreducible representation of $Y(\mfg_N)$ is isomorphic to a subquotient of a tensor product of fundamental representations (this is formulated precisely in Corollary 12.1.13 of \cite{CP}).

\begin{crl}\label{CT:Cor.DIII-CI.FundamentalReps}
Let $V$ be a finite-dimensional irreducible representation of $Y(\mfg_N,\mfgl_n)^{tw}$. Then there exists $m\in \mathbb{N}$, $i_1,\ldots,i_m\in \{1,\ldots,n\}$ and $a,\alpha_{i_1},\ldots,\alpha_{i_m}\in \C$, 
such that $V$ is isomorphic to a subquotient of the $Y(\mfg_N,\mfgl_n)^{tw}$-module 
\begin{equation*}
 L(i_1:\alpha_{i_1})\otimes \cdots \otimes L(i_m:\alpha_{i_m})\otimes V(a).
\end{equation*}
\end{crl}

\begin{proof}
By Corollary \ref{CT:Cor.DIII-CI.Yangians} we may associate a  tuple $(P_1(u),\ldots,P_n(u),\gamma)$ to $V$ satisfying certain conditions. These conditions imply that there exists 
monic polynomials $Q_1(u),\ldots,Q_n(u)$ satisfying the relations 
\begin{equation*}
P_i(u)=(-1)^{\deg Q_i(u)}Q_i(u-\ka/2)Q_i(-u+n-i+2+\delta_{i1}(\tfrac{1\mp 3}{2})-\ka/2)  \; \text{ for all }\;1\leq i\leq n.
\end{equation*}
Let $L(Q(u))$ denote the finite-dimensional irreducible $Y(\mfg_N)$ module associated to the Drinfeld polynomials $Q_1(u),\ldots,Q_n(u)$. By the facts just recalled before the statement of the Corollary, there exists $m\in \mathbb{N}$, $i_1,\ldots,i_m\in \{1,\ldots,n\}$ and $\alpha_{i_1},\ldots,\alpha_{i_m}\in \C$, 
such that $L(Q(u))$ is isomorphic to a subquotient of the $Y(\mfg_N)$-module $L=L(i_1:\alpha_{i_1})\otimes \cdots \otimes L(i_m:\alpha_{i_m})$.
For each $1\leq j\leq m$ let $\xi_j$ denote the highest weight vector of $L(i_j,\alpha_{i_j})$ and set $\xi=\xi_1\otimes \ldots\otimes\xi_m$. Then the cyclic span of 
$\xi$ in $L$ is a highest weight $Y(\mfg_N)$-module with Drinfeld Polynomials $Q_1(u),\ldots,Q_n(u)$ and its unique irreducible quotient is isomorphic to $L(Q(u))$. Set $a=\gamma-\ka$ and let $\eta\in V(a)$ be any nonzero vector. We may view $V(a)$ as a $Y(\mfg_N,\mfgl_n)^{tw}$-module by restriction. Then, by the same argument as used in the proof of Theorem \ref{CT:Thm.DIII-CI}, we see that the $Y(\mfg_N,\mfgl_n)^{tw}$-module $Y(\mfg_N,\mfgl_n)^{tw}(\xi\otimes \eta)\subset L\otimes V(a)$ is a finite-dimensional highest weight module associated to the tuple $(P_1(u),\ldots,P_n(u),\alpha)$. Since $V$ is the unique (up to isomorphism) finite-dimensional irreducible $Y(\mfg_N,\mfgl_n)^{tw}$-module corresponding to this tuple, it is isomorphic to the irreducible quotient of $Y(\mfg_N,\mfgl_n)^{tw}(\xi\otimes \eta)$. Hence,  we have shown that $V$ is isomorphic to a subquotient of the $Y(\mfg_N,\mfgl_n)^{tw}$-module 
\begin{equation*}
L(i_1:\alpha_{i_1})\otimes \cdots \otimes L(i_m:\alpha_{i_m})\otimes V(a). \qedhere
\end{equation*}
\end{proof}


\subsection{Twisted Yangians for symmetric pairs of type BCD0}\label{Subsection:CT.C0-D0}

Now let $(\mfg_N,\mfg_N^\rho)$ be a symmetric pair of type B0, C0, or D0. 

\medskip 

Let $V$ be a $X(\mfg_N,\mfg_N)^{tw}$-module, and recall the subspace $V_+=\{w\in V:s_{kn}(u)w=0 \text{ for }k<n\}$.
By Proposition \ref{CT:Prop.induction}, if $h(u)\in 1+u^{-1}\C[[u^{-1}]]$ satisfies the relation $h(u)h(\ka'-u)^{-1}=p(u+1/2)^{-1}p'(u)$, then the assignment $s^\prime_{ij}(u)\mapsto h(u)s_{ij}^\circ(u)$, where 
$s_{ij}^\circ(u)=s_{ij}(u+1/2)+\frac{\delta_{ij}}{2u}s_{nn}(u+1/2),$ defines a representation of $X(\mfg_{N-2},\mfg_{N-2})^{tw}$ in the space $V_+$. Let us begin by finding an explicit series $h(u)$ satisfying \eqref{h-p}. Setting $\ka'=\ka-1$, since $N=2\ka\pm 2$, we have 
\begin{equation*}
 p(u)=1\mp\frac{1}{2u-\ka}+\frac{N}{2u-2\ka} = \frac{u(2u-\ka\pm 1)}{(\ka-2u)(\ka-u)} \quad \text{ and }\quad p'(u)=1\mp\frac{1}{2u-\ka'}+\frac{N-2}{2u-2\ka'} = \frac{u(2u-\ka'\pm 1)}{(\ka'-2u)(\ka'-u)}.
\end{equation*}
Thus, 
\begin{equation}
 p(u+1/2)^{-1}p'(u)=\frac{2\ka'-2u+1}{2\ka'-2u}\cdot \frac{2u}{2u+1}. \label{h-p.1}
\end{equation}
Consider the series $h(u)$ defined by  
\begin{equation}
 h(u)=\frac{2u-2\ka'-1}{2u-2\ka'}. \label{CT:C0-D0.h}
\end{equation}
We may view $h(u)$ as an element of $\C[[u^{-1}]]$. Then $h(u)$ has constant term $1$ and, as a consequence of relation \eqref{h-p.1}, $h(u)h(\ka'-u)^{-1}=p(u+1/2)^{-1}p'(u)$.
Therefore the assignment $s^\prime_{ij}(u)\mapsto h(u)s_{ij}^\circ(u)$ defines a representation of $X(\mfg_{N-2},\mfg_{N-2})^{tw}$ in the space $V_+$. 

We are now prepared to prove classification theorems for finite-dimensional irreducible modules of extended twisted Yangians of the type $X(\mfg_N,\mfg_N)^{tw}$. Suppose first that $N=2n$.

\begin{thrm}\label{CT:Thm.C0-D0}
 Let $\mu(u)=(\mu_1(u),\ldots,\mu_n(u))$ satisfy \eqref{HWT:nontrivial} so that the irreducible $X(\mfg_{N},\mfg_{N})^{tw}$-module $V(\mu(u))$ exists. Then $V(\mu(u))$ is finite-dimensional if and only if there exist
monic polynomials $P_i(u)$ in $u$ for $1\le i \le n$ with $P_i(u)=P_i(-u+n-i+2)$ for each $i>1$, and
 \begin{flalign}
 &   &&\frac{\wt{\mu}_{i-1}(u)}{\wt{\mu}_i(u)}=\frac{P_i(u+1)}{P_i(u)}\quad \text{ for }\quad  i=2,\ldots,n, \label{CT:C0-D0.Drinfeld.1} &&&\\
\text{while} & \hspace{20mm} && &&& \nn\\
 &   &&\frac{\wt{\mu}_1(\ka-u)}{\wt{\mu}_2(u)}=\frac{\ka-u}{u}\cdot \frac{P_1(u+1)}{P_1(u)} \; \text{ and } \; P_1(u)=P_1(-u+n) \; \text{ if }\; \mfg_{N}=\mfso_{2n} \label{CT:C0-D0.Drinfeld.2}, &&&\\
 &   &&\frac{\wt{\mu}_1(\ka-u)}{\wt{\mu}_1(u)}=\frac{\ka-u}{u}\cdot \frac{P_1(u+2)}{P_1(u)} \; \text{ and } \; P_1(u)=P_1(-u+n+3) \; \text{ if }\; \mfg_{N}=\mfsp_{2n}. \label{CT:C0-D0.Drinfeld.3} &&&
\end{flalign}
Moreover, when $V(\mu(u))$ is finite-dimensional, the associated tuple $(P_1(u),\ldots,P_n(u))$ is unique. 
\end{thrm}

\begin{proof} $(\Longrightarrow)$  Suppose first that the $X(\mfg_N,\mfg_N)^{tw}$-module $V=V(\mu_1(u),\ldots,\mu_n(u))$ is finite-dimensional. By Proposition \ref{HWT:refl.Prop.3}
 there exists monic polynomials $P_2(u),\ldots,P_n(u)$ in $u$ such that $P_i(u)=P_i(-u+n-i+2)$ and  \eqref{CT:C0-D0.Drinfeld.1} holds. Thus it suffices to prove the statement that there there exists a monic polynomial $P_1(u)$ in $u$ such that \eqref{CT:C0-D0.Drinfeld.2} holds if $\mfg_N=\mfso_N$ and \eqref{CT:C0-D0.Drinfeld.3} holds if $\mfg_N=\mfsp_N$. We 
 do this by induction on $n$, taking Propositions \ref{LRR:CI-C0.Class} and \ref{LRR:DIII-D0.Class} as the induction base. Suppose inductively that the statement holds for $n<m$, where
 $m\in \N$ is fixed (with $m>2$ if $\mfg_{2m}=\mfso_{2m}$, $m>1$ if $\mfg_{2m}=\mfsp_{2m}$).  Let $(\mu_1(u),\ldots,\mu_m(u))$ be such that the irreducible $X(\mfg_{2m},\mfg_{2m})^{tw}$-module $V=V(\mu_1(u),\ldots,\mu_m(u))$ exists and is finite-dimensional.  
 Denote its highest weight vector by $\xi$. Let $h(u)$ be the series defined in \eqref{CT:C0-D0.h}, corresponding to $N=2m$. By Proposition \ref{CT:Prop.induction}, the space $V_+$ is an $X(\mfg_{2m-2},\mfg_{2m-2})^{tw}$-module and the cyclic span $X(\mfg_{2m-2},\mfg_{2m-2})^{tw}\xi$  has the structure of a highest weight module with the highest
 weight $h(u)\mu^\circ(u)=(h(u)\mu_1^\circ(u),\ldots,h(u)\mu_{m-1}^\circ(u))$. Moreover, $V_+$ is finite-dimensional and also non-trivial as $\xi \in V_+$. It follows that the irreducible $X(\mfg_{2m-2},\mfg_{2m-2})^{tw}$-module $V(h(u)\mu^\circ(u))$ exists and is finite-dimensional. Thus, by induction there is a monic
 polynomial $Q(u)$ with
 \begin{equation*}
 \frac{h(\ka'-u)\wt{\mu}^\circ_1(\ka'-u)}{h(u)\wt{\mu}^\circ_{3/2\pm1/2}(u)}=\frac{\ka'-u}{u}\cdot \frac{Q(u+3/2\mp 1/2)}{Q(u)} \; \text{ and } \; Q(u) = Q(-u+m + 1/2 \mp 3/2).
 \end{equation*}
By definition of the series $h(u)$ (see \eqref{CT:C0-D0.h}), the first equality on the previous line is equivalent to: 
 \begin{equation*}
 \frac{\wt{\mu}^\circ_1(\ka'-u)}{\wt{\mu}^\circ_{3/2\pm1/2}(u)}=\frac{h(u)}{h(\ka'-u)}\cdot\frac{\ka'-u}{u}\cdot \frac{Q(u+3/2\mp 1/2)}{Q(u)}=\frac{\ka-u-1/2}{u+1/2}\cdot \frac{Q(u+3/2\mp 1/2)}{Q(u)}.
 \end{equation*}
 As in the proof of Theorem \ref{CT:Thm.DIII-CI}, we can observe that $\wt\mu_i^\circ(u)=\wt\mu_i(u+1/2)$ for any $i<m$. Therefore, substituting $u\mapsto u-1/2$ and setting $P_1(u)=Q(u-1/2)$ we obtain 
 \begin{equation*}
 \frac{\wt{\mu}_1(\ka-u)}{\wt{\mu}_{3/2\pm1/2}(u)}= \frac{\ka-u}{u}\cdot \frac{P_1(u+3/2\mp 1/2)}{P_1(u)}.
 \end{equation*}
Moreover, since $Q(u)=Q(-u+m + 1/2 \mp 3/2)$, we have $P_1(u)=P_1(-u+m + 3/2 \mp 3/2)$. Therefore by induction we have shown that there exists 
 $P_1(u)$ satisfying \eqref{CT:C0-D0.Drinfeld.2} in the orthogonal case and \eqref{CT:C0-D0.Drinfeld.3} in the symplectic case. 
 
 \medskip 
 
 $(\Longleftarrow)$ Now suppose that $(\mu_1(u),\ldots,\mu_n(u))$ satisfies \eqref{HWT:nontrivial} and that there exists $P_1(u),\ldots,P_n(u)$ as in the statement of the theorem. We wish to  show that $V=V(\mu_1(u),\ldots,\mu_n(u))$ is finite-dimensional. This portion of the proof will be proven analogously to the corresponding direction in the proof of Theorem
 \ref{CT:Thm.DIII-CI}. 
 
 Suppose first that $\mfg_N=\mfso_{2n}$. Since $P_i(u)=P_i(-u+n-i+2-\delta_{i1})$ for all 
$1\leq i\leq n$, we can find monic polynomials $Q_i(u)$ such that
\begin{equation*}
 P_i(u)=(-1)^{\deg{Q_i(u)}}Q_i(u)Q_i(-u+n-i+2-\delta_{i1}) 
\end{equation*}
for each $i$. Set $\wh{Q}_i(u)=Q_i(u+\ka/2)$ for all $i$, and choose $\lambda_{-1}(u),\lambda_1(u),\ldots,\lambda_n(u)\in 1+u^{-1}\C[[u^{-1}]]$ such that 
\begin{equation*}
\frac{\lambda_{i-1}(u)}{\lambda_i(u)}=\frac{\wh{Q}_i(u+1)}{\wh{Q}_i(u)} \; \text{ for all $i\geq 2$ and } \; \lambda_{-1}(u)=\frac{\wh{Q}_1(u+1)}{\wh{Q}_1(u)}\lambda_2(u). 
\end{equation*}
By Lemma \ref{HWT:Ext} there is a unique way of extending
$(\lambda_{-1}(u),\ldots,\lambda_n(u))$ to a $2n$-tuple $\lambda(u)$ such that the $X(\mfso_{2n})$-module $L(\lambda(u))$ exists. Moreover, by 
\eqref{HWT:Ext.All.P_n} and \eqref{HWT:Ext.D.P_1} $L(\lambda(u))$ must be finite-dimensional. Let $\xi\in L(\lambda(u))$ be the highest weight vector. By Corollary
\ref{HWT:Cor.restrictions} and \eqref{HWT:Ext.non-trivial}, $X(\mfso_{2n},\mfso_{2n})^{tw}\xi$ is a highest weight $X(\mfso_{2n},\mfso_{2n})^{tw}$-module with highest weight
$\mu^\sharp(u)=(\mu_1^\sharp(u),\ldots,\mu_n^\sharp(u))$ whose components satisfy 
\begin{equation*}
 \frac{\wt{\mu}_i^\sharp(u)}{\wt{\mu}_{i+1}^\sharp(u)}=\frac{P_{i+1}(u+1)}{P_{i+1}(u)} \; \text{ for } \; i=1,\ldots,n-1, \ \text{ while } \;  \frac{\wt{\mu}_1^\sharp(\ka-u)}{\wt{\mu}_{2}^\sharp(u)} = \frac{\ka-u}{u} \cdot \frac{P_1(u+1)}{P_1(u)}.
\end{equation*}
(See the proof of Theorem \ref{CT:Thm.DIII-CI} for more details.)
Since $\mu(u)$ and $\mu^\sharp(u)$ both satisfy conditions \eqref{CT:C0-D0.Drinfeld.1} and \eqref{CT:C0-D0.Drinfeld.2}, it follows that there exists an even series $h(u)\in1+u^{-2}\C[[u^{-2}]]$ with the property
that $V(\mu(u))$ is isomorphic to the module obtained by twisting $V(\mu^\sharp(u))$ by the automorphism $\nu_h$ (\textit{cf.}~proof of Theorem~\ref{CT:Thm.DIII-CI}). As  $V(\mu^\sharp(u))$ is isomorphic to the irreducible 
quotient of $X(\mfso_{2n},\mfso_{2n})\xi$, it is finite dimensional, and therefore so is $V(\mu(u))$. 

\medskip 

If instead $\mfg_N=\mfsp_{2n}$, we need only make some minor adjustments.  Since for each $1\leq i\leq n$ we have $P_i(u)=P_i(-u+n-i+2+2\delta_{i1})$, there exists monic polynomials $Q_1(u),\ldots, Q_n(u)$ with 
 \begin{equation*} 
  P_i(u)=(-1)^{\deg Q_i}Q_i(u)Q_i(-u+n-i+2+2\delta_{i1})
 \end{equation*}
for all $1\leq i\leq n$. As before we define the polynomials $\wh{Q}_1(u),\ldots,\wh{Q}_n(u)$ by the formulas $\wh{Q}_i(u)=Q_i(u+\ka/2)$. Choose $\lambda_{-1}(u),\lambda_1(u),\ldots,\lambda_n(u)\in 1+u^{-1}\C[[u^{-1}]]$
satisfying
\begin{equation*}
 \frac{\lambda_{i-1}(u)}{\lambda_i(u)}=\frac{\wh{Q}_i(u+1)}{\wh{Q}_i(u)} \;\text{ for all $2\leq i\leq n$ and} \; \lambda_{-1}(u)=\frac{\wh{Q}_1(u+2)}{\wh{Q}_1(u)}\lambda_1(u).
\end{equation*}
Extend $(\lambda_{-1}(u),\ldots,\lambda_n(u))$ to the unique $2n$-tuple $\lambda(u)$ so that the $X(\mfsp_{2n})$-module $L(\lambda(u))$ exists. The proof that $V(\mu(u))$ is finite dimensional in the $\mfsp_{2n}$-case can be completed as in the $\mfso_{2n}$-case above.

\medskip 

To complete the proof of the Theorem, note that the uniqueness of $(P_1(u),\ldots,P_n(u))$ can be proven the exact same way as the uniqueness of $(P_1(u),\ldots,P_n(u),\gamma)$
in Theorem \ref{CT:Thm.DIII-CI}.
\end{proof}

We now consider the case when $N=2n+1$, that is when $\mfg_N=\mfso_{2n+1}$.  

\begin{thrm}\label{CT:Thm.B0}
 Let $\mu(u)=(\mu_0(u),\ldots,\mu_n(u))$ satisfy \eqref{HWT:nontrivial} and $u\cdot\wt\mu_0(\ka-u)=(\ka-u)\cdot\wt\mu_0(u)$ so that the irreducible $X(\mfg_{N},\mfg_{N})^{tw}$ module $V(\mu(u))$ exists. Then $V(\mu(u))$ is finite-dimensional if and only if there exists
monic polynomials $P_1(u),\ldots,P_n(u)$ in $u$ with $P_i(u)=P_i(-u+n-i+2)$ for each $i>1$, $P_1(u)=P_1(-u+n+1/2)$ and
 \begin{equation}
\frac{\wt{\mu}_{i-1}(u)}{\wt{\mu}_i(u)}=\frac{P_i(u+1-\tfrac{\delta_{i1}}{2})}{P_i(u)}\quad \text{ for }\quad  i=1,\ldots,n. \label{CT:B0.Drinfeld.1}  
\end{equation}
 Moreover, when $V(\mu(u))$ is finite-dimensional, the associated tuple $(P_1(u),\ldots,P_n(u))$ is unique. 
\end{thrm}

\begin{proof}
$(\Longrightarrow)$ Suppose first that the $X(\mfg_N,\mfg_N)^{tw}$-module $V=V(\mu_0(u),\ldots,\mu_n(u))$ is finite-dimensional. By Proposition \ref{HWT:refl.Prop.3}
 there exists monic polynomials $P_2(u),\ldots,P_n(u)$ in $u$ such that $P_i(u)=P_i(-u+n-i+2)$ and  \eqref{CT:B0.Drinfeld.1} holds for $2\leq i\leq n$. As in the proof of Theorem \ref{CT:Thm.C0-D0}, it suffices to prove that there exists  a monic polynomial $P_1(u)$ with $P_1(u)=P_1(-u+n+1/2)$ and \eqref{CT:B0.Drinfeld.1} holds for $i=1$. We 
 do this by induction on $n$, taking Proposition \ref{LRR:B0.Class} as the induction base. Suppose inductively that the statement holds for $n<m$ for some 
 fixed $m>2$.  Let $V=V(\mu_0(u),\ldots,\mu_m(u))$ be a nontrivial finite-dimensional $X(\mfg_{2m+1},\mfg_{2m+1})^{tw}$-module, and 
 denote its highest weight vector by $\xi$. The same argument as in the proof of Theorem \ref{CT:Thm.C0-D0} shows that the irreducible module $V(h(u)\mu^\circ(u))$ is finite-dimensional. Therefore, by induction there exists a monic
 polynomial $Q(u)$ such that $Q(u)=Q(-u+m-1/2)$ and 
 \begin{equation*}
 \frac{h(u)\wt{\mu}^\circ_0(u)}{h(u)\wt{\mu}^\circ_{1}(u)}=\frac{\wt{\mu}^\circ_0(u)}{\wt{\mu}^\circ_{1}(u)}=\frac{Q(u+1/2)}{Q(u)}.
 \end{equation*}
Since $\wt\mu_i^\circ(u)=\wt\mu_i(u+1/2)$, substituting $u\mapsto u-1/2$ and setting $P_1(u)=Q(u-1/2)$ we obtain 
 \begin{equation*}
 \frac{\wt{\mu}_0(u)}{\wt{\mu}_{1}(u)}= \frac{P_1(u+1/2)}{P_1(u)},
 \end{equation*}
 which is exactly \eqref{CT:B0.Drinfeld.1} with $i=1$. Moreover, since $Q(u)=Q(-u+m-1/2)$, 
 $P_1(u)=P_1(-u+m+1/2)$. Therefore by induction we have shown that there exists 
 $P_1(u)$ satisfying the conditions of the theorem. 
 
 \medskip 
 
$(\Longleftarrow)$ Now suppose that $(\mu_0(u),\ldots,\mu_n(u))$ satisfies $u\cdot \wt\mu_0(\ka-u)=(\ka-u)\cdot \wt\mu_0(u)$, condition \eqref{HWT:nontrivial}, and that there exists $P_1(u),\ldots,P_n(u)$ as in the statement of the theorem. We wish to
 show that $V=V(\mu_0(u),\ldots,\mu_n(u))$ is finite-dimensional. This portion of the proof will be proven analogously to the corresponding direction in the proofs of Theorems
 \ref{CT:Thm.DIII-CI} and \ref{CT:Thm.C0-D0}.  Since $P_i(u)=P_i(-u+n-i+2-\tfrac{\delta_{i1}}{2})$ for all 
$1\leq i\leq n$, we can find monic polynomials $Q_i(u)$ such that
\begin{equation*}
 P_i(u)=(-1)^{\deg{Q_i(u)}}Q_i(u)Q_i(-u+n-i+2-\tfrac{\delta_{i1}}{2}) 
\end{equation*}
for each $i$. Set $\wh{Q}_i(u)=Q_i(u+\ka/2)$ for all $i$, and choose $\lambda_0(u),\lambda_1(u),\ldots,\lambda_n(u)\in 1+u^{-1}\C[[u^{-1}]]$ such that 
\begin{equation*}
\frac{\lambda_{i-1}(u)}{\lambda_i(u)}=\frac{\wh{Q}_i(u+1)}{\wh{Q}_i(u)} \; \text{ for all } \; 2 \le i \le n \; \text{ and }  \lambda_{0}(u)=\frac{\wh{Q}_1(u+1/2)}{\wh{Q}_1(u)}\lambda_1(u). 
\end{equation*}
By Lemma \ref{HWT:Ext} there is a unique $N$-tuple $\lambda(u)$ extending
$(\lambda_{0}(u),\ldots,\lambda_n(u))$ with the property that the $X(\mfso_{2n+1})$-module $L(\lambda(u))$ exists. Moreover, by 
\eqref{HWT:Ext.All.P_n} and \eqref{HWT:Ext.B.P_1}, $L(\lambda(u))$ must be finite-dimensional. Let $\xi\in L(\lambda(u))$ be the highest weight vector. By Corollary
\ref{HWT:Cor.restrictions} and \eqref{HWT:Ext.non-trivial}, $X(\mfg_N,\mfg_N)^{tw}\xi$ is a highest weight $X(\mfg_N,\mfg_N)^{tw}$-module with highest weight
$\mu^\sharp(u)=(\mu_0^\sharp(u),\ldots,\mu_n^\sharp(u))$ whose components satisfy
\begin{equation*}
 \frac{\wt{\mu}_i^\sharp(u)}{\wt{\mu}_{i+1}^\sharp(u)}=\frac{P_{i+1}(u+1)}{P_{i+1}(u)} \; \text{ for } \; i=1,\ldots,n-1, \; \text{ while } \frac{\wt{\mu}_0^\sharp(u)}{\wt{\mu}_{1}^\sharp(u)} = \frac{P_1(u+1/2)}{P_1(u)}.
\end{equation*}
Since $\mu(u)$ and $\mu^\sharp(u)$ both satisfy conditions \eqref{CT:B0.Drinfeld.1}, it follows that there exists an even series $h(u)\in1+u^{-2}\C[[u^{-2}]]$ with the property
that $V(\mu(u))$ is isomorphic to the module obtained by twisting $V(\mu^\sharp(u))$ by the automorphism $\nu_h$. 
To see this, note first that since $u\cdot \wt \mu_0(\ka-u)=(\ka-u)\cdot \wt \mu_0(u)$, the $i=1$ statement of \eqref{CT:B0.Drinfeld.1} is equivalent to 
\begin{equation*}
 \frac{\wt{\mu}_0(\ka-u)}{\wt{\mu}_{1}(u)}=\frac{\ka-u}{u}\cdot \frac{P_1(u+1/2)}{P_1(u)}.
\end{equation*}
One then repeats the same argument as given in detail in Theorem \ref{CT:Thm.DIII-CI}. As  $V(\mu^\sharp(u))$ is isomorphic to the irreducible 
quotient of $X(\mfg_N,\mfg_N)^{tw}\xi$, it is finite dimensional, and therefore so is $V(\mu(u))$. 

For a proof of the uniqueness of the tuple $(P_1(u),\ldots,P_n(u))$, see the proof of Theorem \ref{CT:Thm.DIII-CI}. \end{proof}


The decomposition $X(\mfg_N,\mfg_N)^{tw}\cong ZX(\mfg_N,\mfg_N)^{tw}\otimes Y(\mfg_N,\mfg_N)^{tw}$ recalled in \eqref{TX=ZTX*TY} allows us to deduce the following
corollary of Theorems \ref{CT:Thm.C0-D0} and \ref{CT:Thm.B0}:

\begin{crl}\label{CT:Cor.C0-D0.Yangians}
 The isomorphism classes of finite-dimensional irreducible representations of the twisted Yangians $Y(\mfg_N,\mfg_N)^{tw}$ are parametrized by families $(P_1(u),\ldots,P_n(u))$ where 
 the $P_i(u)$ are monic polynomials in $u$ such that $P_i(u)=P_i(-u+n-i+2)$ for all $i>1$, while
 \begin{equation}
                                           P_1(u)=\begin{cases}
                                                   P_1(-u+n+1/2)&\quad \text{ if }\quad \mfg_N=\mfso_{2n+1},\\
                                                   P_1(-u+n) &\quad \text{ if }\quad \mfg_N=\mfso_{2n},\\
                                                   P_1(-u+n+3) &\quad\text{ if }\quad \mfg_N=\mfsp_{2n}.
                                                  \end{cases}
 \end{equation}
\end{crl}
\begin{proof}
This is proved identically to Corollary \ref{CT:Cor.DIII-CI.Yangians}. 
\end{proof}

Given $1\leq i\leq n$ and $\alpha\in \C$, let $L(i:\alpha)$ be the fundamental representation of the Yangian $Y(\mfg_N)$ as defined above Corollary \ref{CT:Cor.DIII-CI.FundamentalReps}. Our last corollary can be proved in the same way as Corollary \ref{CT:Cor.DIII-CI.FundamentalReps}.
\begin{crl}\label{CT:Cor.C0-D0.FundamentalReps}
Let $V$ be a finite-dimensional irreducible representation of $Y(\mfg_N,\mfg_N)^{tw}$. Then there exists $m\in \mathbb{N}$, $i_1,\ldots,i_m\in \{1,\ldots,n\}$ and $\alpha_{i_1},\ldots,\alpha_{i_m}\in \C$, 
such that $V$ is isomorphic to a subquotient of the $Y(\mfg_N,\mfg_N)^{tw}$-module 
\begin{equation*}
  L(i_1:\alpha_{i_1})\otimes \cdots \otimes L(i_m:\alpha_{i_m}).
\end{equation*}
\end{crl}


\end{document}